\documentclass{article}

\RequirePackage[OT1]{fontenc}
\RequirePackage{amsthm,amsmath,bm}
\RequirePackage[colorlinks,citecolor=blue,urlcolor=blue]{hyperref}

\usepackage{amsmath,amssymb, a4wide}
\usepackage{amssymb,latexsym}
\usepackage{dsfont, color}
\usepackage{verbatim}
\usepackage{cancel}
\usepackage{authblk}
\usepackage{appendix}
\usepackage{graphicx}

\numberwithin{equation}{section}
\theoremstyle{plain}
\newtheorem{theorem}{Theorem}
\newtheorem{corollary}{Corollary}
\newtheorem{lemma}{Lemma}

\newtheorem{proposition}{Proposition}
\newtheorem{remark}{Remark}[section]
\newtheorem{example}{Example}

\newcommand{\R}{\mathbb{R}}

\newcommand{\dd}{\mathrm{d}}
\newcommand{\vc}{\mathrm{vec}}
\newcommand{\vch}{\mathrm{vech}}
\newcommand{\tr}{\mathrm{tr}}

\newcommand{\bbeta}{\bm\beta}
\newcommand{\beps}{\bm\epsilon}
\newcommand{\btheta}{\bm\theta}
\newcommand{\bxi}{\bm\xi}
\newcommand{\balpha}{\bm\alpha}

\newcommand{\bgamma}{\bm\gamma}
\newcommand{\bmu}{\bm\mu}
\newcommand{\bSigma}{\mathbf{\Sigma}}

\newcommand{\bGamma}{\bm\Gamma}

\newcommand{\bTheta}{\bm\Theta}
\newcommand{\bzeta}{\bm\zeta}

\newcommand{\by}{\mathbf{y}}
\newcommand{\bx}{\mathbf{x}}
\newcommand{\bs}{\mathbf{s}}
\newcommand{\bu}{\mathbf{u}}
\newcommand{\bv}{\mathbf{v}}
\newcommand{\bt}{\mathbf{t}}
\newcommand{\bz}{\mathbf{z}}

\newcommand{\bb}{\mathbf{b}}
\newcommand{\be}{\mathbf{e}}

\newcommand{\bX}{\mathbf{X}}
\newcommand{\bV}{\mathbf{V}}
\newcommand{\bG}{\mathbf{G}}
\newcommand{\bZ}{\mathbf{Z}}
\newcommand{\bI}{\mathbf{I}}
\newcommand{\bA}{\mathbf{A}}

\newcommand{\bC}{\mathbf{C}}
\newcommand{\bB}{\mathbf{B}}
\newcommand{\bD}{\mathbf{D}}
\newcommand{\bL}{\mathbf{L}}
\newcommand{\bK}{\mathbf{K}}

\newcommand{\bU}{\mathbf{U}}
\newcommand{\bR}{\mathbf{R}}

\newcommand{\bE}{\mathbf{E}}
\newcommand{\bQ}{\mathbf{Q}}
\newcommand{\bH}{\mathbf{H}}

\newcommand{\E}{\mathbb{E}}

\begin{document}

\title{Multivariate MM-estimators with auxiliary Scale for Linear Models with Structured Covariance Matrices}

\author[1]{Hendrik Paul Lopuha\"a}
\affil[1]{\emph{Delft University of Technology}}
\date{\today}

\maketitle

\begin{abstract}
We provide a unified approach to MM-estimation with auxiliary scale for balanced linear models with structured covariance matrices.
This approach leads to estimators that are highly robust against outliers and highly efficient for normal data.
These properties not only hold for estimators of the regression parameter, 
but also for estimators of scale invariant transformations of the variance parameters.
Of main interest are MM-estimators for linear mixed effects models, but our approach also includes
MM-estimators in several other standard multivariate models.
We provide sufficient conditions for the existence of MM-functionals and MM-estimators,
establish asymptotic properties such as consistency and asymptotic normality,
and derive their robustness properties in terms of breakdown point and influence function.
All the results are obtained for general identifiable covariance structures and are established under mild conditions
on the distribution of the observations, which goes far beyond models with elliptically contoured densities.
\end{abstract}

\section{Introduction}
\label{sec:introduction}
Linear models with a structured covariance are a generalization of traditional linear models where the residuals
are assumed to follow a specific covariance structure rather than being independent and identically distributed. 
This approach is useful when the residuals are correlated or exhibit some form of structure that can't be captured by simple uncorrelated noise.
These models are often used in cases like repeated measures, longitudinal data, and hierarchical structures, 
where the observations within a group or over time might be more similar to each other than to observations from other groups or time points.
An example are linear mixed effects models, which explicitly account for both fixed effects (predictors whose effects are the same across all units) and random effects (predictors whose effects vary across groups or subjects). 
In these models, the random effects together with the residuals yields a specific covariance structure depending 
on a vector of unknown covariance parameters.

Maximum likelihood estimation has been studied by Hartley and Rao~\cite{hartley&rao1967},
Rao~\cite{rao1972}, Laird and Ware~\cite{laird&ware1982},
see also Fitzmaurice \emph{et al}~\cite{fitzmaurice-laird-ware2011} and Demidenko~\cite{demidenko2013}.
To be resistant against outliers, robust methods have been investigated for linear mixed effects models
by Pinheiro \emph{et al}~\cite{pinheiro-liu-wu2001},
Copt~\cite{copt2006high},
Copt and Heritier~\cite{copt&heritier2007},
Heritier \emph{et al}~\cite{heritier-cantoni-copt-victoriafeser2009},
Agostinelli and Yohai~\cite{agostinelli2016composite},
and Chervoneva and Vishnyakov~\cite{chervoneva2014},
or for more general linear models with a structured covariance by Lopuha\"a \emph{et al}~\cite{lopuhaa-gares-ruizgazen2023}.
This often concerns S-estimators originally proposed by Rousseeuw and Yohai~\cite{rousseeuw-yohai1984}
for the multiple linear regression model.
These estimators have been extended to several multivariate statistical models
and can be viewed as smooth versions of the minimum volume ellipsoid estimator,
introduced by Rousseeuw~\cite{rousseeuw1985}, 
that are highly resistant against outliers.
However, one drawback of S-estimators is that they suffer from a low efficiency.

Some extensions have been proposed that inherit the robustness of the S-estimator,
but at the same time improve the efficiency.
Among them are the MM-estimators introduced by Yohai~\cite{yohai1987} for the multiple linear regression model.
The idea is to estimate the scale by means of a robust M-estimator, and then estimate
the regression parameter using a regression M-estimator with a different loss function
that yields better efficiency.
This idea has been extended in different ways to multivariate statistical models.
Lopuha\"a~\cite{lopuhaa1992highly} proposed a version for multivariate location,
Copt and Heritier~\cite{copt&heritier2007} used the same approach to estimate the fixed effects 
in a linear mixed effects model, and a similar method has been studied in Lopuha\"a~\cite{lopuhaa2023}
for more general linear models with a structured covariance.
All these proposals use a robust estimator of the entire scatter matrix in the first step
and only allow efficiency improvement of the location or regression estimator.
Tatsuoka and Tyler~\cite{tatsuoka&tyler2000} introduced a more extensive version of 
multivariate MM-estimators for multivariate location and scatter, 
being members of a broad class of multivariate M-estimators with auxiliary scale.
Their proposal only uses a robust M-estimator of the scale of the scatter matrix in the first step
and allows efficiency improvement of both the location estimator as well as the estimator
of the shape of the scatter matrix.
For this reason, this version of multivariate MM-estimators with auxiliary scale
is particularly useful for applications that require estimation of a covariance matrix.

The theory for these estimators is fairly limited.
Kudraszow and Maronna~\cite{kudraszow-maronna2011} study
MM-estimators with auxiliary scale for multivariate linear regression, 
but no rigorous results are derived for the covariance MM-estimator.
Tatsuoka and Tyler~\cite{tatsuoka&tyler2000} study existence of the corresponding MM-functionals,
but no attention is paid to the limiting behavior of the MM-estimators themselves.
As a basis for a robust PCA method, Salibi\'an-Barrera \emph{et al}~\cite{SalibianBarrera-VanAelst-Willems2006}
use covariance MM-estimators and discusses their limiting behavior, but a rigorous derivation is missing.

In view of this, we provide a unified approach to MM-estimation with auxiliary scale in balanced linear models with structured covariance matrices.
The balanced setup is already quite flexible and includes several specific multivariate statistical models.
Of main interest are MM-estimators for linear mixed effects models, but our approach also includes
MM-estimators in several other standard multivariate models, such as 
multivariate linear regression, and multivariate location and scatter.
We provide sufficient conditions for the existence of MM-functionals and MM-estimators,
establish their asymptotic properties, such as consistency and asymptotic normality,
and derive their robustness properties in terms of breakdown point and influence function.
All results are obtained for a large class of identifiable covariance structures, and are established under very mild conditions
on the distribution of the observations, which goes far beyond models with elliptically contoured densities.

The paper is organized as follows.
In Section~\ref{sec:structured covariance model}, we explain the model in detail and
provide some examples of standard multivariate models that are included in our setup.
In Section~\ref{sec:definitions} we define the MM-estimator and MM-functional
and in Section~\ref{sec:existence} we give conditions under which they exist.
In Section~\ref{sec:continuity} we establish continuity of the MM-functional, which is then used to
obtain consistency of the MM-estimator.
Section~\ref{sec:bdp} deals with the breakdown point.
Section~\ref{sec:equations} provides the preparation for Sections~\ref{sec:IF}
and~\ref{sec:asymp norm} in which we obtain the influence function and establish asymptotic normality.
Our results lead to single scalar indices for the asymptotic efficiency and the gross-error-sensitivity of
standardized components of the MM-estimators of the variance parameters.
In Section~\ref{sec:application} we investigate the interplay between these two scalars at the multivariate
normal and Student distributions.
All proofs are available as supplemental material~\cite{supplement2025}.

\section{Balanced linear models with structured covariances}
\label{sec:structured covariance model}
We consider independent observations $(\by_1,\bX_1),\ldots,(\by_n,\bX_n)$,
for which we assume the following model
\begin{equation}
\label{def:model}
\by_i
=
\bX_i\bbeta+\bu_i,
\quad
i=1,\ldots,n,
\end{equation}
where $\by_i\in\R^{k}$ contains repeated measurements for the $i$-th subject,
$\bbeta\in\R^q$ is an unknown parameter vector,
$\bX_i\in\R^{k\times q}$ is a known design matrix, and
the $\mathbf{u}_i\in\R^{k}$ are unobservable independent mean zero random vectors with
covariance matrix $\bV\in\text{PDS}(k)$,
the class of positive definite symmetric $k\times k$ matrices.
The model is balanced in the sense that all~$\mathbf{y}_i$ have the same dimension.
Furthermore, we consider a structured covariance matrix, that is,
the matrix $\bV=\bV(\btheta)$ is a known function of unknown covariance parameters combined in a vector 
$\btheta\in\mathbf{\bTheta}\subset\R^l$.
We first discuss some examples that are covered by this setup
in the context of MM-estimators.

\begin{example}
\label{ex:LME model}
An important case of interest is the (balanced) linear mixed effects model
$\by_i=\bX_i\bbeta+\bZ\bgamma_i+\beps_i$, for
$i=1,\ldots,n$.
This model arises from $\mathbf{u}_i=\mathbf{Z}\bgamma_i+\beps_i$,
for $i=1,\ldots,n$,
where $\mathbf{Z}\in\R^{k\times g}$ is known and $\bgamma_i\in\R^g$ and~$\beps_i\in\R^k$ are independent mean zero random variables, with
unknown covariance matrices~$\mathbf{G}$ and $\bR$, respectively.
In this case, $\bV(\btheta)=\bZ\bG\bZ^T+\bR$ and
$\btheta=(\vch(\bG)^T,\vch(\bR)^T)^T$,
where
\begin{equation}
\label{def:vech}
\vch(\bA)=(a_{11},\ldots,a_{k1},a_{22},\ldots,a_{kk})
\end{equation}
is the unique $k(k+1)/2$-vector that stacks the columns of the lower triangle elements of a symmetric matrix $\bA$.
In full generality, the model is usually overparametrized and one may run into identifiability problems.
A more feasible example is obtained by taking $\bR=\sigma_0^2\bI_k$,
$\bZ=\left[\bZ_1 \,\cdots\,\bZ_r\right]$ and
$\bgamma_i=(\bgamma_{i1}^T,\ldots,\bgamma_{ir}^T)^T$,
where the $\bZ_j$'s are known $k\times g_j$ design matrices and
the $\bgamma_{ij}\in\R^{g_j}$ are independent mean zero random variables with covariance matrix $\sigma_j^2\bI_{g_j}$,
for $j=1,\ldots,r$.
This leads to
\begin{equation}
\label{def:linear mixed effects model Copt}
\by_i=\bX_i\bbeta+\sum_{j=1}^r \bZ_j\bgamma_{ij}+\beps_i,
\quad
i=1,\ldots,n,
\end{equation}
which was considered in Copt and Heritier~\cite{copt&heritier2007}.
In this case, $\bV(\btheta)=\sigma_0^2\bI_k+\sum_{j=1}^r\sigma_j^2\bZ_j\bZ_j^T$ 
and~$\btheta=(\sigma_0^2,\sigma_1^2,\ldots,\sigma_r^2)$.
\end{example}

\begin{example}
\label{ex:multivariate linear regression}
Another example of~\eqref{def:model} is the multivariate linear regression model
\begin{equation}
\label{def:multivariate linear regression model}
\by_i=\bB^T\bx_i+\bu_i,
\qquad
i=1,\ldots,n,
\end{equation}
considered in Kudraszow and Maronna~\cite{kudraszow-maronna2011}, where $\bB\in\R^{q\times k}$ is a matrix of unknown parameters, 
$\bx_i\in\R^q$ is known,
and~$\mathbf{u}_i$, for $i=1,\ldots,n$, are independent mean zero random variables with
covariance matrix~$\bV(\btheta)=\bSigma\in\text{PDS}(k)$.
In this case, the vector of unknown covariance parameters is given  by
$\btheta=\vch(\bSigma)$, where $\vch(\cdot)$ is defined in~\eqref{def:vech}.
The model can be obtained as a special case of~\eqref{def:model}, by taking
$\bX_i=\bx_i^T\otimes \bI_k$ and $\bbeta=\vc(\bB^T)$, where~$\vc(\cdot)$ is the~$k^2$-vector that stacks the columns of a matrix.
Clearly, the multiple linear regression model considered in Yohai~\cite{yohai1987} is a special case 
of~\eqref{def:multivariate linear regression model} with $k=1$.
\end{example}

\begin{example}
\label{ex:multivariate location-scatter}
Also the multivariate location-scatter model, as considered in Lopuha\"a~\cite{lopuhaa1992highly},
Tatsuoka and Tyler~\cite{tatsuoka&tyler2000}, and
Salibi\'an-Barrera \textit{et al}~\cite{SalibianBarrera-VanAelst-Willems2006},
can be obtained as a special case of~\eqref{def:model},
by taking $\bX_i=\bI_k$, the $k\times k$ identity matrix.
In this case, $\bbeta\in\R^k$ is the unknown location parameter and covariance matrix
$\bV(\btheta)=\bSigma\in\text{PDS}(k)$, with $\btheta=\vch(\bSigma)$.
Note that this model can also be obtained as a special case of~\eqref{def:multivariate linear regression model}
by taking $\bx_i=1$ and~$\bB^T=\bbeta$.
This means that results in Kudraszow and Maronna~\cite{kudraszow-maronna2011} for model~\eqref{def:multivariate linear regression model}
also apply to the multivariate location-scatter model.
\end{example}

\begin{example}
\label{ex:time series}
Model~\eqref{def:model} also includes examples, for which $\bu_1,\ldots,\bu_n$ are generated by a time series.
An example is the case where $\bu_i$ has a covariance matrix with elements
$v_{st}=\sigma^2\rho^{|s-t|}$,
for $s,t=1,\ldots,n$.
This arises when the $\bu_i$'s are generated by an autoregressive process of order one.
The vector of unknown covariance parameters is $\btheta=(\sigma^2,\rho)\in(0,\infty)\times(-1,1)$.
A general stationary process leads to
$v_{st}=\theta_{|s-t|+1}$,
for $s,t=1,\ldots,n$,
in which case $\btheta=(\theta_1,\ldots,\theta_k)^T\in\R^k$,
where $\theta_{|s-t|+1}$ represents the autocovariance over lag~$|s-t|$.
\end{example}

Throughout the manuscript we will assume that the parameter $\btheta$ is identifiable in the sense that,
$\bV(\btheta_1)=\bV(\btheta_2)$ implies $\btheta_1=\btheta_2$.
This is true for all models in
Examples~\ref{ex:multivariate linear regression},
\ref{ex:multivariate location-scatter}, and~\ref{ex:time series}.
This may not be true in general for the linear mixed effects model
in Example~\ref{ex:LME model} with unknown $\vch(\bG)$ and $\vch(\bR)$.
For linear mixed effects models
in~\eqref{def:linear mixed effects model Copt},
identifiability of~$\btheta=(\sigma_0^2,\sigma_1^2,\ldots,\sigma_r^2)$
holds for particular choices of the design matrices
$\bZ_1,\ldots,\bZ_r$.

\section{Definitions}
\label{sec:definitions}
We start by representing our observations as points in $\R^k\times\R^{kq}$  in the following way.
For $r=1,\ldots, k$, let $\bx_r^T$ denote the $r$-th row of the $k\times q$ matrix $\bX$,
so that~$\bx_r\in\R^q$.
We represent the pair $\mathbf{s}=(\mathbf{y},\bX)$ as an element in $\R^k\times\R^{kq}$ defined 
by~$\bs^T=(\by^T,  \bx_{1}^T,\ldots,  \bx_{k}^T)$.
In this way our observations can be represented as $\bs_1,\ldots,\bs_n$, with $\bs_i=(\by_i,\bX_i)\in\R^k\times\R^{kq}$.

Similar to MM-estimators for multiple linear regression introduced by Yohai~\cite{yohai1987}, 
MM-estimators for $(\bbeta,\btheta)$ are based on two loss functions.
We require the following conditions for a loss function~$\rho$:
\begin{itemize}
\item[(R1)]
$\rho$ is symmetric around zero with $\rho(0)=0$ and $\rho$ is continuous at zero;
\item[(R2)]
There exists a finite constant $c>0$, such that $\rho$ is non-decreasing on $[0,c]$ and constant on $[c,\infty)$;
\item[(R3)]
$\rho$ is continuous and strictly increasing on $[0,c]$.
\end{itemize}
In comparison with other proposals for MM-estimators,
conditions (R1)-(R3) imply condition (A1) in Yohai~\cite{yohai1987}
and Definition~2 in Kudraszow and Maronna~\cite{kudraszow-maronna2011}.
The conditions are similar to the ones in Tatsuoka and Tyler~\cite{tatsuoka&tyler2000}
and the ones in Salibi\'an-Barrera \textit{et al}~\cite{SalibianBarrera-VanAelst-Willems2006}.
\begin{itemize}
\item[STAGE 1:]
Let $\bbeta_{0,n}$ and $\btheta_{0,n}$ be initial (high breakdown) estimators for $\bbeta$ and $\btheta$,
and consider the shape estimator $\bGamma(\btheta_{0,n})$, where for $\btheta\in\bTheta$,
\begin{equation}
\label{def:Gamma}
\bGamma(\btheta)
=
\frac{\bV(\btheta)}{|\bV(\btheta)|^{1/k}},
\end{equation}
where $|\bA|$ denotes the determinant of $\bA$.
\item[STAGE 2:]
Let $\rho_0$ satisfy (R1)-(R3) and
determine $\sigma_n$ by solving $\sigma$ from
\begin{equation}
\label{def:initial estimators}
\frac{1}{n}
\sum_{i=1}^n
\rho_0
\left(
\frac{\displaystyle\sqrt{(\by_i-\bX_i\bbeta_{0,n})^T
\bGamma(\btheta_{0,n})^{-1}
(\by_i-\bX_i\bbeta_{0,n})}}{\sigma}
\right)
=
b_0,
\end{equation}
where $0<b_0<\sup\rho_0$.
\item[STAGE 3:]
Let $\rho_1$ satisfy (R1)-(R3) and is such that
\begin{equation}
\label{eq:ineq rho functions}
\frac{\rho_1}{\sup\rho_1}\leq \frac{\rho_0}{\sup\rho_0}.
\end{equation}
For $(\bbeta,\bC)\in\R^q\times\text{PDS}(k)$, define
\begin{equation}
\label{def:Rn}
R_{n}(\bbeta,\bC)
=
\frac1n
\sum_{i=1}^{n}
\rho_1\left(
\frac{\displaystyle\sqrt{(\by_i-\bX_i\bbeta)^T\bC^{-1}(\by_i-\bX_i\bbeta)}}{\sigma_n}
\right),
\end{equation}
and let
$\mathfrak{D}
=
\left\{
(\bbeta,\bgamma)\in\R^q\times\bTheta:\bV(\bgamma)\in\text{PDS}(k)\text{ with }|\bV(\gamma)|=1
\right\}$.
Let $(\bbeta_{1,n},\bgamma_n)\in \mathfrak{D}$ be any local minimum of
$R_{n}(\bbeta,\bV(\bgamma))$
that satisfies
\begin{equation}
\label{eq:ineq Rn}
R_{n}(\bbeta,\bV(\bgamma))
\leq
R_{n}(\bbeta_{0,n},\bGamma(\btheta_{0,n})),
\end{equation}
where $\bGamma$ is defined in~\eqref{def:Gamma}.
Update the covariance estimator by $\bV_{1,n}=\sigma_n^2\bV(\bgamma_n)$
and update the estimator~$\btheta_{1,n}$ for the vector of covariance parameters as the solution of
\begin{equation}
\label{eq:update theta}
\bV(\btheta)=\sigma_n^2\bV(\bgamma_n).
\end{equation}
\end{itemize}
The idea is to choose estimators $\bbeta_{0,n}$ and $\btheta_{0,n}$ with high breakdown point
and to choose loss function~$\rho_0$ suitably, so that $\sigma_n$ will also have high breakdown point.
Estimators~$\bbeta_{1,n}$ and~$\btheta_{1,n}$ will be shown to inherit this high breakdown point, 
but at the same time the regression estimator~$\bbeta_{1,n}$ as well as the estimator of shape $\bV_{1,n}/|\bV_{1,n}|^{1/k}$
and the estimator of direction~$\btheta_{1,n}/\|\btheta_{1,n}\|$ will also have high efficiency relative to the least squares estimators
by suitable choice of loss function~$\rho_1$.
We will show that the absolute minimum of $R_n(\bbeta,\bV(\bgamma))$ with $|\bV(\bgamma)|=1$ exists.
Clearly, this absolute minimum satisfies~\eqref{eq:ineq Rn}.
However, any local minimum satisfying~\eqref{eq:ineq Rn}, will also be an MM-estimator
with high breakdown point and high efficiency.

Examples of loss functions satisfying~\eqref{eq:ineq rho functions} can be constructed from 
Tukey's bi-weight, defined as
\begin{equation}
\label{def:biweight}
\rho_{\mathrm{B}}(s;c)
=
\begin{cases}
s^2/2-s^4/(2c^2)+s^6/(6c^4), & |s|\leq c\\
c^2/6 & |s|>c.
\end{cases}
\end{equation}
The functions $\rho_0(s)=\rho_{\mathrm{B}}(s;c_0)$ and $\rho_1(s)=\rho_{\mathrm{B}}(s;c_1)$, for $0<c_0\leq c_1<\infty$,
satisfy (R1)-(R3) as well as~\eqref{eq:ineq rho functions}.
Examples of $\bbeta_{0,n}$ and $\btheta_{0,n}$ with high breakdown point are the S-estimators 
discussed in Lopuha\"a \textit{et al}~\cite{lopuhaa-gares-ruizgazen2023} defined with $\rho_0(s)=\rho_{\mathrm{B}}(s;c_0)$.
Small values of the cut-off constant $c_0$ will then correspond to a high breakdown point.

This definition of MM-estimators for the linear mixed effects model differs from the ones in Lopuha\"a~\cite{lopuhaa2023} and
Copt and Heritier~\cite{copt&heritier2007}, where 
the entire initial covariance matrix is used as initial estimator,
which then involves minimization of~$R_n$ over $\bbeta$ only.
The current definition only uses the univariate estimator $\sigma_n$ for the scale parameter 
$|\bV(\btheta)|^{1/(2k)}$ as an auxiliary statistic.
The advantage is that this version of the MM-estimator allows for improvement of the efficiency of both the regression estimator 
as well as the estimator of the shape component of $\bV(\btheta)$ and the estimator of the direction component of $\btheta$.

The corresponding MM-functionals are defined similarly.
\begin{itemize}
\item[STAGE 1:]
Let $\bbeta_0(P)$ and $\btheta_0(P)$ be initial functionals and consider the shape functional~$\bGamma(\btheta_0(P))$, where $\bGamma$ is defined in~\eqref{def:Gamma}.
\item[STAGE 2:]
Let $\rho_0$ satisfy (R1)-(R3) and
determine $\sigma(P)$ by solving $\sigma$ from
\begin{equation}
\label{def:sigma}
\int
\rho_0
\left(
\frac{\displaystyle\sqrt{(\by-\bX\bbeta_0(P))^T\bGamma(\btheta_0(P))^{-1}(\by-\bX\bbeta_0(P))}}{\sigma}
\right)
\,
\text{d}P(\bs)
=
b_0,
\end{equation}
where $0<b_0<\sup\rho_0$.
\item[STAGE 3:]
Let $\rho_1$ satisfy (R1)-(R3) and is such that~\eqref{eq:ineq rho functions} holds.
For $(\bbeta,\bC)\in\R^q\times\text{PDS}(k)$, define
\begin{equation}
\label{def:R_P}
R_P(\bbeta,\bC)
=
\int
\rho_1\left(
\frac{\displaystyle\sqrt{(\by-\bX\bbeta)^T\bC^{-1}(\by-\bX\bbeta)}}{\sigma(P)}
\right)
\text{d}P(\bs).
\end{equation}
Let $(\bbeta_1(P),\bgamma(P))\in \mathfrak{D}$ be any local minimum of
$R_P(\bbeta,\bV(\bgamma))$
that satisfies
\begin{equation}
\label{eq:ineq RP}
R_{P}(\bbeta,\bV(\bgamma))
\leq
R_{P}(\bbeta_0(P),\bGamma(\btheta_0(P))),
\end{equation}
where $\bGamma$ is defined in~\eqref{def:Gamma}.
Update the covariance functional by
$\bV_1(P)=\sigma^2(P)\bV(\bgamma(P))$
and update the functional~$\btheta_1(P)$ for the vector of covariance parameters as the solution of
\begin{equation}
\label{def:theta1}
\bV(\btheta)=\sigma^2(P)\bV(\bgamma(P)).
\end{equation}
\end{itemize}

Let $\mathbb{P}_n$ be the empirical measure corresponding to observations $(\by_1,\bX_1),\ldots,(\by_n,\bX_n)$.
We assume that the initial functionals $\bbeta_0(\cdot)$ and $\btheta_0(\cdot)$ are such that
\begin{equation}
\label{eq:cond beta0 theta0}
(\bbeta_0(\mathbb{P}_n),\btheta_0(\mathbb{P}_n))=(\bbeta_{0,n},\btheta_{0,n}),
\end{equation}
where $(\bbeta_{0,n},\btheta_{0,n})$ are the initial estimators for $\bbeta$ and $\btheta$.
Examples for which~\eqref{eq:cond beta0 theta0} holds, are the S-functionals discussed in 
Lopuha\"a \textit{et al}~\cite{lopuhaa-gares-ruizgazen2023} defined with loss function $\rho_0$.
If~\eqref{eq:cond beta0 theta0} holds, then 
$\sigma_n=\sigma(\mathbb{P}_n)$, $(\bbeta_{1,n},\bgamma_n)=(\bbeta_1(\mathbb{P}_n),\bgamma(\mathbb{P}_n))$, 
and $\btheta_{1,n}=\btheta_1(\mathbb{P}_n)$.

The definition of MM-estimators and corresponding functionals in our current setup 
includes several special cases that are already available in the literature.
For the multivariate location and scatter model of Example~\ref{ex:multivariate location-scatter},
our MM-functionals $\bbeta_1(P)$ and $\sigma^2(P)\bV(\bgamma(P))$
coincide with the multivariate location and scatter M-functionals 
with auxiliary scale $\sigma(P)$, as discussed in Tatsuoka and Tyler~\cite{tatsuoka&tyler2000}.
When, in addition, $\bbeta_{0,n}$ and $\bC_{0,n}=\bV(\btheta_{0,n})$ are the S-estimators for location and scatter defined by means of~$\rho_0$,
then our MM-estimators~$\bbeta_{1,n}$ and~$\sigma_n^2\bV(\bgamma_n)$ coincide with the MM-estimators for location and scatter considered in Salibi\'an-Barrera \textit{et al}~\cite{SalibianBarrera-VanAelst-Willems2006}.
For the multivariate linear regression model of Example~\ref{ex:multivariate linear regression},
our MM-estimators~$\bbeta_{1,n}$ and~$\sigma_n^2\bV(\bgamma_n)$ coincide 
with the ones for multivariate linear regression in Kudraszow and Maronna~\cite{kudraszow-maronna2011}.
If, in addition $k=1$, our regression MM-estimator coincides with the one for multiple linear regression, as introduced in Yohai~\cite{yohai1987}.
Our MM-functionals then coincide with the M-functional with general scale $\sigma(P)$,
as treated in Martin \textit{et al}~\cite{martin-yohai-zamar1989}. 

\section{Existence}
\label{sec:existence}
We will first establish existence of the functionals $\sigma(P)$, $\bbeta_1(P)$, $\bgamma(P)$, and $\btheta_1(P)$,
under particular conditions on the probability measure $P$.
As a consequence, this will also yield the existence of the estimators $\sigma_n$, $\bbeta_{1,n}$, $\bgamma_n$, and $\btheta_{1,n}$.
Recall that the observations are represented as points $(\by_i,\bX_i)$ in $\R^k\times\R^{kq}$.
Note however, that for linear models with intercept the first column of each $\bX_i$ consists of 1's.
This means that the points~$(\by_i,\bX_i)$ are concentrated in a lower dimensional subset of $\R^k\times\R^{kq}$.
A similar situation occurs when all $\bX_i$ are equal to the same design matrix, such as in~Copt and Heritier~\cite{copt&heritier2007}.
In view of this, define $\mathcal{X}\subset\R^{kq}$ as the subset with the lowest dimension
$p=\text{dim}(\mathcal{X})\leq kq$ satisfying
\begin{equation}
\label{def:Xspace}
P(\bX\in \mathcal{X})=1.
\end{equation}
Hence, $P$ is then concentrated on the subset $\R^k\times \mathcal{X}$ of $\R^k\times\R^{kq}$, which
is of dimension~$k+p$, which may be smaller than $k+kq$.

The first condition that we require, expresses the fact that $P$
cannot have too much mass at infinity, in relation to the ratio $r_0=b_0/\sup\rho_0$.
\begin{itemize}
\item[$(\mathrm{C1}_\epsilon)$]
There exists a compact set $K_\epsilon\subset\R^k\times \mathcal{X}$,
such that $P(K_\epsilon)\geq r_0+\epsilon$.
\end{itemize}
The second condition requires that $P$ cannot have too much mass at arbitrarily thin strips in~$\R^k\times \mathcal{X}$.
For $\balpha\in\R^{k+kq}$, such that $\|\balpha\|=1$, $\ell\in\R$, and $\delta\geq0$, we define a strip
$H(\balpha,\ell,\delta)$ as follows:
\begin{equation}
\label{def:strip}
H(\balpha,\ell,\delta)
=
\left\{
\mathbf{s}\in\R^k\times\R^{kq}: \ell-\delta/2\leq \balpha^T\mathbf{s}\leq \ell+\delta/2
\right\}.
\end{equation}
Defined in this way, a strip is the area between two parallel hyperplanes
which are symmetric around the hyperplane $H(\balpha,\ell,0)
=
\left\{
\mathbf{s}\in\R^k\times\R^{kq}: \balpha^T\mathbf{s}=\ell
\right\}$.
Since the distance between two parallel hyperplanes
$\balpha^T\mathbf{s}=\ell_1$ and $\balpha^T\mathbf{s}=\ell_2$ is $|\ell_1-\ell_2|$,
the strip $H(\balpha,\ell,\delta)$ defined in~\eqref{def:strip} has width~$\delta$.
We require the following condition.
\begin{itemize}
\item[$(\mathrm{C2}_\epsilon)$]
The value
$
\delta_\epsilon
=
\inf
\left\{
\delta:
P\left(H(\balpha,\ell,\delta)
\right)\geq \epsilon,
\balpha\in\R^{k+kq},\|\balpha\|=1,\ell\in\R,\delta\geq 0
\right\}$
\newline
is strictly positive.
\end{itemize}
According to~\eqref{def:Xspace},  in $(\mathrm{C2}_\epsilon)$ one only needs to consider strips in $\R^k\times \mathcal{X}$.

Both conditions are satisfied for any $0<\epsilon\leq 1-r_0$ by any probability measure $P$
that is absolutely continuous.
Clearly, condition~$(\mathrm{C1}_\epsilon)$ holds for any $0\leq\epsilon\leq 1-r_0$ for
the empirical measure~$\mathbb{P}_n$ corresponding to a collection of $n$ points
$\mathcal{S}_n=\{\mathbf{s}_1,\ldots,\mathbf{s}_n\}\subset\R^k\times \mathcal{X}$.
Condition~$(\mathrm{C2}_\epsilon)$ with $\epsilon=(k+p+1)/n$
is also satisfied by the empirical measure $\mathbb{P}_n$, when the collection~$\mathcal{S}_n$ is in \emph{general position}, i.e.,
no subset $J\subset \mathcal{S}_n$ of $k+p+1$ points is contained in the same hyperplane
in~$\R^k\times \mathcal{X}$.
Conditions $(\mathrm{C1}_\epsilon)$ and~$(\mathrm{C2}_\epsilon)$ are the same as in
Lopuha\"a \textit{et al}~\cite{lopuhaa-gares-ruizgazen2023} and they
are similar to condition~$(\mathrm{C}_\epsilon)$
in Lopuha\"a~\cite{lopuhaa1989}.
The reason that $(\mathrm{C1}_\epsilon)$ slightly deviates from Lopuha\"a~\cite{lopuhaa1989}, is to handle
the presence of $\bX$ in minimizing~\eqref{eq:ineq RP}.

To establish existence of $\sigma(P)$ we follow the reasoning in Yohai~\cite{yohai1987}.
We require the following condition.
\begin{itemize}
\item[(C0)]
For 
$
E_0
=
\left\{
(\by,\bX)\in\R^k\times\R^{kq}:
\|\by-\bX\bbeta_0(P)\|=0
\right\}$,
it holds $P(E_0)<1-b_0/\sup\rho_0$.
\end{itemize}
We then have the following lemma.
\begin{lemma}
\label{lem:existence sigma}
Let $\rho_0$ satisfy (R1)-(R3) and let $(\bbeta_0(P),\btheta_0(P))\in\R^{q}\times\bTheta$ be the pair of initial functionals at $P$,
such that~(C0) holds.
Then a solution $\sigma(P)>0$ to~\eqref{def:sigma} exists and is unique.
\end{lemma}
To establish the existence of $(\bbeta_1(P),\bgamma(P))$, we follow the reasoning in Lopuha\"a \textit{et al}~\cite{lopuhaa-gares-ruizgazen2023}.
The idea is to argue that one can restrict oneself to a compact set for finding solutions to
minimizing~$R_P(\bbeta,\bV(\bgamma))$ subject to $|\bV(\bgamma)|=1$.
When $R_P(\bbeta,\bV(\bgamma))$ is continuous, this immediately yields the existence of a minimum.
To this end, we assume the following condition.
\begin{itemize}
\item[(V1)]
The mapping $\btheta\mapsto\bV(\btheta)$ is continuous.
\end{itemize}
To restrict oneself to $(\bbeta,\bgamma)$ in a compact set, we make use of Lemma~4.1
in Lopuha\"a \textit{et al}~\cite{lopuhaa-gares-ruizgazen2023}.
It requires that the identity is an element of
$\mathcal{V}=\{\bV(\btheta)\in\text{PDS}(k):\btheta\in\bTheta\subset\R^l\}$
and that $\mathcal{V}$ is closed under multiplication with a positive scalar.
\begin{itemize}
\item[(V2)]
There exists a $\btheta\in\bTheta\subset\R^l$, such that $\bV(\btheta)=\bI_k$.
For any $\bV(\btheta)\in \mathcal{V}$ and any $\alpha>0$, it holds that
$\alpha\bV(\btheta)=\bV(\btheta')$, for some $\btheta'\in\bTheta\subset\R^l$.
\end{itemize}
Conditions (V1)-(V2) are not very restrictive.
For example, all examples discussed in Section~\ref{sec:structured covariance model}
satisfy these conditions.
Also note that (V2) implies that~\eqref{def:theta1} has a solution $\btheta_1(P)$
and similarly for~\eqref{eq:update theta}.

Lemma~4.1 in Lopuha\"a \textit{et al}~\cite{lopuhaa-gares-ruizgazen2023} will ensure that there exists a compact set
in $\R^q\times\text{PDS}(k)$
that contains all pairs $(\bbeta,\bV(\bgamma))$ that correspond to possible minima
$(\bbeta,\bgamma)$ of~$R_P(\bbeta,\bV(\bgamma))$.
To establish that there also exists a compact set in $\mathfrak{D}$ that contains all 
possible minima $(\bbeta,\bgamma)$ of~$R_P(\bbeta,\bV(\bgamma))$,
we need that the pre-image
$\{\btheta\in\bTheta: \bV(\btheta)\in K\}$ of a compact set $K\subset\R^{k\times k}$ is again compact.
Recall that subsets of $\R^l$ are compact if and only if they are closed and bounded,
and note that the pre-image of a continuous mapping of a closed set is closed.
Hence, in view of condition (V1), it suffices to require the following condition.
\begin{itemize}
\item[(V3)]
The mapping $\btheta\mapsto \bV(\btheta)$ is such that the pre-image of a bounded set is bounded.
\end{itemize}
We then have the following theorem.
\begin{theorem}
\label{th:existence}
Let $\rho_0$ and $\rho_1$ satisfy~(R1)-(R2) and~\eqref{eq:ineq rho functions}.
Suppose $\rho_1$ is continuous and suppose that~$\bV$ satisfies (V1)-(V3).
Suppose $P$ satisfies $(\text{C1}_\epsilon)$ and $(\text{C2}_\epsilon)$, for some $0<\epsilon\leq 1-r_0$,
where $r_0=b_0/\sup\rho_0$. 
Let $(\bbeta_0(P),\btheta_0(P))\in\R^{q}\times\bTheta$ be the pair of initial functionals at $P$
and let $\sigma(P)$ be a solution to~\eqref{def:sigma}.
Then there exists a pair $(\bbeta_1(P),\bgamma(P))\in \mathfrak{D}$
that minimizes~$R_P(\bbeta,\bV(\bgamma))$
and a vector~$\btheta_1(P)\in\bTheta$ that is the unique solution of~\eqref{def:theta1}.
\end{theorem}
Theorem~\ref{th:existence} has a direct corollary for the existence of the MM-estimators, when dealing with a collections of points.
Let $\mathcal{S}_n=\{\bs_1,\ldots,\bs_n\}$, with $\bs_i=(\by_i,\bX_i)$, be a collection of $n$ points
in~$\R^k\times \mathcal{X}$.
Define
\begin{equation}
\label{def:k(S)}
\kappa(\mathcal{S}_n)
=
\text{maximal number of points of $\mathcal{S}_n$ lying on the same hyperplane in~$\R^k\times \mathcal{X}$.}
\end{equation}
For example, if the distribution $P$ is absolutely continuous, then
$\kappa(\mathcal{S}_n)\leq k+p$ with probability one.
Existence of $\sigma_n$ can be obtained from Lemma~\ref{lem:existence sigma}.
Suppose that~\eqref{eq:cond beta0 theta0} holds and 
that~$\#\{i: 1\leq i\leq n,\, 
\|\by_i-\bX_i\bbeta_{0,n}\|=0
\}
<n(1-b_0/\sup\rho_0)$.
Then $\mathbb{P}_n$ satisfies condition~(C0), so that the solution $\sigma_n$ of~\eqref{def:initial estimators}
exists and is unique, according to Lemma~\ref{lem:existence sigma}.
We then have the following corollary.
\begin{corollary}
\label{cor:existence estimator structured}
Suppose that $\rho_0$, $\rho_1$, and $\bV$ satisfy the conditions of Theorem~\ref{th:existence}.
For a collection $\mathcal{S}_n=\{\bs_1,\ldots,\bs_n\}\subset\R^k\times \mathcal{X}$,
with $\bs_i=(\by_i,\bX_i)$, for $i=1,\ldots,n$, 
let~$(\bbeta_{0,n},\btheta_{0,n})\in\R^{q}\times\R^l$ be the pair of initial estimators
satisfying~\eqref{eq:cond beta0 theta0} and let $\sigma_n$ be a solution to~\eqref{def:initial estimators}.
If $\kappa(\mathcal{S}_n)+1\leq n(1-r_0)$, where $r_0=b_0/\sup\rho_0$,
then there exists a pair~$(\bbeta_{1,n},\bgamma_n)\in \mathfrak{D}$ that minimizes~$R_n(\bbeta,\bV(\bgamma))$
and a vector~$\btheta_{1,n}$ that is the unique solution of~\eqref{eq:update theta}.
\end{corollary}
For the multivariate linear regression model of Example~\ref{ex:multivariate linear regression},
Kudraszow and Maronna~\cite{kudraszow-maronna2011} prove existence of $\bbeta_{1,n}=\vc(\bB_{1,n}^T)$
and $\bV(\bgamma_n)$, assuming $\kappa(\mathcal{S}_n)<n/2$.
Hence, their Theorem~1 follows from our Corollary~\ref{cor:existence estimator structured},
as long as $r_0\leq1/2-1/n$.
This holds for example, when S-estimators with maximal breakdown point are used as initial estimators
(see Theorem~6.1 in Lopuha\"a \textit{et al}~\cite{lopuhaa-gares-ruizgazen2023}).
Existence of the corresponding functionals is not discussed in Kudraszow and Maronna~\cite{kudraszow-maronna2011}.
This now follows from our Theorem~\ref{th:existence}.
For the multivariate location and scatter model in Example~\ref{ex:multivariate location-scatter}, 
the MM-functionals coincide with the multivariate location and scatter M-functionals defined
with loss function $\rho_1$ and with auxiliary scale~$\sigma(P)$,
defined as the solution of~\eqref{def:sigma}. 
Tatsuoka and Tyler~\cite{tatsuoka&tyler2000} establish existence for these functionals
under the assumption
\begin{equation}
\label{eq:condition tatsuoka&tyler}
\inf_{(\bbeta,\bgamma)\in \mathfrak{D}}
R_P(\bbeta,\bV(\bgamma))
<
(1-P(B))\sup\rho_1,
\end{equation}
for all hyperplanes $B\subset\R^k$.
It can be seen, using~\eqref{def:initial estimators} and~\eqref{eq:ineq rho functions}, that if our condition $(\text{C2}_\epsilon)$ holds for some $\epsilon<1-r_0$, 
then condition~\eqref{eq:condition tatsuoka&tyler} is satisfied.

Existence of MM-estimators has been obtained from the existence of MM-functionals at the empirical measure $\mathbb{P}_n$,
which converges to $P$, as $n$ tends to infinity.
The following corollary shows that existence can be established in general,
for probability measures that are close to~$P$.
This will become useful when we want to establish existence at 
perturbed measures~$(1-h)P+h\delta_{\bs}$, for $h$ sufficiently small, in order to determine the influence function of the functionals at~$P$
(see Section~\ref{sec:IF}).
It requires the following condition on $P$.
\begin{itemize}
\item[(C3)]
Let $\mathfrak{C}$ be the class of all measurable convex subsets of $\R^k\times \R^{kq}$.
Every $C\in \mathfrak{C}$ is a $P$-continuity set, i.e., $P(\partial C)=0$,
where~$\partial C$ denotes the boundary of $C$.
\end{itemize}
Condition (C3) is needed to apply Theorem 4.2 in Ranga Rao~\cite{rangarao1962}.
Clearly, this condition is satisfied if $P$ is absolutely continuous.
\begin{corollary}
\label{cor:existence weak convergence}
Suppose that $\rho_0$ satisfies the conditions of Lemma~\ref{lem:existence sigma}.
Let $P$ satisfy~(C0) and~(C3), and let $(\bbeta_0(P),\btheta_0(P))\in\R^{q}\times\bTheta$ 
be the pair of initial functionals at $P$.
Let $P_t$, $t\geq0$, be a sequence of probability measures on $\R^k\times \R^{kq}$ that converges weakly to~$P$, as $t\to\infty$.
Suppose that $(\bbeta_0(P_t),\btheta_0(P_t))$ exist, for $t$ sufficiently large, 
such that~$\bbeta_0(P_t)\to\bbeta_0(P)$.
Then
\begin{itemize}
\item[(i)]
for $t$ sufficiently large,
equation~\eqref{def:sigma} with $P=P_t$, has a unique solution $\sigma(P_t)$.
\end{itemize}
In addition, suppose that $\rho_0$, $\rho_1$, and $\bV$ satisfy the conditions of Theorem~\ref{th:existence},
and suppose that $P$ satisfies~$(\text{C1}_{\epsilon'})$ and $(\text{C2}_\epsilon)$,
for some $0<\epsilon<\epsilon'\leq 1-r_0$, where $r_0=b_0/\sup\rho_0$.
Then
\begin{itemize}
\item[(ii)]
for $t$ sufficiently large,
there exists~$(\bbeta_1(P_t),\bgamma(P_t))\in \mathfrak{D}$ that minimizes~$R_{P_t}(\bbeta,\bV(\bgamma))$
and a vector $\btheta_1(P_t)\in\bTheta$ that is the unique solution of equation~\eqref{def:theta1} with~$P=P_t$.
\end{itemize}
\end{corollary}

\section{Continuity and consistency}
\label{sec:continuity}
Consider a sequence $P_t$, $t\geq0$, of probability measures on $\R^k\times\R^{kq}$ that converges weakly to $P$, as $t\to\infty$.
By continuity of the MM-functional $(\bbeta_1(P),\btheta_1(P))$ we mean that 
$(\bbeta_1(P_t),\btheta_1(P_t))\to(\bbeta_1(P),\btheta_1(P))$, as $t\to\infty$.
An example of such a sequence is the sequence of empirical measures $\mathbb{P}_n$, $n=1,2,\ldots$, that converges weakly to $P$, almost surely.
Continuity of the MM-functional for this sequence would then mean that the MM-estimator 
$(\bbeta_{1,n},\btheta_{1,n})$ is consistent,
i.e., $(\bbeta_1(\mathbb{P}_n),\btheta_1(\mathbb{P}_n))\to(\bbeta_1(P),\btheta_1(P))$,
almost surely.

We have the following theorem establishing continuity of the MM-functionals.
\begin{theorem}
\label{th:continuity}
Let $\rho_0$ satisfy (R2)-(R3) and $\bV$ satisfy~(V1).
Let $(\bbeta_0(P),\btheta_0(P))\in\R^{q}\times\bTheta$ be the pair of initial functionals at $P$.
Let $P_t$, $t\geq0$, be a sequence of probability measures on $\R^k\times \R^{kq}$ that converges weakly to~$P$, as $t\to\infty$.
Suppose that $(\bbeta_0(P_t),\btheta_0(P_t))$ exist, for~$t$ sufficiently large, 
and suppose that $(\bbeta_0(P_t),\btheta_0(P_t))\to(\bbeta_0(P),\btheta_0(P))$.
Let $\sigma(P)$ be the unique solution of~\eqref{def:sigma}
and let $\sigma(P_t)$ be a solution of~\eqref{def:sigma}, with~$P=P_t$.
Then
\begin{itemize}
\item[(i)]
$\sigma(P_t)\to\sigma(P)$, as $t\to\infty$.
\end{itemize}
In addition, suppose that $\rho_1$ satisfies~\eqref{eq:ineq rho functions}
and~(R2)-(R3),
and that $\bV$ satisfies~(V3).
Suppose that~$P$ satisfies~(C3), as well as~$(\text{C1}_{\epsilon'})$ and $(\text{C2}_\epsilon)$,
for some $0<\epsilon<\epsilon'\leq 1-r_0$,
where $r_0=b_0/\sup\rho_0$.
For $t$ sufficiently large, 
let~$(\bbeta_1(P_t),\bgamma(P_t))\in \mathfrak{D}$ 
be a local minimum of~$R_{P_t}(\bbeta,\bV(\bgamma))$ that satisfies~\eqref{eq:ineq RP} for $P=P_t$,
and let~$(\bbeta_1(P),\bgamma(P))\in \mathfrak{D}$ be the unique minimizer of $R_P(\bbeta,\bV(\bgamma))$.
Then
\begin{itemize}
\item[(ii)]
$(\bbeta_1(P_t),\bgamma(P_t))\to(\bbeta_1(P),\bgamma(P))$, as $t\to\infty$;
\end{itemize}
Let $\btheta_1(P)$ and $\btheta_1(P_t)$ be solutions of~\eqref{def:theta1}
and~\eqref{def:theta1} with $P=P_t$, respectively.
Then
\begin{itemize}
\item[(iii)]
$\btheta_1(P_t)\to \btheta_1(P)$, as $t\to\infty$.
\end{itemize}
\end{theorem}

Continuity of the MM-functionals will be used to derive the influence function of the MM-functionals in Section~\ref{sec:IF}.
Another convenient consequence of the continuity of the MM-functionals is that one can directly obtain consistency of the MM-estimators.
Let $\mathcal{S}_n=\{\bs_1,\ldots,\bs_n\}$, with $\bs_i=(\by_i,\bX_i)$, be a collection of $n$ points
in~$\R^k\times \mathcal{X}$.
We apply Theorem~\ref{th:continuity} to the sequence $\mathbb{P}_n$, $n=1,2,\ldots$,
of probability measures, where~$\mathbb{P}_n$ is the empirical measure corresponding to $\mathcal{S}_n$.
\begin{corollary}
\label{cor:consistency}
Let $\rho_0$ and $\bV$ satisfy the conditions of Theorem~\ref{th:continuity}(i).
For a collection $\mathcal{S}_n=\{\bs_1,\ldots,\bs_n\}\subset\R^k\times \mathcal{X}$,
with $\bs_i=(\by_i,\bX_i)$, for $i=1,\ldots,n$, 
let $(\bbeta_{0,n},\btheta_{0,n})\in\R^{q}\times\R^l$ be the pair of initial estimators
satisfying~\eqref{eq:cond beta0 theta0} 
and suppose that $(\bbeta_{0,n},\btheta_{0,n})\to(\bbeta_0(P),\btheta_0(P))$, with probability one.
Let $\sigma(P)$ the unique solution of~\eqref{def:sigma}
and let $\sigma_n$ be a solution of~\eqref{def:initial estimators}.
Then
\begin{itemize}
\item[(i)]
$\sigma_n\to\sigma(P)$, with probability one.
\end{itemize}
In addition, suppose $\rho_1$, $\bV$, and $P$ satisfy the conditions of Theorem~\ref{th:continuity}(ii).
Let~$(\bbeta_{1,n},\bgamma_n)\in \mathfrak{D}$ be a local minimum of $R_n(\bbeta,\bV(\bgamma))$
that satisfies~\eqref{eq:ineq Rn}, 
and let~$(\bbeta_1(P),\bgamma(P))\in \mathfrak{D}$ be the unique minimizer of~$R_P(\bbeta,\bV(\bgamma))$.
Then
\begin{itemize}
\item[(ii)]
$(\bbeta_{1,n},\bgamma_n)\to(\bbeta_1(P),\bgamma(P))$, with probability one;
\end{itemize}
Let $\btheta_1(P)$ and $\btheta_{1,n}$ be solutions of~\eqref{def:theta1} 
and~\eqref{eq:update theta}, respectively.
Then 
\begin{itemize}
\item[(iii)]
$\btheta_{1,n}\to \btheta_1(P)$, with probability one.
\end{itemize}
\end{corollary}
When $\bV$ also satisfies~(V1), then for the covariance MM-estimator 
it follows from Corollary~\ref{cor:consistency} that 
$\bV(\btheta_{1,n})\to\bV(\btheta_1(P))=\sigma^2(P)\bV(\bgamma(P))$,
with probability one.
This extends Theorem~5 in Kudraszow and Maronna~\cite{kudraszow-maronna2011}.
Their result applies to MM-estimators for the multivariate models
in Examples~\ref{ex:multivariate linear regression} and~\ref{ex:multivariate location-scatter},
but is obtained only for distributions with an elliptical contoured density.

Theorem~\ref{th:continuity} and Corollary~\ref{cor:consistency} require that~$(\bbeta_1(P),\bgamma(P))\in \mathfrak{D}$ 
uniquely minimizes $R_P(\bbeta,\bV(\bgamma))$.
This situation is very similar to that of multivariate location-scatter M-estimators with
auxiliary scale, considered by Tatsuoka and Tyler~\cite{tatsuoka&tyler2000}.
For the special case that $\bX=\bI_k$, their Theorem~4.2 shows that~$R_P(\bbeta,\bC)$ has a unique minimum
for a broad class of distributions,
consisting of affine transformations of distributions on $\R^k$,
which are invariant under permutations and sign changes of its components and which have densities $g$ such that
$g\circ \exp$ is Schur-concave (see~\cite{tatsuoka&tyler2000} for details), i.e.,
\begin{equation}
\label{def:tyler density}
f_{\bmu,\bSigma}(\by)
=
|\bSigma|^{-1/2}
g(\bSigma^{-1/2}(\by-\bmu)).
\end{equation}
The next theorem is a direct consequence of that result.
Note that elliptically contoured densities 
are special cases of~\eqref{def:tyler density}.
Let $\E_{\bmu,\bSigma}$ denote the expectation with respect to $f_{\bmu,\bSigma}$.
\begin{theorem}
\label{th:davies}
Let $\rho_0$ satisfy (R1)-(R3) and suppose~$\rho_1$ is continuous and
satisfies~(R2) and~\eqref{eq:ineq rho functions}.
Suppose that $P$ is absolutely continuous,
such that for some $(\bbeta^*,\btheta^*)\in\R^q\times\bTheta$,
for all $\bX$, the distribution of $\by\mid\bX$ has density~$f_{\bmu,\bSigma}$ from~\eqref{def:tyler density}, with $\bmu=\bX\bbeta^*$ and $\bSigma=\bV(\btheta^*)$.
Suppose that $g$ in~\eqref{def:tyler density}
is strictly $M$-concave (see~\cite[Definition 4.4]{tatsuoka&tyler2000}).
Suppose~$\bV$ satisfies (V1)-(V3) and suppose $\bX$ has full rank with probability one.
Let $(\bbeta_0(P),\btheta_0(P))\in\R^{q}\times\bTheta$ be the pair of initial functionals at~$P$
satisfying $(\bbeta_0(P),\btheta_0(P))=(\bbeta^*,\btheta^*)$. 
Then, the following holds with probability one. 
\begin{itemize}
\item[(i)]
Equation~\eqref{def:sigma} has a unique solution $\sigma(P)$ 
and the function
$R_P(\bbeta,\bV(\bgamma))$ has a unique minimum $(\bbeta_1(P),\bgamma(P))\in \mathfrak{D}$,
that satisfies
$\bbeta_1(P)=\bbeta^*$  and 
$\bV(\bgamma(P))=\bSigma/|\bSigma|^{1/k}$.
\item[(ii)]
When $\bV(\alpha\btheta)=\alpha\bV(\btheta)$, for all $\alpha>0$, then $\btheta_1(P)=\btheta^*\sigma^2(P)/|\bSigma|^{1/k}$.
\item[(iii)]
When $b_0=\E_{\mathbf{0},\bI_k}\rho_0(\|\bz\|)$, then $\sigma(P)=|\bSigma|^{1/(2k)}$.
\end{itemize}
\end{theorem}
An example of initial functionals $(\bbeta_0(P),\btheta_0(P))$ that satisfy the 
conditions of Theorem~\ref{th:davies}, are the S-functionals defined with loss function~$\rho_0$,
see Theorem~5.3 in Lopuha\"a \emph{et al}~\cite{lopuhaa-gares-ruizgazen2023}
or Theorem~1 in Davies~\cite{davies1987} for the multivariate location-scatter model.

The proof of Theorem~\ref{th:davies} depends heavily on the application of
Theorem~4.2 in Tatsuoka and Tyler~\cite{tatsuoka&tyler2000} on the uniqueness
of multivariate M-functionals with auxiliary scale.
It considers strict M-concave densities $g$ in~\eqref{def:tyler density}, 
which is a broad class of densities that includes spherical symmetric densities, 
among others, see Tatsuoka and Tyler~\cite{tatsuoka&tyler2000} for details.
In this way, Theorem~\ref{th:davies} can be seen as an extension
of Theorem~1 in Davies~\cite{davies1987} on the uniqueness of multivariate 
location-scatter S-functionals at distributions with an elliptically contoured density.

\section{Global robustness: the breakdown point}
\label{sec:bdp}
Consider a collection of points $\mathcal{S}_n=\{\bs_i=(\by_i,\bX_i),i=1,\ldots,n\}\subset \R^k\times \mathcal{X}$.
To investigate the global robustness of the estimators,
we compute their finite-sample (replacement) breakdown point.
For a given collection~$\mathcal{S}_n$, the finite-sample breakdown point
(see Donoho and Huber~\cite{donoho&huber1983})
of an estimator is defined as the smallest proportion of points
from~$\mathcal{S}_n$ that one needs to replace in order to
send the estimator to the boundary of its parameter space.
To emphasize the dependence on the collection $\mathcal{S}_n$,
denote an estimator for the regression parameter by $\bbeta_n(\mathcal{S}_n)$
and an estimator for the vector of covariance parameters by $\btheta_n(\mathcal{S}_n)$.
For a given collection $\mathcal{S}_n$, the finite-sample breakdown point
of a regression estimator $\bbeta_n$ is defined as
\begin{equation}
\label{def:BDP beta}
\epsilon_n^*(\bbeta_n,\mathcal{S}_n)
=
\min_{1\leq m\leq n}
\left\{
\frac{m}{n}:
\sup_{\mathcal{S}_m'}
\left\|
\bbeta_n(\mathcal{S}_n)-\bbeta_n(\mathcal{S}_m')
\right\|
=\infty
\right\},
\end{equation}
where the minimum runs over all possible collections $\mathcal{S}_m'$ that can be obtained from $\mathcal{S}_n$
by replacing~$m$ points of $\mathcal{S}_n$ by arbitrary points in $\R^k\times \mathcal{X}$.

An estimator $\btheta_n$ for the vector of covariance parameters determines the covariance estimator~$\bV(\btheta_n)$.
For this reason it seems natural to let the breakdown point of $\btheta_n$ correspond to the breakdown of a covariance estimator.
For any $k\times k$ matrix $\bA$, let $\lambda_k(\bA)\leq\cdots\leq\lambda_1(\bA)$ denote the eigenvalues of $\bA$.
We define the finite sample (replacement) breakdown point of an estimator~$\btheta_n$ at a collection~$\mathcal{S}_n$, as
\begin{equation}
\label{def:BDP theta}
\epsilon_n^*(\btheta_n,\mathcal{S}_n)
=
\min_{1\leq m\leq n}
\left\{
\frac{m}{n}:
\sup_{\mathcal{S}_m'}
\text{dist}(\bV(\btheta_n(\mathcal{S}_n))),\bV(\btheta_n(\mathcal{S}_m'))
=\infty
\right\},
\end{equation}
with $\text{dist}(\cdot,\cdot)$ defined as
$\text{dist}(\bA,\mathbf{B})
=
\max\left\{
\left|\lambda_1(\bA)-\lambda_1(\mathbf{B})\right|,
\left|\lambda_k(\bA)^{-1}-\lambda_k(\mathbf{B})^{-1}\right|
\right\}$,
where the minimum runs over all possible collections $\mathcal{S}_m'$ that can be obtained from $\mathcal{S}_n$
by replacing~$m$ points of $\mathcal{S}_n$ by arbitrary points in $\R^k\times \mathcal{X}$.
So the breakdown point of $\btheta_n$ is the smallest proportion of points from~$\mathcal{S}_n$ that one needs to replace in order to
make the largest eigenvalue of $\bV(\btheta(\mathcal{S}_m'))$ arbitrarily large (explosion), or
to make the smallest eigenvalue of~$\bV(\btheta(\mathcal{S}_m'))$ arbitrarily small (implosion).

Good global robustness is illustrated by a high breakdown point.
The breakdown point of the MM-estimators is given in the theorem below.
\begin{theorem}
\label{th:bdp}
Let $\rho_0$ satisfy (R1)-(R3).
Let $\rho_1$ satisfy (R2) and~\eqref{eq:ineq rho functions}
and suppose $\bV$ satisfies (V1)-(V3).
Let $\mathcal{S}_n\subset \R^k\times \mathcal{X}$ be a collection of~$n$ points $\bs_i=(\by_i,\bX_i)$,  $i=1,\ldots,n$.
Let $r_0=b_0/\sup\rho_0$ and suppose that $0<\lfloor nr_0\rfloor <n-\kappa(\mathcal{S}_n)$,
where~$\kappa(\mathcal{S}_n)$ is defined by~\eqref{def:k(S)}.
Let $(\bbeta_{0,n},\btheta_{0,n})$ be initial estimators for $(\bbeta,\btheta)$.
Let $(\bbeta_{1,n},\bgamma_n)\in \mathfrak{D}$ satisfy~\eqref{eq:ineq Rn}
and let $\btheta_{1,n}$ be a solution of~\eqref{eq:update theta}.
Then
\[
\epsilon^*_n(\bbeta_{1,n},\bgamma_n,\btheta_{1,n},\mathcal{S}_n)
\geq
\min\left\{
\epsilon^*_n(\bbeta_{0,n},\btheta_{0,n},\mathcal{S}_n),
\frac{\lceil nr_0\rceil}{n},
\frac{\lceil n-nr_0\rceil-\kappa(\mathcal{S}_n)}{n}
\right\}.
\]
\end{theorem}
An example of initial estimators $(\bbeta_{0,n},\btheta_{0,n})$ with high breakdown point,
are S-estimators defined with the function $\rho_0$,
as discussed in Lopuha\"a \emph{et al}~\cite{lopuhaa-gares-ruizgazen2023}.
According to their Theorem~6.1 and Remark~3,
it holds that 
$\epsilon^*_n(\bbeta_{0,n},\btheta_{0,n},\mathcal{S}_n)\geq \min\{\lceil nr_0\rceil,\lceil n-nr_0\rceil-\kappa(\mathcal{S}_n)\}/n$.
In this case, the lower bound in Theorem~\ref{th:bdp} simplifies to
$\min\{\lceil nr_0\rceil,\lceil n-nr_0\rceil-\kappa(\mathcal{S}_n)\}/n$.
The largest possible value of this lower bound is attained when $r_0=(n-\kappa(\mathcal{S}_n))/(2n)$.
In this case 
$\lceil nr_0\rceil=\lceil n-nr_0\rceil-\kappa(\mathcal{S}_n)
=\lceil (n-\kappa(\mathcal{S}_n))/2\rceil
=\lfloor (n-\kappa(\mathcal{S}_n)+1)/2\rfloor$.
When the collection $\mathcal{S}_n$ is in general position, then $\kappa(\mathcal{S}_n)= k+p$.
In that case the breakdown point of the MM-estimators is at least equal
to~$\lfloor (n-k-p+1)/2\rfloor/n$.
When all $\bX_i$ are equal to the same $\bX$, 
as in the multivariate location-scatter model, but 
also in the linear mixed effects models considered in Copt and Victoria-Feser~\cite{copt2006high}
and Copt and Heritier~\cite{copt&heritier2007},
one has $p=0$ and $\kappa(\mathcal{S}_n)=k$.
In that case, the lower bound of the breakdown point is equal to $\lfloor (n-k+1)/2\rfloor/n$.
This value coincides with the maximal breakdown point for affine equivariant estimators
for $k\times k$ covariance matrices (see Davies~\cite[Theorem 6]{davies1987}).

The breakdown point for (a simpler version of) regression MM-estimators for the
linear mixed effects model~\eqref{def:linear mixed effects model Copt} 
has only been discussed in Copt and Heritier~\cite{copt&heritier2007}.
They conjecture that the exact value can be derived using the technique in
Van Aelst and Willems~\cite{vanaelst&willems2005}, but do not pursue a rigorous derivation.
The result in Theorem~\ref{th:bdp} applies to the current more extensive version of MM-estimators 
for the linear mixed effects model~\eqref{def:linear mixed effects model Copt}.
Furthermore, for $0<r_0\leq (n-\kappa(\mathcal{S}_n))/(2n)$, it holds that 
$\lceil nr_0\rceil\leq \lceil n-nr_0\rceil-\kappa(\mathcal{S}_n)$.
In this case, the lower bound for the breakdown point in Theorem~\ref{th:bdp}
coincides with that of the regression MM-estimator considered in Lopuha\"a~\cite{lopuhaa2023}.

For the multivariate linear regression model,
Kudraszow and Maronna~\cite{kudraszow-maronna2011} take $r_0=1/2$ and consider the case $\kappa(\mathcal{S}_n)<n/2$.
For this situation $\lceil nr_0\rceil>\lceil n-nr_0\rceil-\kappa(\mathcal{S}_n)$.
Hence, their Theorem~3 follows from our Theorem~\ref{th:bdp} for the case $r_0=1/2$.
For the multivariate location-scatter model, Salibi\'an-Barrera \emph{et al}~\cite{SalibianBarrera-VanAelst-Willems2006}
consider MM-estimators with S-estimators as initial estimators.
Our Theorem~\ref{th:bdp} then coincides with their Theorem~1.
For the MM-estimators in this model, Tyler~\cite{tyler2002} considers the gross error breakdown point,
which for finite collections is related to the finite sample contamination breakdown point.

\section{Score equations}
\label{sec:equations}
Up to this point, properties of MM-functionals and MM-estimators have been derived from 
minimizing $R_n(\bbeta,\bV(\bgamma))$ and $R_P(\bbeta,\bV(\bgamma))$,
as defined in~\eqref{def:Rn} and~\eqref{def:R_P}, respectively.
To obtain the influence function and to establish the limiting distribution of MM-estimators,
we use the score equations that can be found by differentiation of the Lagrangian corresponding to the
constrained minimization problem.
To this end, we require the following additional condition on the function~$\rho_1$,
\begin{itemize}
\item[(R4)]
$\rho_1$ is continuously differentiable and $u_1(s)=\rho_1'(s)/s$ is continuous,
\end{itemize}
and the following condition on the mapping $\btheta\mapsto\bV(\btheta)$,
\begin{itemize}
\item[(V4)]
$\bV(\btheta)$ is continuously differentiable.
\end{itemize}
Obviously, condition (V4) implies the former condition~(V1).
For 
$\by\in\R^k,\bt\in\R^k$, 
and $\bC\in\text{PDS}(k)$, define the Mahalanobis distances by
\begin{equation}
\label{def:Mahalanobis distance}
d^2(\by,\bt,\bC)=(\by-\bt)^T\bC^{-1}(\by-\bt).
\end{equation}
We then have the following proposition.
\begin{proposition}
\label{prop:score equations}
Let $\rho_1$ satisfy~(R2) and~(R4), and $\bV$ satisfy~(V4), and suppose that~$\E\|\bX\|<\infty$.
Let $(\bbeta_0(P),\btheta_0(P))$ be the pair of initial functionals and let $\sigma(P)$ be a solution of~\eqref{def:sigma}.
Then any local minimum $\bxi(P)=(\bbeta_1(P),\bgamma(P))\in \mathfrak{D}$ of $R_P(\bbeta,\bV(\bgamma))$ satisfies
\begin{equation}
\label{eq:Psi=0}
\int
\Psi(\bs,\bxi,\sigma(P))
\,\dd P(\bs)
=
\mathbf{0},
\end{equation}
where $\Psi=(\Psi_{\bbeta},\Psi_{\bgamma})$, with $\Psi_{\bbeta}$ and $\Psi_{\bgamma}=(\Psi_{\bgamma,1},\ldots,\Psi_{\bgamma,l})$
given by
\begin{equation}
\label{def:Psi}
\begin{split}
\Psi_{\bbeta}(\bs,\bxi,\sigma)
&=
u_1\left(\frac{d}{\sigma}\right)
\bX^T\bV^{-1}(\by-\bX\bbeta)\\
\Psi_{\bgamma,j}(\bs,\bxi,\sigma)
&=
u_1\left(\frac{d}{\sigma}\right)
(\by-\bX\bbeta)^T\bV^{-1}
\bH_{1,j}
\bV^{-1}(\by-\bX\bbeta)\\
&\qquad\qquad\qquad\qquad\quad-
\mathrm{tr}\left(\bV^{-1}\frac{\partial \bV}{\partial \gamma_j}\right)\log|\bV|,
\end{split}
\end{equation}
with $d=d(\by,\bX\bbeta,\bV(\bgamma))$, as defined in~\eqref{def:Mahalanobis distance}, 
and $\bV(\bgamma)$ is abbreviated by $\bV$, and where
\begin{equation}
\label{def:Hj}
\bH_{1,j}
=
\mathrm{tr}\left(\bV^{-1}\frac{\partial \bV}{\partial \gamma_j}\right)
\left(
\sum_{t=1}^l\gamma_t\frac{\partial \bV}{\partial \gamma_t}
\right)
-
\mathrm{tr}\left(\bV^{-1}\sum_{t=1}^l\gamma_t\frac{\partial \bV}{\partial \gamma_t}\right)
\frac{\partial \bV}{\partial \gamma_j},
\end{equation}
for $j=1,\ldots,l$.
\end{proposition}
Since $d(\by,\bX\bbeta,\bV(\bgamma))/\sigma=d(\by,\bX\bbeta,\sigma^2\bV(\bgamma))$,
the regression score equation for $\Psi_{\bbeta}$ is similar to 
the one for the regression MM-estimator considered in Lopuha\"a~\cite{lopuhaa2023},
defined with initial covariance functional $\bV_0(P)$, 
and both score equations coincide when $\bV_0(P)=\sigma^2(P)\bV(\bgamma(P))$.
Similarly, the regression score equation for $\Psi_{\bbeta}$ with the empirical measure $\mathbb{P}_n$ for $P$ in~\eqref{eq:Psi=0}
is similar to equation~(8) for the regression MM-estimator in the linear mixed effects model considered in Copt and Heritier~\cite{copt&heritier2007},
defined with initial covariance estimator $\widehat{\bSigma}_S$, 
and both equations coincide when $\widehat{\bSigma}_S=\sigma_n^2\bV(\bgamma_n)$.
For the multivariate linear regression model, the empirical score equation for 
$\Psi_{\bbeta}$ coincides with equation~(2.10) for the regression MM-estimator 
discussed in Kudraszow and Maronna~\cite{kudraszow-maronna2011}. 
When the initial estimators $(\bbeta_{0,n},\btheta_{0,n})$ are S-estimators, then
$\sigma_n=|\bV(\btheta_{0,n})|^{1/(2k)}$.
In that case, for the multivariate location-scatter model,
the empirical score equation for $\Psi_{\bbeta}$ 
coincides with fixed point equation~(16) for the location MM-estimator 
considered in Salibi\'an-Barrera \textit{et al}~\cite{SalibianBarrera-VanAelst-Willems2006}.

The function~$\Psi_{\bgamma}$ simplifies in situations where~$\bV(\bgamma)$ has a linear structure, that is
\begin{equation}
\label{def:V linear}
\bV(\bgamma)=\sum_{j=1}^l\gamma_j\bL_j,
\end{equation}
The covariance structures in Examples~\ref{ex:LME model}, \ref{ex:multivariate linear regression}, and~\ref{ex:multivariate location-scatter},
satisfy this property.
\begin{proposition}
\label{prop:score equations linear}
Suppose the conditions of Proposition~\ref{prop:score equations} hold and that $\bV$
has a linear structure~\eqref{def:V linear}.
Let $v_1(s)=\rho_1'(s)s$ and let $\bL$ be the $k^2\times l$ matrix
\begin{equation}
\label{def:L}
\bL
=
\Big[
  \begin{array}{ccc}
    \vc\left(\bL_1\right)
 & \cdots &     \vc\left(\bL_l\right) \\
  \end{array}
\Big].
\end{equation}
Then any local minimum 
$\bxi(P)=(\bbeta_1(P),\bgamma(P))\in \mathfrak{D}$ of~$R_P(\bbeta,\bV(\bgamma))$
satisfies~\eqref{eq:Psi=0},
with $\Psi=(\Psi_{\bbeta},\Psi_{\bgamma})$, where
\begin{equation}
\label{def:Psi linear}
\begin{split}
\Psi_{\bbeta}(\bs,\bxi,\sigma)
&=
u_1\left(\frac{d}{\sigma}\right)
\bX^T\bV^{-1}(\by-\bX\bbeta)\\
\Psi_{\bgamma}(\bs,\bxi,\sigma)
&=
-\bL^T
\left(
\bV^{-1}
\otimes
\bV^{-1}
\right)
\vc\left(
\Psi_\bV(\bs,\bxi,\sigma)
\right),
\end{split}
\end{equation}
where 
\begin{equation}
\label{def:PsiV}
\Psi_\bV(\bs,\bxi,\sigma)
=
k
u_1\left(\frac{d}{\sigma}\right)
(\by-\bX\bbeta)(\by-\bX\bbeta)^T
-
v_1\left(\frac{d}{\sigma}\right)
\sigma^2\bV
-
\bV
\log|\bV|,
\end{equation}
with $d=d(\by,\bX\bbeta,\bV(\bgamma))$, as defined in~\eqref{def:Mahalanobis distance},
and $\bV(\bgamma)$ is abbreviated by $\bV$.
\end{proposition}
For the multivariate linear regression model in Example~\ref{ex:multivariate linear regression}, 
one has $\bV(\bgamma)=\bGamma=\bSigma/|\bSigma|^{1/k}$, with $\bgamma=\vch(\bGamma)$.
The matrix $\bL=\partial\vc(\bV)/\partial\bgamma^T$ is then equal to the so-called duplication matrix~$\mathcal{D}_k$,
which is the unique $k^2\times k(k+1)/2$ matrix, with the properties
$\mathcal{D}_k\vch(\bC)=\vc(\bC)$ and $(\mathcal{D}_k^T\mathcal{D}_k)^{-1}\mathcal{D}_k^T\vc(\bC)=\vch(\bC)$
(e.g., see~\cite[Ch.~3, Sec.~8]{magnus&neudecker1988}).
Because~$\bV(\bgamma)$ has full rank, it follows that equation~\eqref{eq:Psi=0} holds for $\Psi=(\Psi_{\bbeta},\Psi_{\bV})$.
When we insert $\sigma_n^2\bV(\bgamma_n)=\sigma_n^2\bGamma_n$, the resulting score equations for the empirical measure~$\mathbb{P}_n$ corresponding to observations
$(\by_i,\bX_i)$, for $i=1,\ldots,n$, are then equivalent to 
equations~(2.10) and~(2.11) found in
Kudraszow and Maronna~\cite{kudraszow-maronna2011} for the regression MM-estimators.
For the multivariate location-scatter model in Example~\ref{ex:multivariate location-scatter}, 
one also has  $\bV(\bgamma)=\bGamma$, with
$\bgamma=\vch(\bGamma)$,
so that again equation~\eqref{eq:Psi=0} holds for $\Psi=(\Psi_{\bbeta},\Psi_{\bV})$.
When the initial estimators~$(\bbeta_{0,n},\btheta_{0,n})$ are S-estimators, 
it can be shown that the empirical score equation for $\Psi_{\bV}$
coincides with fixed point equation~(17) for the covariance shape MM-estimator 
considered in Salibi\'an-Barrera \textit{et al}~\cite{SalibianBarrera-VanAelst-Willems2006}.

\section{Local robustness: the influence function}
\label{sec:IF}
For $0<h<1$ and $\bs=(\by,\bX)\in\R^k\times\R^{kq}$ fixed, define the perturbed probability measure
$P_{h,\bs}=(1-h)P+h\delta_{\bs}$,
where $\delta_{\bs}$ denotes the Dirac measure at $\bs\in\R^k\times\R^{kq}$.
The \emph{influence function} of a functional~$T(\cdot)$ at probability measure $P$,
is defined as
\begin{equation}
\label{def:IF}
\text{IF}(\bs;T,P)
=
\lim_{h\downarrow0}
\frac{T((1-h)P+h\delta_{\bs})-T(P)}{h},
\end{equation}
if this limit exists.
In contrast to the global robustness measured by the breakdown point,
the influence function measures the local robustness.
It describes the effect of an infinitesimal contamination at a single point $\bs$ on the functional
(see Hampel~\cite{hampel1974}).
Good local robustness is therefore illustrated by a bounded influence function.

\subsection{The general case}
\label{subsec:IF general}
We will investigate when the limit in~\eqref{def:IF} exists for the functionals $\sigma$,
$\bxi=(\bbeta_1,\bgamma)$, and~$\btheta_1$ 
and derive their expression at general $P$.
Since the value of $\btheta_1$ determines the covariance matrix~$\bV(\btheta_1)$,
we also include the influence function of the covariance functional.
Consider the MM-functional at~$P_{h,\bs}$.
From the Portmanteau theorem~\cite[Theorem 2.1]{billingsley1968} it can easily be seen that $P_{h,\bs}\to P$, weakly, as~$h\downarrow0$.
Therefore, under the conditions of Corollary~\ref{cor:existence weak convergence} and Theorem~\ref{th:continuity}, it follows
that there exist a unique solution~$\sigma(P_{h,\bs})$ 
of equation~\eqref{def:sigma} with $P=P_{h,\bs}$, 
a pair $\bxi(P_{h,\bs})=(\bbeta_1(P_{h,\bs}),\bgamma(P_{h,\bs}))$ 
that minimizes $R_{P_{h,\bs}}(\bbeta,\bV(\bgamma))$ over $(\bbeta,\bgamma)\in \mathfrak{D}$,
and a vector~$\btheta_1(P_{h,\bs})\in\bTheta$ that uniquely solves~\eqref{eq:ineq RP}, for $P=P_{h,\bs}$.
Moreover, under these conditions 
$(\sigma(P_{h,\bs}),\bxi(P_{h,\bs}),\btheta_1(P_{h,\bs}))\to(\sigma(P),\bxi(P),\btheta_1(P))$,
as $h\downarrow0$.

For $\bxi=(\bbeta,\bgamma)\in \mathfrak{D}$ and $\sigma>0$,
define
\begin{equation}
\label{def:Lambda}
\Lambda(\bxi,\sigma)
=
\int \Psi(\bs,\bxi,\sigma)\,\dd P(\bs),
\end{equation}
where $\Psi=(\Psi_{\bbeta},\Psi_{\bgamma})$ is defined in~\eqref{def:Psi},
and write $\Lambda=(\Lambda_{\bbeta},\Lambda_{\bgamma})$, where
\begin{equation}
\label{def:Lambda-beta-gamma}
\begin{split}
\Lambda_{\bbeta}(\bxi,\sigma)
&=
\int
\Psi_{\bbeta}(\bs,\bxi,\sigma)\,\dd P(\bs),
\qquad
\Lambda_{\bgamma}(\bxi,\sigma)
=
\int
\Psi_{\bgamma}(\bs,\bxi,\sigma)\,\dd P(\bs).
\end{split}
\end{equation}
By definition of the variance components MM-functional $\btheta_1$, 
it can be expected that for general~$P$,
the influence function of $\btheta_1$ depends
on the influence functions of $\sigma$ and $\bgamma$.
Furthermore, because the function $\Psi$ depends on $\sigma$, it can be expected that for general $P$,
the influence function of~$\bxi=(\bbeta_1,\bgamma)$ depends on the influence function of~$\sigma$,
see Lemma~\ref{lem:IF xi sigma} in the supplemental material~\cite{supplement2025}.
In turn, since the functional $\sigma$ is defined as a solution of~\eqref{def:sigma}, it can be expected that for general $P$,
its influence function will depend on the initial functionals $\bzeta_0=(\bbeta_0,\btheta_0)$,
see Lemma~\ref{lem:IF sigma zeta} in the supplemental material~\cite{supplement2025}.

The situation becomes a bit simpler if we assume some kind of symmetry of the distribution~$P$.
A vector $\bb\in\R^q$ is called a \emph{point of symmetry} of $P$, if for almost all $\bX$, it holds that
$P\left(\bX\bb+A\mid\bX\right)
=
P\left(\bX\bb-A\mid\bX\right)$,
for all measurable sets $A\subset\R^k$,
where for $\lambda\in\R$ and $\bb\in\R^q$,
$\bX\bb+\lambda A$ denotes the set $\{\bX\bb+\lambda \by:\by\in A\}$.
If $\bb$ is a point of symmetry of $P$, it has the property that
$\E[G(\by-\bX\bb)]=\mathbf{0}$,
for any~$G(\bz)$, which is an odd function of $\bz\in\R^k$.
Furthermore, in order to obtain simpler expressions we also require the following condition on the function $\rho_1$.
\begin{itemize}
\item[(R5)]
$\rho_1$ is twice continuously differentiable.
\end{itemize}
Condition (R5) is needed to ensure
that $\partial\Lambda_{\bbeta}/\partial\bbeta$ is continuous at $(\bxi(P),\sigma(P))$.
The expressions for $\text{\rm IF}(\bs;\bgamma,P)$ and $\text{\rm IF}(\bs;\sigma,P)$ 
given in Lemmas~\ref{lem:IF xi sigma} and~\ref{lem:IF sigma zeta} in~\cite{supplement2025}
do simplify, but without further knowledge on the influence function
of~$\bzeta_0$ we can still not provide an explicit expression.
The situation is different for~$\bbeta_1$,
for which the influence function is given in the next theorem.
\begin{theorem}
\label{th:point of symmetry}
Suppose~$\rho_1$ satisfies~(R4) and $\bV$ satisfies~(V1).
Let $\sigma(P)$ be a solution of~\eqref{def:sigma} and let 
$\bxi(P)\in \mathfrak{D}$ be a local minimum of~$R_P(\bbeta,\bV(\bgamma))$.
Let~$\sigma(P_{h,\bs})$ be a solution of~\eqref{def:sigma} with $P=P_{h,\bs}$
and let $\bxi(P_{h,\bs})\in \mathfrak{D}$ be a local minimum of $R_P(\bbeta,\bV(\bgamma))$ with~$P=P_{h,\bs}$.
Suppose that $(\bxi(P_{h,\bs}),\sigma(P_{h,\bs}))\to(\bxi(P),\sigma(P))$,
as $h\downarrow0$.
Let $\Lambda_{\bbeta}$ be defined by~\eqref{def:Lambda-beta-gamma}
with $\Psi_{\bbeta}$ from~\eqref{def:Psi}
and suppose $\Lambda_{\bbeta}$ is continuously differentiable  
with a non-singular derivative 
$\bD_{\bbeta}=\partial\Lambda_{\bbeta}/\partial\bbeta$ at $(\bxi(P),\sigma(P))$.
Moreover, suppose that~$\bbeta_1(P)$ is a point of symmetry of~$P$.
Then, for $\bs\in\R^k\times\R^{kq}$, we have $\text{\rm IF}(\bs;\bbeta_1,P)=-\bD_{\bbeta}^{-1}\Psi_{\bbeta}(\bs,\bxi(P),\sigma(P))$.
\end{theorem}
The expression of $\text{\rm IF}(\bs;\bbeta_1,P)$ in Theorem~\ref{th:point of symmetry} 
is similar to that of the influence function of the regression MM-functional considered in Lopuha\"a~\cite{lopuhaa2023}
defined with~$\rho_1$ and initial covariance functional $\bV_0$, and both expressions
coincide when $\bV_0(P_{h,\bs})\to\sigma^2(P)\bV(\bgamma(P))$.

If one is allowed to interchange integration and differentiation in $\bD_{\bbeta}$, then
the expression for $\text{\rm IF}(\bs_0;\bbeta_1,P)$ in Corollary~\ref{th:point of symmetry}
coincides with that of the regression MM-functional in the multiple linear regression model considered in 
Yohai~\cite{yohai1987}.
For linear mixed effects models, multivariate linear regression models, or multivariate location-scatter model, 
expressions for the influence function of MM-functionals are either not available
or are restricted to distributions with an elliptically contoured density.
This situation is discussed in the next section for model~\eqref{def:model}.

\subsection{Elliptically contoured densities}
\label{subsec:IF elliptical}
We can obtain an even more detailed expression for the influence functions,
when $\bV$ has a linear structure and~$P$ has the following property.
\begin{itemize}
\item[(E)]
There exists $(\bbeta^*,\btheta^*)\in\R^q\times\bTheta$,
such that for all $\bX$, the distribution of $\by\mid\bX$ has an elliptically contoured density
\begin{equation}
\label{eq:elliptical}
f_{\bmu,\bSigma}(\by)
=
|\bSigma|^{-1/2}
m\left(
(\by-\bmu)^T
\bSigma^{-1}
(\by-\bmu)
\right),
\end{equation}
with $\bmu=\bX\bbeta^*$ and $\bSigma=\bV(\btheta^*)$ and $m:[0,\infty)\to[0,\infty)$.
\end{itemize}
We require the following condition on the mapping $\btheta\mapsto\bV(\btheta)$.
\begin{itemize}
\item[(V5)]
$\bV(\btheta)$ is twice continuously differentiable.
\end{itemize}
Condition (V5) is needed to
interchange the order of integration and differentiation in $\partial\Lambda/\partial\bxi$,
where $\Lambda$ is defined in~\eqref{def:Lambda}.
Clearly, condition (V5) implies former conditions (V4) and~(V1).

When the MM-functionals are affine equivariant,
then it suffices to determine the influence function for the case $(\bmu,\bSigma)=(\mathbf{0},\bI_k)$.
However, this does not hold in general for the MM-functionals in our setting.
Nevertheless, for the general case with $\bmu\in\R^k$ and $\bSigma\in\text{PDS}(k)$,
we can still use the fact that, conditionally on $\bX$, the distribution of $\by$ is the same as that of
$\bSigma^{1/2}\bz+\bmu$, where~$\bz$ has a spherical density $f_{\mathbf{0},\bI_k}$.
Let 
\begin{equation}
\label{def:alpha1-gamma1}
\begin{split}
\alpha_1
&=
\mathbb{E}_{\mathbf{0},\bI_k}
\left[
\left(
1-\frac{1}{k}
\right)
\frac{\rho_1'\left(c_\sigma\|\bz\|\right)}{c_\sigma\|\bz\|}
+
\frac1k
\rho_1''\left(c_\sigma\|\bz\|\right)
\right],\\
\gamma_1
&=
\frac{\mathbb{E}_{0,\mathbf{I}_k}
\left[
\rho_1''(c_\sigma\|\bz\|)(c_\sigma\|\bz\|)^2+(k+1)\rho_1'(c_\sigma\|\bz\|)c_\sigma\|\bz\|
\right]}{k+2},
\end{split}
\end{equation}
where $c_\sigma=|\bSigma|^{1/(2k)}/\sigma(P)$.
The influence functions of the MM-functionals~$\bbeta_1$ 
and~$\bgamma$ at distributions $P$, such that~$\by\mid\bX$ has an elliptical contoured density, 
are now given by the following theorem.
\begin{theorem}
\label{th:IF elliptical}
Suppose $P$ satisfies~(E) for some $(\bbeta^*,\btheta^*)\in\R^q\times\bTheta$
and~$\E\|\bX\|^2<\infty$.
Let~$\rho_1$ satisfy (R2), (R4)-(R5) 
and suppose $\bV$ satisfies~(V5) and has a linear structure~\eqref{def:V linear}.
Let $\sigma(P)$ be a solution of~\eqref{def:sigma} and let 
$\bxi(P)\in \mathfrak{D}$ be a local minimum of~$R_P(\bbeta,\bV(\bgamma))$
that satisfies $\bbeta_1(P)=\bbeta^*$ and $\bV(\bgamma(P))=\bSigma/|\bSigma|^{1/k}$.
Let~$\sigma(P_{h,\bs})$ be a solution of~\eqref{def:sigma} with $P=P_{h,\bs}$
and let $\bxi(P_{h,\bs})\in \mathfrak{D}$ be a local minimum of $R_P(\bbeta,\bV(\bgamma))$ with~$P=P_{h,\bs}$.
Suppose that $(\bxi(P_{h,\bs}),\sigma(P_{h,\bs}))\to(\bxi(P),\sigma(P))$,
as $h\downarrow0$.
Let $\alpha_1$ be defined in~\eqref{def:alpha1-gamma1} and suppose that~$\alpha_1\ne0$.
Suppose that~$\bX$ has full rank with probability one.
Then, for $\bs_0\in\R^k\times\R^{kq}$,
\[
\text{\rm IF}(\bs_0;\bbeta_1,P)
=
\frac{u_1\left(c_\sigma\|\bz_0\|\right)}{\alpha_1}
\left(
\E\left[
\bX^T\bSigma^{-1}\bX
\right]
\right)^{-1}
\bX_0^T\bSigma^{-1/2}\bz_0,
\]
where $u_1(s)=\rho_1'(s)/s$,
$\bz_0=\bSigma^{-1/2}(\by_0-\bX_0\bbeta^*)$, and  $c_\sigma=|\bSigma|^{1/(2k)}/\sigma(P)$.
In addition, 
suppose that $\text{\rm IF}(\bs;\sigma,P)$ exists.
Let $\gamma_1$ be defined in~\eqref{def:alpha1-gamma1} and suppose that $\gamma_1>0$.
Suppose the $k^2\times l$ matrix~$\bL$, as defined in~\eqref{def:L}, has full rank.
Then, for $\bs_0\in\R^k\times\R^{kq}$,
\[
\begin{split}
\text{\rm IF}(\bs_0;\bgamma,P)
&=
\frac{ku_1\left(c_\sigma\|\bz_0\|\right)}{\sigma^2(P)\gamma_1}
\Big(\bL^T(\bSigma^{-1}\otimes\bSigma^{-1})\bL\Big)^{-1}
\bL^T\left(\bSigma^{-1/2}\otimes\bSigma^{-1/2}\right)
\vc\left(\bz_0\bz_0^T\right)\\
&\quad-
\frac{v_1\left(c_\sigma\|\bz_0\|\right)}{|\bSigma|^{1/k}\gamma_1}
\btheta^*,
\end{split}
\]
where $v_1(s)=\rho_1'(s)s$.
\end{theorem}
When $c_\sigma=1$ (e.g., see Theorem~\ref{th:davies}(iii)), 
then $\text{\rm IF}(\bs_0;\bbeta_1,P)$ coincides with the influence function
of the regression MM-functional considered in Lopuha\"a~\cite{lopuhaa2023}
defined with~$\rho_1$ and an initial Fisher consistent covariance functional $\bV_0$.
Moreover, it also coincides with the influence function of the regression S-functional defined with~$\rho_1$
for model~\eqref{def:model},
see Corollary~8.4 in Lopuha\"a \emph{et al}~\cite{lopuhaa-gares-ruizgazen2023}.
This confirms the claim made by Salibi\'an-Barrera \emph{et al}~\cite{SalibianBarrera-VanAelst-Willems2006}
about the influence function of the location MM-functional in the model
of Example~\ref{ex:multivariate location-scatter}.

For the linear mixed effects model~\eqref{def:linear mixed effects model Copt},
Copt and Heritier~\cite{copt&heritier2007} discuss the influence function, but an expression is not provided.
The expression for the influence function of the regression MM-functional now follows from 
Theorem~\ref{th:IF elliptical}.
For the multivariate linear regression model of Example~\ref{ex:multivariate linear regression}, 
the expression for $\text{\rm IF}(\bs_0;\bbeta_1,P)$ in Theorem~\ref{th:IF elliptical}
coincides with one found for the regression MM-functional in Kudraszow and Maronna~\cite{kudraszow-maronna2011}.
Since the multivariate location-scatter model of Example~\ref{ex:multivariate location-scatter} 
is a special case of the multivariate linear regression model
by taking $\bx=1$ and $\bB^T=\bmu$, 
this also yields the expression for the influence function of the location MM-functional.
Finally, there is an interesting connection with the CM-functionals considered in Kent and Tyler~\cite{kent&tyler1996},
whose influence function depends on a parameter $\lambda_0$. 
The expression for $\text{\rm IF}(\bs_0;\bbeta_1,P)$ in Theorem~\ref{th:IF elliptical} is similar
to the one for the location CM-functional and they both coincide when $c_\sigma=1/\sqrt{\lambda_0}$.

It remains to determine the influence function of the variance component MM-functional~$\btheta_1$.
This is done in the next corollary.
\begin{corollary}
\label{cor:IF theta1}
Suppose that the conditions hold of Theorem~\ref{th:IF elliptical}
and suppose that~$\rho_0$ satisfies~(R2) and~(R4).
Let $\bzeta_0=(\bbeta_0,\btheta_0)$ be the pair of initial functionals
satisfying $(\bbeta_0(P),\btheta_0(P))=(\bbeta^*,\btheta^*)$,
and suppose that $\text{\rm IF}(\bs,\bzeta_0,P)$ exists.
Let $\btheta_1(P)$ and~$\btheta_1(P_{h,\bs_0})$ be solutions of equation~\eqref{def:theta1}
and equation~\eqref{def:theta1} with~$P=P_{h,\bs_0}$, respectively.
Let $u_1(s)=\rho_1'(s)/s$ and $v_1(s)=\rho_1'(s)s$,
and let $\bL$ be defined in~\eqref{def:L}.
Let $c_\sigma=|\bSigma|^{1/(2k)}/\sigma(P)$ and suppose that 
$\E_{\mathbf{0},\bI_k}[
\rho_0'
\left(
c_\sigma\|\bz\|
\right)
c_\sigma\|\bz\|]
>0$.
Then, for $\bs_0\in\R^k\times\R^{kq}$,
\[
\begin{split}
\text{\rm IF}(\bs_0,\btheta_1,P)
&=
\frac{ku_1\left(c_\sigma\|\bz_0\|\right)}{\gamma_1}
\Big(\bL^T(\bSigma^{-1}\otimes\bSigma^{-1})\bL\Big)^{-1}
\bL^T\left(\bSigma^{-1/2}\otimes\bSigma^{-1/2}\right)
\vc\left(\bz_0\bz_0^T\right)\\
&\qquad-
\frac{1}{c_\sigma^2}
\left(
\frac{v_1(c_\sigma\|\bz_0\|)}{\gamma_1}
-
\frac{2(\rho_0(c_\sigma\|\bz_0\|)-b_0)}{\E_{\mathbf{0},\bI_k}
\left[
\rho_0'(c_\sigma\|\bz\|)c_\sigma\|\bz\|
\right]}
\right)
\btheta^*,
\end{split}
\]
where $\bz_0=\bSigma^{-1/2}(\by_0-\bX_0\bbeta^*)$.
\end{corollary}
Note that for $\bV$ with a linear structure, one has $\vc(\bV(\btheta_1(P)))=\bL\btheta_1(P)$.
Hence, the influence function for the covariance MM-functional follows immediately from Corollary~\ref{cor:IF theta1}:
\begin{equation}
\label{eq:IF Vtheta1}
\text{\rm IF}(\bs_0,\vc(\bV(\btheta_1)),P)
=
\bL
\text{\rm IF}(\bs_0,\btheta_1,P).
\end{equation}
Since $\rho_1'(s)s$ and $\rho_0(s)$ are bounded,
it follows that the influence functions of $\bgamma$, $\btheta_1$, and~$\bV(\btheta_1)$
are bounded uniformly in both $\by_0$ and $\bX_0$,
whereas $\mathrm{IF}(\bs_0,\bbeta_1,P)$ is bounded uniformly in $\by_0$, but not in $\bX_0$.

The expressions for the influence function of the covariance MM-functionals in Corollary~\ref{cor:IF theta1} and~\eqref{eq:IF Vtheta1}
are characterized by two real-valued functions,
\begin{equation}
\label{def:alphaC betaC}
\alpha_C(s)=\frac{k\rho_1'(s)}{\gamma_1s},
\quad
\beta_C(s)=
\frac{1}{c_\sigma^2}
\left(
\frac{\rho_1'(s)s}{\gamma_1}
-
\frac{2(\rho_0(s)-b_0)}{\E_{\mathbf{0},\bI_k}
\left[
\rho_0'(c_\sigma\|\bz\|)c_\sigma\|\bz\|
\right]}
\right).
\end{equation}
When $c_\sigma=1$, using that for $\bV$ with a linear structure, it holds $\bL\btheta^*=\vc(\bSigma)$,
this matches with the characterization of general structured covariance functionals obtained in Lopuha\"a~\cite{lopuhaa2025}. 
Such a characterization was already observed by Croux and Haesbroeck~\cite{croux-haesbroeck2000} for affine equivariant covariance functionals.
The function $\alpha_C$ in~\eqref{def:alphaC betaC} coincides with the one for the covariance S-functional defined with~$\rho_1$,
see Corollary~8.4 in Lopuha\"a \emph{et al}~\cite{lopuhaa-gares-ruizgazen2023}.
The function~$\beta_C$ in~\eqref{def:alphaC betaC} (with $c_\sigma=1$) has the same structure as the one for covariance S-functionals,
but the first term is built from~$\rho_1$, whereas the second term is built from $\rho_0$.
As expected, when $\rho_0=\rho_1=\rho$, the above characterization coincides with the 
influence function of the covariance S-functional defined with loss function~$\rho$.

Furthermore, for covariance functionals~$\bC$, it holds that
the influence function of a scale invariant mapping~$H(\bC)$,
i.e., $H(\lambda\bC)=H(\bC)$, for $\lambda>0$, only depends on the function~$\alpha_C$,
see~(8.3) in Kent and Tyler~\cite{kent&tyler1996} for covariance CM-functionals
or see Lemma~2 in Lopuha\"a~\cite{lopuhaa2025} for linearly structured covariance functionals.
Because the characterizations of the influence functions of covariance MM- and S-functionals have
the same function $\alpha_C$, it follows that the influence functions of any scale invariant mapping 
of covariance MM- and S-functionals are the same.
A typical example is the shape component $\bGamma(\btheta_1)$ of the covariance 
MM-functional, where $\bGamma$ is defined in~\eqref{def:Gamma}.
Lemma~2 in Lopuha\"a~\cite{lopuhaa2025}, together with Corollary~\ref{cor:IF theta1} and~\eqref{eq:IF Vtheta1},
yields that $\bGamma(\btheta_1)$ has influence function
\begin{equation}
\label{eq:IF Vgamma}
\begin{split}
\text{IF}(\bs_0;\bGamma(\btheta_1),P)
=
\frac{ku_1(c_\sigma\|\bz_0\|)}{\sigma^2(P)\gamma_1}
\bigg\{
\bL\Big(\bL^T(\bSigma^{-1}\otimes\bSigma^{-1})\bL)\Big)^{-1}
\bL^T
\left(\bSigma^{-1/2}\otimes\bSigma^{-1/2}\right)\\
\vc\left(\bz_0\bz_0^T\right)
-
\frac{\|\bz_0\|^2}{k}
\vc(\bSigma)
\bigg\},
\end{split}
\end{equation}
where $\bz_0=\bSigma^{-1/2}(\by_0-\bX_0\bbeta^*)$ 
and $\gamma_1$ is defined in~\eqref{def:alpha1-gamma1}.
When $c_\sigma=1$, then this is the same as the influence function of the shape S-functional defined with~$\rho_1$,
see Example~6 in Lopuha\"a~\cite{lopuhaa2025}.
For the location-scatter model of Example~\ref{ex:multivariate location-scatter},
this confirms the claim made in Salibi\'an-Barrera \emph{et al}~\cite{SalibianBarrera-VanAelst-Willems2006}
for the shape MM-functional.

Similarly, the influence function of a scale invariant mapping of variance component functionals
only depends on the function $\alpha_C$, see Lemma~2 in Lopuha\"a~\cite{lopuhaa2025}.
Since the characterizations of the variance components MM- and S-functional share the same function $\alpha_C$,
see Corollary~\ref{cor:IF theta1} and Corollary~8.4 in Lopuha\"a \emph{et al}~\cite{lopuhaa-gares-ruizgazen2023},
it follows that the influence functions of any scale invariant mapping of variance component MM- and S-functionals are the same.
Examples are $\btheta/\|\btheta\|$ or $\btheta/|\bV(\btheta)|^{1/k}$, for linear covariance structures,
which represent a direction component of $\btheta$.
\begin{remark}
\label{rem:IF Vgamma and gamma}
From~\eqref{def:theta1} and the fact that $|\bV(\bgamma)|=1$, 
it follows that $\bGamma(\btheta_1)=\bV(\bgamma)$.
This means that $\bV(\bgamma)$ represents the shape component of $\bV(\btheta_1)$,
so that the expression in~\eqref{eq:IF Vgamma}
also coincides with the
influence function of~$\bV(\bgamma)$.
Similarly, the expression for~$\text{\rm IF}(\bs_0;\bgamma,P)$ in Theorem~\ref{th:IF elliptical}
coincides with the influence function of the direction component~$\btheta_1/|\bV(\btheta_1)|^{1/k}$, 
corresponding to the variance components MM-functional $\btheta_1$.
The influence function of the direction component $\btheta_1/\|\btheta_1\|$
can be found in Example~7 in Lopuha\"a~\cite{lopuhaa2025}.
\end{remark}
The results in Corollary~\ref{cor:IF theta1},  and in~\eqref{eq:IF Vtheta1} and~\eqref{eq:IF Vgamma} can be applied
to derive influence functions for the covariance functionals in the multivariate statistical models
of Examples~\ref{ex:LME model}, \ref{ex:multivariate linear regression}, and~\ref{ex:multivariate location-scatter}.
Details are given in the supplemental material~\cite{supplement2025}.
\section{Asymptotic normality}
\label{sec:asymp norm}
To establish asymptotic normality of the MM-estimators, we use the score equations
obtained in Proposition~\ref{prop:score equations}.
We will use score equation~\eqref{eq:Psi=0}, 
with~$P$ replaced by the empirical measure $\mathbb{P}_n$
corresponding to observations
$\bs_1,\ldots,\bs_n$, with $\bs_i=(\by_i,\bX_i)\in\R^k\times\R^{kq}$.
From Proposition~\ref{prop:score equations}, we see that any local minimum $\bxi_n=(\bbeta_{1,n},\bgamma_n)\in \mathfrak{D}$ 
of~$R_n(\bbeta,\bV(\bgamma))$ with $|\bV(\gamma)|=1$, must satisfy
\begin{equation}
\label{eq:M-equation estimator}
\int
\Psi(\mathbf{s},\bxi_n,\sigma_n)\,\dd \mathbb{P}_n(\mathbf{s})
=
\mathbf{0},
\end{equation}
where $\Psi=(\Psi_{\bbeta},\Psi_{\bgamma})$ is defined in~\eqref{def:Psi}
and $\sigma_n$ is a solution of~\eqref{def:initial estimators}.

\subsection{The general case}
\label{subsec:asymp norm general}
Since $\Psi$ also depends on $\sigma$ it can be expected that for general~$P$, 
the limiting behavior of $\bxi_n$ will depend on that of $\sigma_n$,
see Lemma~\ref{lem:asymp relation xi sigma} in the supplemental material~\cite{supplement2025}.
In turn, since $\sigma_n$ is a solution of~\eqref{def:initial estimators}, it can be expected
that in general its limiting behavior depends on that of the initial estimators 
$\bzeta_{0,n}=(\bbeta_{0,n},\btheta_{0,n})$, 
see Lemma~\ref{lem:asymp relation sigma zeta} in the supplemental material~\cite{supplement2025}.

Similar to Section~\ref{sec:IF}, the situation becomes somewhat simpler if the distribution $P$ 
has a point of symmetry.
The asymptotic expansion for $\sigma_n-\sigma(P)$ obtained in Lemma~\ref{lem:asymp relation sigma zeta} 
in~\cite{supplement2025} does simplify,
but details on the limiting distribution of $\bxi_n-\bxi(P)$ can still not be provided without further information on the limiting
behavior of $\bzeta_{0,n}-\bzeta_0(P)$.
The situation differs for $\bbeta_{1,n}-\bbeta_1(P)$, for which the 
limiting distribution is given by the following corollary.
\begin{theorem}
\label{th:asymp norm symmetry}
Suppose $\rho_1$ satisfies (R2) and~(R4), such that $u_1(s)$ is of bounded variation.
Suppose~$\bV$ satisfies~(V1) and $\E\|\bs\|^2<\infty$
Let $\sigma_n$ and $\sigma(P)$ be solutions of~\eqref{def:initial estimators} and~\eqref{def:sigma}, respectively,
and let $\bxi_n$ and~$\bxi(P)$ be local minima of $R_n(\bbeta,\bV(\bgamma))$ and 
$R_P(\bbeta,\bV(\bgamma))$, respectively.
Suppose $(\bxi_n,\sigma_n)\to(\bxi(P),\sigma(P))$, in probability,
and that~$\bbeta_1(P)$ is a point of symmetry of $P$.
Let $\Lambda_{\bbeta}$ be defined by~\eqref{def:Lambda-beta-gamma}
with $\Psi_{\bbeta}$ from~\eqref{def:Psi}
and suppose that $\Lambda_{\bbeta}$ is continuously differentiable with a non-singular derivative 
$\bD_{\bbeta}=\partial\Lambda_{\bbeta}/\partial\bbeta$ at $(\bxi(P),\sigma(P))$.
Then $\sqrt{n}(\bbeta_{1,n}-\bbeta_1(P))$ is asymptotically normal with mean zero and variance
\[
\bD_{\bbeta}^{-1}
\E\left[
\Psi_{\bbeta}(\mathbf{s},\bxi(P),\sigma(P))
\Psi_{\bbeta}(\mathbf{s},\bxi(P),\sigma(P))^T
\right]
\bD_{\bbeta}^{-1}.
\]
\end{theorem}
The limiting distribution of~$\bbeta_{1,n}$ given in Theorem~\ref{th:asymp norm symmetry}
is similar to that of the regression MM-estimator~$\bbeta_1$ in Lopuha\"a~\cite{lopuhaa2023}
defined with loss function $\rho_1$ and initial covariance estimator $\bV_{0,n}$,
and they both coincide when $\bV_{0,n}\to\sigma^2(P)\bV(\bgamma(P))$, in probability.

If one is allowed to interchange integration and differentiation in $\bD_{\bbeta}$, 
then for the multiple linear regression model,
the limiting distribution of $\bbeta_{1,n}$ established in Theorem~\ref{th:asymp norm symmetry}
coincides with that of the regression MM-estimator considered in 
Yohai~\cite{yohai1987}.
For linear mixed effects models, multivariate linear regression models, 
or multivariate location-scatter models, the limiting distribution of MM-estimators is only available 
at distributions with an elliptically contoured density.
This situation is discussed in the next section for model~\eqref{def:model}.

\subsection{Elliptical contoured densities}
\label{subsec:asymp norm elliptical}
Consider the special case that $P$ satisfies~(E).
As before, when determining the limiting normal distribution of the MM-estimators,
we cannot use affine equivariance and restrict ourselves to the case~$(\mathbf{0},\bI_k)$.
Instead, we use some of the results obtained in Section~\ref{subsec:IF elliptical}
to establish the limiting normal distributions of the MM-estimators 
$\bxi_n=(\bbeta_{1,n},\bgamma_n)$,
$\btheta_{1,n}$, and $\bV(\btheta_{1,n})$.
Let 
\begin{equation}
\label{def:sigma1}
\sigma_1
=
\frac{k\E_{\mathbf{0},\bI_k}\left[\rho_1'(c_\sigma\|\bz\|)^2(c_\sigma\|\bz\|)^2\right]}{(k+2)\gamma_1^2},
\end{equation}
where $c_\sigma=|\bSigma|^{1/(2k)}/\sigma(P)$
and $\gamma_1$ is defined in~\eqref{def:alpha1-gamma1}.
The limiting distribution of MM-estimators~$\bbeta_{1,n}$ and $\bgamma_n$ 
at distributions $P$, such that~$\by\mid\bX$ has an elliptical contoured density, are now given by the following theorem.
\begin{theorem}
\label{th:asymp norm elliptical}
Suppose $P$ satisfies~(E) for some $(\bbeta^*,\btheta^*)\in\R^q\times\bTheta$
and~$\E\|\bs\|^2<\infty$.
Suppose $\rho_1$ satisfies (R1)-(R5), such that $u_1(s)$ is of bounded variation,
and suppose~$\bV$ satisfies~(V5) and has a linear structure~\eqref{def:V linear}.
Let $\sigma_n$ and $\sigma(P)$ be solutions of~\eqref{def:initial estimators} and~\eqref{def:sigma}, respectively,
and suppose that $\sigma_n-\sigma(P)=O_P(1/\sqrt{n})$.
Let $\bxi_n=(\bbeta_{1,n},\bgamma_n)$ and~$\bxi(P)=(\bbeta_1(P),\bgamma(P))$ be local minima of $R_n(\bbeta,\bV(\bgamma))$ and 
$R_P(\bbeta,\bV(\bgamma))$, respectively, 
and suppose that $\bxi_n\to\bxi(P)$, in probability.
Suppose that $\bbeta_1(P)=\bbeta^*$ and that $\bV(\bgamma(P))=\bSigma/|\bSigma|^{1/k}$.
Let $\alpha_1$ and $\gamma_1$ be defined in~\eqref{def:alpha1-gamma1}
and suppose that $\alpha_1\ne0$ and $\gamma_1>0$.
Suppose $\bX$ has full rank with probability one
and $\bL$, as defined in~\eqref{def:L}, has full rank.
Then~$\sqrt{n}(\bbeta_{1,n}-\bbeta^*)$ and $\sqrt{n}(\bgamma_n-\bgamma(P))$ are 
asymptotically independent.

Furthermore, $\sqrt{n}(\bbeta_{1,n}-\bbeta^*)$ is asymptotically normal with mean zero and variance
\[
\frac{\E_{\mathbf{0},\bI_k}\left[\rho_1'(c_\sigma\|\bz\|)^2\right]}{c_\sigma^2k\alpha_1^2}
\left(
\mathbb{E}\left[\mathbf{X}^T\bSigma^{-1}\mathbf{X}\right]
\right)^{-1},
\]
where $c_\sigma=|\bSigma|^{1/(2k)}/\sigma(P)$,
and $\sqrt{n}(\bgamma_n-\bgamma(P))$ is asymptotically normal with mean zero and variance
\[
\frac{2\sigma_1}{|\bSigma|^{2/k}}
\left\{
\Big(\bL^T\left(\bSigma^{-1}\otimes\bSigma^{-1}\right)\bL\Big)^{-1}
-
\frac{1}{k}\btheta^*(\btheta^*)^T
\right\},
\]
where $\sigma_1$ is defined in~\eqref{def:sigma1}.
\end{theorem}
When $c_\sigma=1$, 
then similar to Theorem~\ref{th:asymp norm symmetry}, we find that the limiting distribution of 
$\bbeta_{1,n}$ coincides with that of the regression MM-estimator considered in Lopuha\"a~\cite{lopuhaa2023},
defined with loss function~$\rho_1$ and an initial covariance estimator $\bV_{0,n}$ that is consistent for $\bSigma$.
Moreover, it also coincides with the limiting distribution of the regression S-estimator defined with loss function~$\rho_1$,
see Corollary~9.2 in Lopuha\"a \emph{et al}~\cite{lopuhaa-gares-ruizgazen2023}.
This confirms the claim made by Salibi\'an-Barrera \emph{et al}~\cite{SalibianBarrera-VanAelst-Willems2006}
about the location MM-estimator in the model of Example~\ref{ex:multivariate location-scatter}.

For the linear mixed effects model, the limiting distribution of $\bbeta_{1,n}$ obtained in Theorem~\ref{th:asymp norm elliptical}
extends Theorem~1 in Copt and Heritier~\cite{copt&heritier2007}, which is restricted to $\bX_i=\bX$.
For the multivariate linear regression model, the limiting distribution of $\bbeta_{1,n}$ in 
Theorem~\ref{th:asymp norm elliptical} coincides with the one found for the regression MM-estimator in 
Kudraszow and Maronna~\cite{kudraszow-maronna2011}. 
This also applies to the location MM-estimator in the multivariate location-scatter model, 
since this model is a special case of the multivariate linear regression model.
Furthermore, there is a connection with CM-estimators considered in
Kent and Tyler~\cite{kent&tyler1996},
whose limiting distribution depends on a parameter~$\lambda_0$.
The limiting distribution of $\sqrt{n}(\bbeta_{1,n}-\bbeta^*)$ obtained in Theorem~\ref{th:asymp norm elliptical} 
is similar to that of the location CM-estimator, see~(7.9) in Kent and Tyler~\cite{kent&tyler1996},
and they both coincide when $c_\sigma=1/\sqrt{\lambda_0}$.

Let
\begin{equation}
\label{def:sigma3}
\sigma_3=
\frac{4\E_{\mathbf{0},\bI_k}
\left[
\left(\rho_0(c_\sigma\|\bz\|)-b_0\right)^2
\right]}{
\big(\E_{\mathbf{0},\bI_k}
\left[
\rho_0'(c_\sigma\|\bz\|)c_\sigma\|\bz\|
\right]
\big)^2},
\end{equation}
where $c_\sigma=|\bSigma|^{1/(2k)}/\sigma(P)$.
It remains to determine the limiting distribution of the variance components MM-estimator~$\btheta_{1,n}$.
This is given in the next corollary.
\begin{corollary}
\label{cor:asymp norm theta1}
Suppose that the conditions hold of Theorem~\ref{th:asymp norm elliptical}
and suppose that~$\rho_0$ satisfies~(R1), (R2) and~(R4).
Let $\bzeta_{0,n}=(\bbeta_{0,n},\btheta_{0,n})$ be the pair of initial estimators
and let $\bzeta_0=(\bbeta_0,\btheta_0)$ be the corresponding functional.
Suppose that $(\bbeta_0(P),\btheta_0(P))=(\bbeta^*,\btheta^*)$ and that $\bzeta_{0,n}-\bzeta_0(P)=O_P(1/\sqrt{n})$.
Let $\sigma_n$ and $\sigma(P)$ be solutions of~\eqref{def:initial estimators} and~\eqref{def:sigma}, respectively,
and suppose that $\sigma_n\to\sigma(P)$, in probability.
Let $\btheta_{1,n}$ and $\btheta_1(P)$ be solutions of~\eqref{eq:update theta} and~\eqref{def:theta1}, respectively,
and suppose that $\E_{\mathbf{0},\bI_k}[\rho_0'(c_\sigma\|\bz\|)c_\sigma\|\bz\|]>0$,
where~$c_\sigma=|\bSigma|^{1/(2k)}/\sigma(P)$.
Then $\sqrt{n}(\btheta_{1,n}-\btheta_1(P))$ is asymptotically normal with mean zero and variance
\[
\frac{2\sigma_1}{c_\sigma^2}
\Big(\bL^T\left(\bSigma^{-1}\otimes\bSigma^{-1}\right)\bL\Big)^{-1}
+
\left(
-\frac{2\sigma_1}{kc_\sigma^2}
+
\sigma_3
\right)
\btheta^*(\btheta^*)^T,
\]
where $\sigma_1$ and $\sigma_3$ are defined in~\eqref{def:sigma1} and~\eqref{def:sigma3}.
\end{corollary}
For linearly structured $\bV$, one has $\vc(\bV(\btheta_1(P)))=\bL\btheta_1(P)$ and $\vc(\bSigma)=\bL\btheta^*$.
Hence, application of the delta-method yields that
$\sqrt{n}(\vc(\bV(\btheta_{1,n}))-\vc(\bSigma))$ is asymptotically normal with 
mean zero and variance
\begin{equation}
\label{eq:asymp var theta1}
\frac{2\sigma_1}{c_\sigma^2}
\bL\Big(\bL^T\left(\bSigma^{-1}\otimes\bSigma^{-1}\right)\bL\Big)^{-1}\bL^T
+
\left(
-\frac{2\sigma_1}{kc_\sigma^2}
+
\sigma_3
\right)
\vc(\bSigma)
\vc(\bSigma)^T.
\end{equation}
The expressions for the limiting variances of the covariance
MM-estimators in Corollary~\ref{cor:asymp norm theta1} and~\eqref{eq:asymp var theta1} are characterized by
two scalars $\sigma_1/c_\sigma^2$ and $\sigma_2=-2\sigma_1/(kc_\sigma^2)+\sigma_3$.
When $c_\sigma=1$,
this matches with the characterization of general structured covariance estimators,
see Corollary~2 in Lopuha\"a~\cite{lopuhaa2025}.
Such a characterization was already observed by Tyler~\cite{tyler1982} for 
affine equivariant covariance estimators in the multivariate location-scatter model.
The constant~$\sigma_1$ (with $c_\sigma=1$) coincides with the one for the covariance S-estimator defined
with loss function~$\rho_1$.
The constant $\sigma_2$ (with $c_\sigma=1$) has the same structure as the one for covariance S-estimators,
but the first term $-2\sigma_1/k$ is built from $\rho_1$,
whereas the second term $\sigma_3$ is built from~$\rho_0$.
As expected, when $\rho_0=\rho_1=\rho$, the above characterization coincides with the 
one for the covariance S-estimator defined with loss function~$\rho$.

Note that the limiting variance of scale invariant mappings $H(\bC_n)$ of a covariance estimator~$\bC_n$,
only depends on the scalar~$\sigma_1$, see~(8.2) in Kent and Tyler~\cite{kent&tyler1996} for
affine equivariant covariance estimators or Theorem~3 in Lopuha\"a~\cite{lopuhaa2025}
for estimators of a linearly structured covariance.
Because the characterizations of the limiting variances of covariance MM- and S-estimators have
the same scalar $\sigma_1$, it follows that the limiting distributions of any scale invariant mapping 
of covariance MM- and S-estimators are the same.
A typical example is the shape component $\bGamma(\btheta_{1,n})$ of the covariance MM-estimator,
where $\bGamma$ is defined in~\eqref{def:Gamma}.
Theorem~2 and Example~4 in Lopuha\"a~\cite{lopuhaa2025}, 
together with~\eqref{eq:asymp var theta1},
yield that $\sqrt{n}(\vc(\bGamma(\btheta_{1,n}))-\vc(\bGamma(\btheta_1(P))))$
is asymptotically normal with mean zero and variance
\begin{equation}
\label{eq:asymp var shape}
\frac{2\sigma_1}{c_\sigma^2|\bSigma|^{2/k}}
\left\{
\bL\Big(\bL^T\left(\bSigma^{-1}\otimes\bSigma^{-1}\right)\bL\Big)^{-1}\bL^T
-
\frac{1}{k}
\vc(\bSigma)
\vc(\bSigma)^T
\right\}.
\end{equation}
When $c_\sigma=1$, this coincides with the limiting distribution
of the shape S-estimator defined with $\rho_1$,
see Examples~3 and~4 in Lopuha\"a~\cite{lopuhaa2025}.
For the location-scatter model in Example~\ref{ex:multivariate location-scatter},
this confirms the claim made in Salibi\'an-Barrera \emph{et al}~\cite{SalibianBarrera-VanAelst-Willems2006}
for the shape MM-estimator.

Similarly, the limiting distribution of a scale invariant mapping of variance component estimators
only depends on the scalar $\sigma_1$, see Theorem~2 in Lopuha\"a~\cite{lopuhaa2025}.
Since the characterizations of the limiting distribution of the variance components MM- and S-estimators share the same scalar $\sigma_1$,
see Corollary~\ref{cor:asymp norm theta1} and Corollary~9.2 in Lopuha\"a \emph{et al}~\cite{lopuhaa-gares-ruizgazen2023},
it follows that the limiting distributions of any scale invariant mapping of variance component MM- and S-estimators are the same.
Examples are direction components, such as $\btheta/\|\btheta\|$ or~$\btheta/|\bV(\btheta)|^{1/k}$, for linear covariance structures.

\begin{remark}
\label{rem:asymp distr Vgamma and gamma}
From~\eqref{eq:update theta} and the fact that $|\bV(\bgamma_n)|=1$, 
it follows that $\bGamma(\btheta_{1,n})=\bV(\bgamma_n)$.
This means that $\bV(\bgamma_n)$ represents the shape component of $\bV(\btheta_{1,n})$,
so that the limiting distribution of the shape component of $\bV(\btheta_{1,n})$
is the same as that of~$\bV(\bgamma_n)$.
Similarly, the limiting distribution of $\sqrt{n}(\bgamma_n-\bgamma(P))$ established
in Theorem~\ref{th:asymp norm elliptical}, 
coincides with that of the direction component~$\btheta/|\bV(\btheta)|^{1/k}$, 
corresponding to the variance components MM-estimator~$\btheta_{1,n}$.
The limiting distribution of the direction component $\btheta_{1,n}/\|\btheta_{1,n}\|$
can be found in Example~5 in Lopuha\"a~\cite{lopuhaa2025}.
\end{remark}
The results in Corollary~\ref{cor:asymp norm theta1},  and in~\eqref{eq:asymp var theta1} 
and~\eqref{eq:asymp var shape} can be applied
to derive the limiting distributions for the covariance estimators in the multivariate statistical models
of Examples~\ref{ex:LME model}, \ref{ex:multivariate linear regression}, and~\ref{ex:multivariate location-scatter}.
Details are given in the supplemental material~\cite{supplement2025}.

\section{Application}
\label{sec:application}
We apply our results to MM-estimators and MM-functionals to linear model~\eqref{def:model}.
Consider a distribution $P$ that satisfies~(E), where $\bV$ is has linear structure~\eqref{def:V linear}.
The loss functions $\rho_0$ and~$\rho_1$ are constructed from Tukey's biweight, as defined in~\eqref{def:biweight}, by taking
$\rho_j(s)=\rho_{\text{B}}(s;c_j)$, for $j=0,1$, such that $0<c_0\leq c_1<\infty$.
As initial estimators we use the S-estimators $(\bbeta_{0,n},\btheta_{0,n})$,
defined by minimizing $|\bV(\btheta)|$, subject to
\[
\frac{1}{n}
\sum_{i=1}^{n}
\rho_0\left(
d(\by_i,\bX_i\bbeta,\bV(\btheta))
\right)
=
b_0,
\]
where $d$ is defined in~\eqref{def:Mahalanobis distance} and $b_0=\E_{\mathbf{0},\bI_k}[\rho_0(\|\bz\|)]$,
and where the minimum is taken over all $\bbeta\in\R^{q}$ and $\btheta\in\bTheta\subset\R^l$,
such that $\bV(\btheta)\in\text{PDS}(k)$.
The cut-off $c_0$ is chosen such that $b_0/(c_0^2/6)=0.5$, so that the initial S-estimator has (asymptotic) breakdown point 0.50.
This means that according to Theorem~\ref{th:asymp norm elliptical},
the scalar $\lambda=\E_{\mathbf{0},\bI_k}[\rho_1'(\|\bz\|)^2]/(k\alpha_1^2)$,
where $\alpha_1$ is defined in~\eqref{def:alpha1-gamma1} with $c_\sigma=1$, 
represents the asymptotic efficiency of the regression MM-estimator relative to
the least squares estimator (for which $\lambda=1$).
Similarly, according to (part two of) Theorem~\ref{th:asymp norm elliptical} and~\eqref{eq:asymp var shape}, the scalar $\sigma_1=k\E_{\mathbf{0},\bI_k}\left[\rho_1'(\|\bz\|)^2(\|\bz\|)^2\right]/((k+2)\gamma_1^2)$,
where $\gamma_1$ is defined in~\eqref{def:alpha1-gamma1} with $c_\sigma=1$,
represents the asymptotic relative efficiency of both the MM-estimator of shape
as well as the MM-estimator for the direction of the variance components, relative
to the least squares estimators of shape and direction, respectively (for which $\sigma_1=1$).
Hence, the scalars $\lambda$ and $\sigma_1$ only depend on the function~$\rho_1$.
By keeping $c_0$ fixed, the breakdown point of the MM-estimators remains unaffected,
and by varying $c_1\geq c_0$ we will investigate how the scalars $\lambda$ and $\sigma_1$ for the asymptotic relative efficiency 
will vary.

We further investigate how at the same time the gross error sensitivity (GES) of the corresponding MM-functionals will vary.
For simplicity we only consider perturbations in~$\by$ and leave~$\bX$ unchanged.
According to Theorem~\ref{th:IF elliptical}, it can be seen that for any norm,
$\|\text{IF}(\bs;\bbeta_1,P)\|$ is proportional to $\alpha_1^{-1}\left|\rho_1'(d(\by))\right|$,
where $\alpha_1$ is defined in~\eqref{def:alpha1-gamma1} with $c_\sigma=1$ and 
$d(\by)^2=(\by-\bX\bbeta^*)^T\bSigma^{-1}(\by-\bX\bbeta^*)$.
Therefore, we propose the scalar
\[
G_1=\frac{1}{\alpha_1}\sup_{s>0}\left|\rho_1'(s)\right|,
\]
as an index for the GES of regression MM-functionals.
This coincides with the GES index for location CM-functionals in Kent and Tyler~\cite{kent&tyler1996}.
Similarly, from (part two of) Theorem~\ref{th:IF elliptical} and~\eqref{eq:IF Vgamma} it follows 
that $\|\text{IF}(\by;\bgamma,P)\|$ and $\|\text{IF}(\by;\bV(\bgamma),P)\|$ are proportional 
to~$\gamma_1^{-1}\left|\rho_1'(d(\by))d(\by)\right|$,
where $\gamma_1$ is defined in~\eqref{def:alpha1-gamma1} with $c_\sigma=1$.
We propose scalar
\[
G_2=\frac{k}{(k+2)\gamma_1}\sup_{s>0}|\rho_1'(s)s|,
\]
as an index for the GES of shape and direction functionals.
In this way, $G_2$ coincides with the GES index for CM-functionals of shape in Kent and Tyler~\cite{kent&tyler1996}.

We investigate how the asymptotic efficiency at the multivariate normal of MM-estimators, 
and the GES of the corresponding MM-functionals behave as we vary the cut-off constant~$c_1\geq c_0$.
\begin{figure}[t]
  \centering
  \includegraphics[width=\textwidth]{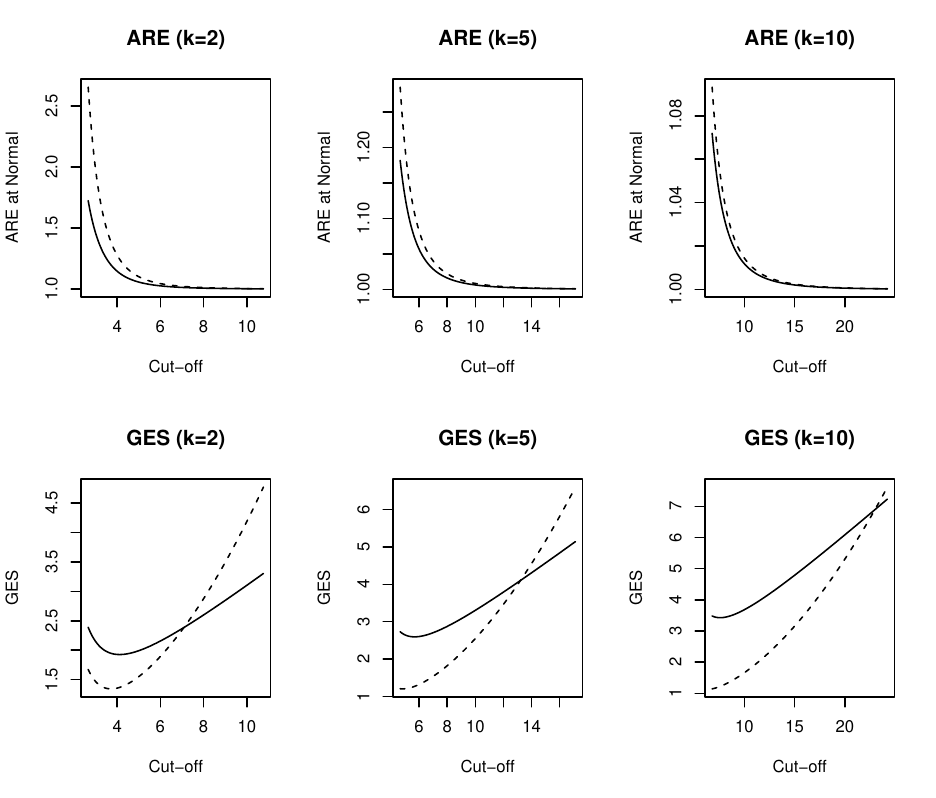}
  \caption{Asymptotic efficiencies at the multivariate normal distribution (top row) 
  for the MM-estimators of regression (solid) and for shape and direction (dashed) 
  and the GES (bottom row) of the corresponding MM-functionals for dimensions $k=2,5,10$.}
  \label{fig:ARE-GES-Normal}
\end{figure}
In Figure~\ref{fig:ARE-GES-Normal}, on the top row we plotted the indices $\lambda$ (solid lines) 
and $\sigma_1$ (dashed lines) together as a function of $c_1\geq c_0$ in dimensions $k=2,5$ and 10.
In dimension $k=2$, the asymptotic efficiencies $\lambda=1.725$ and $\sigma_1=2.656$ of the MM-estimators at cut-off $c_1=c_0=2.661$
are the same as that of the initial 50\% breakdown S-estimators.
When increasing the cut-off $c_1$, one can both gain efficiency and lower the GES.
For example, the GES index for the shape and direction MM-functional attains its minimal value 
$G_2=1.344$ at $c_1=3.724$.
For this cut-off value the GES index for the regression MM-functional is $G_1=1.947$
and the asymptotic efficiencies are $\lambda=1.197$ and $\sigma_1=1.383$.
Similarly, the GES index for the regression MM-functional attains its minimal value 
$G_1=1.927$ at $c_1=4.113$.
This would yield $G_2=1.368$, $\lambda=1.131$ and $\sigma_1=1.246$.

In dimension $k=5$, the asymptotic efficiencies at cut-off $c_1=c_0=4.652$ are $\lambda=1.182$ and $\sigma_1=1.285$.
The GES index for the regression MM-functional attains its minimal value $G_1=2.595$ at $c_1=5.675$.
For this cut-off value the GES index $G_2=1.270$ and the asymptotic efficiencies are
$\lambda=1.073$ and $\sigma_1=1.107$.
The index for the shape and direction MM-functional attains its minimal value 
$G_2=1.204$ at $c_1=c_0=4.652$.

In dimension $k=10$, the asymptotic efficiencies at cut-off $c_1=c_0=6.776$ are $\lambda=1.072$ and $\sigma_1=1.093$.
The GES index for the regression MM-functional attains its minimal value $G_1=3.426$ at $c_1=7.580$.
For this cut-off value the GES index $G_2=1.270$ and the asymptotic efficiencies are
$\lambda=1.042$ and $\sigma_1=1.053$.
The GES index for the shape and direction MM-functional attains its minimal value 
$G_2=1.142$ at $c_1=c_0=6.776$.

\begin{figure}[t]
  \centering
  \includegraphics[width=\textwidth]{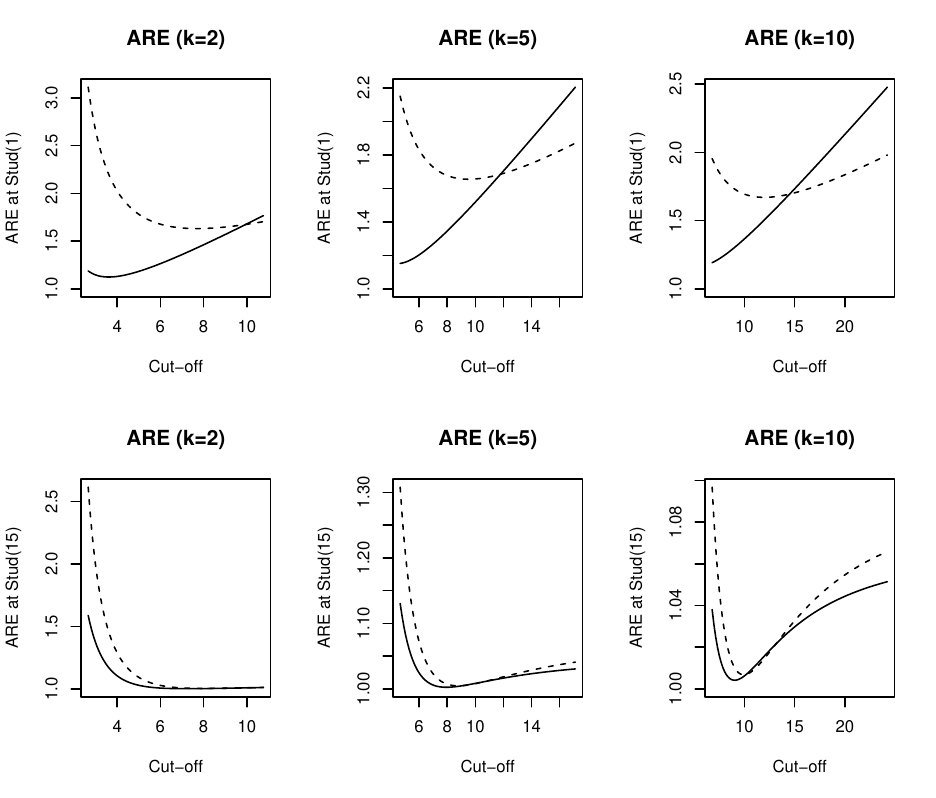}
  \caption{Asymptotic efficiencies at the multivariate Student distribution with degrees of freedom $\nu=1$ (top row) 
  and $\nu=15$ (bottom row) for the MM-estimators of regression (solid) and for shape and direction (dashed) 
  for dimensions $k=2,5,10$.}
  \label{fig:ARE-Student}
\end{figure}

In the top row of Figure~\ref{fig:ARE-GES-Normal} it can be seen that the asymptotic efficiencies become closer to one when the dimension is large.
This is a well known phenomenon observed when the efficiency is computed under a multivariate normal setting.
As a comparison, we have investigated whether this behavior is observed in a neighborhood of the multivariate normal.
We have computed asymptotic efficiencies relative to the maximum likelihood estimator at the $k$-variate Student distribution
with degrees of freedom $\nu=1$ and $\nu=15$.
The scalars $\lambda$ and~$\sigma_1$ for the ML estimator at the Student($\nu$) distribution are given by
\[
\lambda^{\text{ML}}
=
\frac{k\E_{\nu}[w_1(\|\bz\|)^2\|\bz\|^2]}{
\Big(k\E_{\nu}[w_1(\|\bz\|)+\E_{\nu}[w_1'(\|\bz\|)\|\bz\|]
\Big)^2};
\quad
\sigma_1^{\text{ML}}
=
\frac{k(k+2)\E_{\nu}[w_1(\|\bz\|)^2\|\bz\|^4]}{
\Big(\E_{\nu}[w_1'(\|\bz\|)\|\bz\|^3+k(k+2)\Big)^2},
\]
where $w_1(s)=(\nu+k)/(k+s^2)$.

The asymptotic efficiencies relative to the ML estimator at the $k$-variate $\text{Student}(\nu)$ distribution with
$\nu\in\{1,15\}$ are visible in Figure~\ref{fig:ARE-Student}.
The graphs in the top row correspond to $\nu=1$ and are quite different from the ones in the top row in Figure~\ref{fig:ARE-GES-Normal}.
Moreover, the behavior of the MM-regression estimator (solid lines) differs from that of the MM-estimators of shape and direction (dashed lines).
The best efficiencies for the regression MM-estimator $\lambda=1.124, 1.152, 1.192$, for $k=2,5,10$, 
are obtained for values of $c_1$ very close or equal to~$c_0$,
whereas the efficiency for the MM-estimators of shape and direction can be improved for larger values of $c_1$
and are equal to $\sigma_1=1.631,1.655,1.671$ at $c_1=7.656,9.561,11.893$, for $k=2,5,10$.
As expected both MM-estimators with large values for $c_1$ have poor efficiencies, because they tend to behave similar
to the least squares estimators. 

For the Student distribution with $\nu=15$ degrees of freedom, the behavior of the efficiency is more or less 
in between the ones at the multivariate normal and the Student distribution with $\nu=1$ degrees of freedom.
The graphs in the bottom row of Figure~\ref{fig:ARE-Student} are more similar
to the ones in the top row of Figure~\ref{fig:ARE-GES-Normal}, although in higher dimensions the efficiencies get worse.
The best efficiencies for the regression MM-estimator $\lambda=1.001, 1.002, 1.004$, for $k=2,5,10$, 
are obtained at $c_1=7.246,7.925,9.070$,
and the best efficiencies for the MM-estimators of shape and direction 
are equal to $\sigma_1=1.003,1.005,1.008$, for $k=2,5,10$,
obtained at $c_1=8.065,8.806,9.952$.

\bibliographystyle{abbrv}
\bibliography{MMLME}       

\newpage

\setcounter{page}{1}

\section{Supplemental Material}

\subsection{Proofs for Section~\ref{sec:existence}}
For any $k\times k$ matrix $\bA$, let $\lambda_k(\bA)\leq\cdots\leq\lambda_1(\bA)$ denote the eigenvalues of $\bA$.

\paragraph*{Proof of Lemma~\ref{lem:existence sigma}}
\begin{proof}
From (R1)-(R2) we have that $\rho_0$ is bounded and continuous at zero.
Hence, by dominated convergence, it follows that
\[
\lim_{\sigma\to\infty}
\int
\rho_0
\left(
\frac{\displaystyle\sqrt{(\by-\bX\bbeta_0(P))^T\bGamma(\btheta_0(P))^{-1}(\by-\bX\bbeta_0(P))}}{\sigma}
\right)
\,
\text{d}P(\bs)
=0.
\]
Similarly, together with (C0) and (R3), we find that
\[
\begin{split}
\lim_{\sigma\downarrow0}
\int
\rho_0
\left(
\frac{\displaystyle\sqrt{(\by-\bX\bbeta_0(P))^T\bGamma(\btheta_0(P))^{-1}(\by-\bX\bbeta_0(P))}}{\sigma}
\right)
\,
\text{d}P(\bs)\\
=
(\sup\rho_0)(1-P(E_0))
>
b_0.
\end{split}
\]
Since $\rho_0$ is continuous and $0<b_0<\sup\rho_0$, we conclude that there exists a solution $\sigma(P)>0$ to~\eqref{def:sigma}.
Because $\rho_0$ is strictly increasing on $[0,c_0]$, together with (C0), it follows that $\sigma(P)$ is unique.
\end{proof}

\paragraph*{Proof of Theorem~\ref{th:existence}}
\begin{proof}
For $(\bbeta,\bgamma)\in\R^k\times\R^l$,
define cylinder
\begin{equation}
\label{def:cylinder gamma}
\mathcal{C}(\bbeta,\bgamma,c)=
\left\{
(\by,\bX)\in\R^k\times\R^{kq}:
(\by-\bX\bbeta)^T
\bV(\bgamma)^{-1}
(\by-\bX\bbeta)
\leq
c^2
\right\}.
\end{equation}
According to (V2), there exists $\bgamma_0\in\bTheta$, such that
\[
\bV(\bgamma_0)=\frac{\bV(\btheta_0(P))}{|\bV(\btheta_0(P))|^{1/k}}=\bGamma(\btheta_0(P)),
\]
and clearly $|\bV(\bgamma_0)|=1$.
If $(\bbeta,\bgamma)\in \mathfrak{D}$ minimizes~$R_P(\bbeta,\bV(\bgamma))$, then together with~\eqref{eq:ineq rho functions}
and~\eqref{def:sigma},
it must satisfy
\begin{equation}
\label{eq:lower bound prob cylinder}
\begin{split}
&
P\left(
\mathcal{C}\left(\bbeta,\bgamma,c_1\sigma(P)\right)
\right)\\
&\geq
1-\frac{1}{\sup\rho_1}
\int
\rho_1\left(
\frac{\displaystyle\sqrt{(\by-\bX\bbeta)^T\bV(\bgamma)^{-1}(\by-\bX\bbeta)}}{\sigma(P)}
\right)
\text{d}P(\bs)\\
&\geq
1-\frac{1}{\sup\rho_1}
\int
\rho_1\left(
\frac{\displaystyle\sqrt{(\by-\bX\bbeta_0(P))^T\bV(\bgamma_0)^{-1}(\by-\bX\bbeta_0(P))}}{\sigma(P)}
\right)
\text{d}P(\bs)\\
&\geq
1-\frac{1}{\sup\rho_0}
\int
\rho_0
\left(
\frac{\displaystyle\sqrt{(\by-\bX\bbeta_0(P))^T\bGamma_0(P)^{-1}(\by-\bX\bbeta_0(P))}}{\sigma(P)}
\right)
\text{d}P(\bs)\\
&=
1-\frac{b_0}{\sup\rho_0}
=
1-r_0
\geq
\epsilon.
\end{split}
\end{equation}
Since $P$ satisfies $(\text{C2}_\epsilon)$ for $1-r_0\geq \epsilon$, 
from Lemma~4.1(i) in Lopuha\"a \textit{et al}~\cite{lopuhaa-gares-ruizgazen2023}, 
it follows that there exist $a_1>0$, only depending on $c_1$ and $(\mathrm{C2}_{\epsilon})$, 
such that~$\lambda_k(\sigma^2(P)\bV(\bgamma))\geq a_1$, so that $\lambda_k(\bV(\bgamma))\geq a_1/\sigma^2(P)>0$.
Since $|\bV(\bgamma)|=1$, it immediately follows that
\[
\lambda_1(\bV(\bgamma))
\leq
a_2:=\left(\frac{\sigma^2(P)}{a_1}\right)^{k-1}<\infty.
\]
Let $K\subset\R^{k}\times \mathcal{X}$ be a compact set, such that $P(K)\geq r_0+\epsilon$, which exists according to condition~$(\mathrm{C1}_{\epsilon})$.
From Lemma~4.1(iii) in Lopuha\"a \textit{et al}~\cite{lopuhaa-gares-ruizgazen2023}, 
it follows that $\|\bbeta\|\leq M<\infty$, for some $M>0$ that only depends on $c_1$, $a_2$, $\sigma(P)$, $K$ and $(\text{C2}_\epsilon)$.
We conclude that~$\bbeta$ is in a compact subset of $\R^q$ and $\bV(\bgamma)$ 
is in a compact set $B\subset \R^{k\times k}$.
By identifiability, the mapping $\bgamma\mapsto \bV(\gamma)$ is one-to-one, 
so we can restrict~$\bgamma$ to the pre-image~$\bV^{-1}(B)$.
Then with conditions (V1) and (V3), it follows that also $\bV^{-1}(B)$ is a compact set in~$\bTheta$.
We conclude that for minimizing $R_P(\bbeta,\bV(\bgamma))$
we can restrict ourselves to a compact set~$B'\subset \mathfrak{D}$.

Because $\rho_1$ is continuous, together with condition~(V1) and dominated convergence,
it follows that $R_P(\bbeta,\bV(\bgamma))$ is a continuous function of
$(\bbeta,\bgamma)$, so that it must attain a minimum on~$B'$.
Hence, there exists a pair $(\bbeta_1(P),\bgamma(P))\in \mathfrak{D}$ that minimizes~$R_P(\bbeta,\bV(\bgamma))$.
Finally, condition~(V2) immediately yields that there exists a $\btheta_1(P)$ that solves~\eqref{def:theta1}
and by identifiability it follows that $\btheta_1(P)$ is unique.
\end{proof}

\paragraph*{Proof of Corollary~\ref{cor:existence estimator structured}}
\begin{proof}
Let $\mathbb{P}_n$ be the empirical measure corresponding to the collection $\mathcal{S}_n$.
Then~$\mathbb{P}_n$ satisfies~$(\text{C1}_\epsilon)$ for any $0<\epsilon\leq 1-r_0$ and
satisfies $(\text{C2}_\epsilon)$, for $\epsilon=(\kappa(\mathcal{S}_n)+1)/n$.
Clearly, $0<(\kappa(\mathcal{S}_n)+1)/n\leq 1-r_0$, where $r_0=b_0/\sup\rho_0$.
Furthermore, since $(\bbeta_{0,n},\btheta_{0,n})$ satisfies~\eqref{eq:cond beta0 theta0},
it follows that~$\sigma_n=\sigma(\mathbb{P}_n)$ is a solution of~\eqref{def:sigma}, with $P=\mathbb{P}_n$.
Hence, according to Theorem~\ref{th:existence} there exists 
a pair $(\bbeta_1(\mathbb{P}_n),\bgamma(\mathbb{P}_n))$ that minimizes~$R_P(\bbeta,\bV(\bgamma))$, with $P=\mathbb{P}_n$,
and a vector $\btheta_1(\mathbb{P}_n)$ that is the unique solution of~\eqref{def:theta1}, with $P=\mathbb{P}_n$.
But this is equivalent with saying that there exists a pair $(\bbeta_{1,n},\bgamma_n)\in \mathfrak{D}$ that minimizes~$R_n(\bbeta,\bV(\bgamma))$
and a vector~$\btheta_{1,n}$ that is the unique solution of~\eqref{eq:update theta}.
\end{proof}

\paragraph*{Proof of Corollary~\ref{cor:existence weak convergence}}
\begin{proof}
Because $P$ satisfies condition (C3), according to Ranga Rao~\cite[Theorem 4.2]{rangarao1962} we have
\begin{equation}
\label{eq:ranga rao structured}
\sup_{C\in \mathfrak{C}}
\left|
P_t(C)-P(C)
\right|
\to
0,
\quad
\text{as }t\to\infty.
\end{equation}
Consider the set $E_{t,0}=
\left\{
(\by,\bX)\in\R^k\times\R^{kq}:
\|\by-\bX\bbeta_0(P_t)\|=0
\right\}$.
Then $E_{t,0}\in \mathfrak{C}$,
so that  $P_t(E_{t,0})-P(E_{t,0})\to0$, as $t\to\infty$.
Since $\bbeta_0(P_t)\to\bbeta_0(P)$, as $t\to\infty$,  it follows that $P(E_{t,0})\to P(E_0)$, 
which implies that
\[
P_t(E_{t,0})
=
P_t(E_{t,0})-P(E_{t,0})+P(E_{t,0})\to
P(E_0)
<
1-\frac{b_0}{\sup\rho_0}.
\]
Therefore, $P_t$ satisfies~(C0), for $t$ sufficiently large.
According to Lemma~\ref{lem:existence sigma}, a solution~$\sigma(P_t)$ 
of~\eqref{def:sigma} with $P=P_t$ exists and is unique.
This proves part(i).

The argument that minimizing
$R_{P_t}(\bbeta,\bV(\bgamma))$ over $(\bbeta,\bgamma)\in\mathfrak{D}$
has at least one solution, is similar to the proof of Corollary~4.4 in Lopuha\"a \textit{et al}~\cite{lopuhaa-gares-ruizgazen2023}.
First note there exists $0<\eta<\epsilon'-\epsilon$.
Because strips $H(\balpha,\ell,\delta)\in \mathfrak{C}$,
property~\eqref{eq:ranga rao structured} implies that every strip with $P_t(H(\balpha,\ell,\delta))\geq \epsilon+\eta$
must also satisfy $P(H(\balpha,\ell,\delta))\geq \epsilon$,
for $t$ sufficiently large.
Together with the fact that~$P$ satisfies~$(\text{C2}_{\epsilon})$,
this means that, for $t$ sufficiently large,
\[
\inf
\left\{
\delta:P_t(H(\balpha,\ell,\delta))\geq \epsilon+\eta
\right\}
\geq
\inf
\left\{
\delta:P(H(\balpha,\ell,\delta))\geq \epsilon
\right\}>0.
\]
It follows that, for $t$ sufficiently large, $P_t$ satisfies condition $(\text{C2}_{\epsilon+\eta})$.
Next, consider the compact set $K$ from ($\text{C1}_{\epsilon'}$).
Without loss of generality we may assume that it belongs to~$\mathfrak{C}$.
Therefore,
as $P(K)\geq r_0+\epsilon'$, for $t$ sufficiently large $P_t(K)\geq r_0+\epsilon+\eta$.
It follows that, for~$t$ sufficiently large, $P_t$ satisfies condition $(\text{C1}_{\epsilon+\eta})$.
Since $\epsilon+\eta<1-r_0$, according to Theorem~\ref{th:existence},
for $t$ sufficiently large,
there exists a pair $(\bbeta_1(P_t),\bgamma(P_t))\in \mathfrak{D}$ that minimizes~$R_{P_t}(\bbeta,\bV(\bgamma))$
and a vector $\btheta_1(P_t)$ that is the unique solution of~\eqref{def:theta1} with $P=P_t$.
This proves part(ii).
\end{proof}

\subsection{Proofs for Section~\ref{sec:continuity}}

\paragraph*{Proof of Theorem~\ref{th:continuity}}
\begin{proof}
Let $\bbeta_{0,t}=\bbeta_0(P_t)$ and $\bbeta_{0,P}=\bbeta_0(P)$.
Let $\bGamma$ be the functional defined in~\eqref{def:Gamma},
and define $\bGamma_{0,t}=\bGamma(\btheta_0(P_t))$ and $\bGamma_{0,P}=\bGamma(\btheta_0(P))$.
Since $\rho_0$ satisfies (R2)-(R3) and $\bV$ satisfies~(V1),
we can apply Lemma~3.2 from Lopuha\"a~\cite{lopuhaa1989}.
As $(\bbeta_{0,t},\bGamma_{0,t})\to(\bbeta_{0,P},\bGamma_{0,P})$, 
it follows that for $s$ fixed,
\begin{equation}
\label{eq:LemmaBillingsley}
\int
\rho_0
\left(
\frac{d(\by,\bX\bbeta_{0,t},\bGamma_{0,t})}{s}
\right)
\,
\text{d}P_t(\bs)
\to
\int
\rho_0
\left(
\frac{d(\by,\bX\bbeta_{0,P},\bGamma_{0,P})}{s}
\right)
\,
\text{d}P(\bs),
\end{equation}
where $d$ is defined in~\eqref{def:Mahalanobis distance}.
Let $\sigma(P)$ be the unique solution of~\eqref{def:sigma}.
Let $\delta>0$ and suppose that $\liminf_{t\to\infty} \sigma(P_t)>\sigma(P)+\delta$.
Since $\rho_0$ is strictly increasing on $[0,c_0]$,
together with~\eqref{eq:LemmaBillingsley}, 
it follows that
\[
\begin{split}
\int
\rho_0
\left(
\frac{d(\by,\bX\bbeta_{0,t},\bGamma_{0,t})}{\sigma(P_t)}
\right)
\,
\text{d}P_t(\bs)
&\leq
\int
\rho_0
\left(
\frac{d(\by,\bX\bbeta_{0,t},\bGamma_{0,t})}{\sigma(P)+\delta}
\right)
\,
\text{d}P_t(\bs)\\
&\to
\int
\rho_0
\left(
\frac{d(\by,\bX\bbeta_{0,P},\bGamma_{0,P})}{\sigma(P)+\delta}
\right)
\,
\text{d}P(\bs)\\
&<
\int
\rho_0
\left(
\frac{d(\by,\bX\bbeta_{0,P},\bGamma_{0,P})}{\sigma(P)}
\right)
\,
\text{d}P(\bs)
=
b_0,
\end{split}
\]
which is in contradiction with the definition of $\sigma(P_t)$.
The argument is similar for $\sigma(P_t)<\sigma(P)-\delta$.
We conclude that $|\sigma(P_t)-\sigma(P)|<\delta$, for $t$ sufficiently large.
Since $\delta>0$ is arbitrary, this means that $\sigma(P_t)\to\sigma(P)$.
This proves part~(i).

To prove part~(ii), 
first note that there exists $0<\eta<\epsilon'-\epsilon$.
Because $(\bbeta_{1,t},\bgamma_t)=(\bbeta_1(P_t),\bgamma(P_t))$
is a local minimum of $R_{P_t}(\bbeta,\bV(\bgamma))$
that satisfies~\eqref{eq:ineq RP}, we have
\begin{equation}
\label{eq:ineq RPt}
R_{P_t}(\bbeta_{1,t},\bV(\bgamma_t))\leq R_{P_t}(\bbeta_{0,t},\bGamma_{0,t}).
\end{equation}
Then, together with~\eqref{eq:ineq rho functions}, similar to~\eqref{eq:lower bound prob cylinder}
we find that
$P_t(\mathcal{C}(\bbeta_{1,t},\bgamma_t,c_1\sigma(P_t)))\geq 1-r_0$.
Therefore, since $P$ satisfies (C3) and $\mathcal{C}(\bbeta_{1,t},\bgamma_t,c_1\sigma(P_t))\in \mathfrak{C}$,
and $1-r_0>1-r_0-\eta$,
it follows from~\eqref{eq:ranga rao structured} that
\begin{equation}
\label{eq:bound P(Ct)}
\begin{split}
P(\mathcal{C}(\bbeta_{1,t},\bgamma_t,c_1\sigma(P_t))
&\geq
P_t(\mathcal{C}(\bbeta_{1,t},\bgamma_t,c_1\sigma(P_t)))
-
\sup_{C\in \mathfrak{C}}
\left|
P_t(C)-P(C)
\right|\\
&\geq
1-r_0-\eta,
\end{split}
\end{equation}
for $t$ sufficiently large.
Since $1-r_0-\eta>\epsilon$, 
according to Lemma~4.1(i) in Lopuha\"a \textit{et al}~\cite{lopuhaa-gares-ruizgazen2023},
there exists $a_1>0$ only depending only depending on $c_1$ and $(\mathrm{C2}_{\epsilon})$, such that
$\lambda_k(\sigma^2(P_t)\bV(\bgamma_t))\geq a_1$.
Hence,
\[
\lambda_k(\bV(\bgamma_t))\geq 
\frac{a_1}{\sigma^2(P_t)},
\]
for $t$ sufficiently large.
Since $|\bV(\bgamma_t)|=1$, it immediately follows that
\[
\lambda_1(\bV(\bgamma_t))\leq \left(\frac{\sigma^2(P_t)}{a_1}\right)^{k-1}.
\]
Because $\sigma(P_t)\to\sigma(P)$ according to part~(i), there exists $0<L_1<L_2<\infty$, such that
\begin{equation}
\label{eq:bounds V1t}
L_1\leq \lambda_k(\bV(\bgamma_t))
\leq
\lambda_1(\bV(\bgamma_t))\leq L_2,
\end{equation}
for $t$ sufficiently large.
Let $K\subset\R^{k}\times \mathcal{X}$ be a compact set, such that $P(K)\geq r_0+\epsilon'\geq r_0+\epsilon+\eta$, 
which exists according to condition~$(\mathrm{C1}_{\epsilon'})$.
Then according to~\eqref{eq:bound P(Ct)}, it follows from
Lemma~4.1(iii) in Lopuha\"a \textit{et al}~\cite{lopuhaa-gares-ruizgazen2023} 
with $a=1-r_0-\eta$, that $\|\bbeta_{1,t}\|\leq M<\infty$, 
for some $M>0$ that only depends on $c_1$, $a_2$, $\sigma(P)$, $K$ and $(\text{C2}_\epsilon)$.
We conclude that for $t$ sufficiently large the sequence~$\{\bbeta_{1,t}\}$ 
lies in a compact subset of $\R^q$ and 
the sequence~$\{\bV(\bgamma_t)\}$ lies in a compact set $B\subset\R^{k\times k}$.
Then similar to the second part of the proof of Theorem~\ref{th:existence},
conditions~(V1) and~(V3) yield there exists a compact set $B'\subset\R^{q+l}$,
such that for $t$ sufficiently large 
the sequence $\{(\bbeta_{1,t},\bgamma_t)\}\subset B'$.

Since $\rho_1$ satisfies (R2)-(R3),
together with of part~(i), similar to~\eqref{eq:LemmaBillingsley},
for fixed $(\bbeta,\bgamma)\in\R^k\times\bTheta\subset\R^k\times\R^l$,
\begin{equation}\label{eq:bound difference}
R_{P_t}(\bbeta,\bV(\bgamma))\to R_P(\bbeta,\bV(\bgamma)),
\end{equation}
where $R_P$ is defined in~\eqref{def:R_P}.
For the sake of brevity, let us write $R_t=R_{P_t}$.
Since the sequence~$\{(\bbeta_{1,t},\bgamma_t)\}$ lies in a compact set, 
it has a convergent subsequence
$(\bbeta_{1,t_j},\bgamma_{t_j})\to(\bbeta_{1,L},\bgamma_{L})$.
Since $\rho_1$ satisfies~(R2)-(R3) and $\bV$ satisfies (V1), 
similar to~\eqref{eq:LemmaBillingsley}, it follows that
\[
\lim_{j\to\infty}
R_{t_j}(\bbeta_{1,t_j},\bV(\bgamma_{t_j}))
=
R_P(\bbeta_{1,L},\bV(\bgamma_{L})).
\]
Now, suppose that $(\bbeta_{1,L},\bgamma_{L})\neq (\bbeta_1(P),\bgamma(P))$.
Then, since $R_P(\bbeta,\bV(\bgamma))$ is uniquely minimized at $(\bbeta_1(P),\bgamma(P))$,
this would mean that there exists $\epsilon>0$, such that together with~\eqref{eq:bound difference},
\[
\begin{split}
R_{t_j}(\bbeta_{1,t_j},\bV(\bgamma_{t_j}))
&>
R_P(\bbeta_{1,L},\bV(\bgamma_{L}))-\epsilon
\geq
R_P(\bbeta_1(P),\bV(\bgamma(P)))+2\epsilon\\
&\geq
R_{t_j}(\bbeta_1(P),\bV(\bgamma(P)))+\epsilon
>
R_{t_j}(\bbeta_1(P),\bV(\bgamma(P))),
\end{split}
\]
for $t_j$ sufficiently large.
This would mean that $(\bbeta_{1,t_j},\bgamma_{t_j})$ is not the minimizer of $R_{t_j}(\bbeta,\bV(\bgamma))$.
We conclude that $(\bbeta_{1,L},\bgamma_{L})=(\bbeta_1(P),\bgamma(P))$, which proves part~(ii).

Finally, from part~(i) and (V1), we have that
\[
\bV(\btheta_1(P_t))
=
\sigma^2(P_t)\bV(\bgamma(P_t))
\to
\sigma^2(P)\bV(\bgamma(P))
=
\bV(\btheta_1(P)).
\]
Because $\bV$ is continuous and one-to-one, part~(iii) follows.
\end{proof}

\paragraph*{Proof of Corollary~\ref{cor:consistency}}
\begin{proof}
Let $\mathbb{P}_n$ be the empirical measure corresponding to the collection $\mathcal{S}_n$.
According to the Portmanteau Theorem (e.g., see Theorem~2.1 in~\cite{billingsley1968}),
$\mathbb{P}_n$ converges weakly to $P$, with probability one.
Because $(\bbeta_{0,n},\btheta_{0,n})$ satisfies~\eqref{eq:cond beta0 theta0}~,
it follows that $(\bbeta_0(\mathbb{P}_n),\btheta_0(\mathbb{P}_n))=(\bbeta_{0,n},\btheta_{0,n})\to (\bbeta_0(P),\btheta_0(P))$
and that $\sigma_n=\sigma(\mathbb{P}_n)$ the unique solution of~\eqref{def:sigma}, with $P=\mathbb{P}_n$.
Hence, part~(i) follows immediately from Theorem~\ref{th:continuity}(i).
Furthermore, $(\bbeta_{1,n},\bgamma_n)=(\bbeta_1(\mathbb{P}_n),\bgamma(\mathbb{P}_n))$
is a local minimum of $R_P(\bbeta,\bV(\bgamma))$, with $P=\mathbb{P}_n$,
that satisfies~\eqref{eq:ineq RP}, for $P=\mathbb{P}_n$.
Therefore, part~(ii) follows immediately from Theorem~\ref{th:continuity}(ii).
Finally, since $\btheta_{1,n}=\btheta_1(\mathbb{P}_n)$,
part~(iii) follows immediately from Theorem~\ref{th:continuity}(iii).
\end{proof}

\paragraph*{Proof of Theorem~\ref{th:davies}}
\begin{proof}
First consider the multivariate location-scatter M-functional with auxiliary scale at the distribution $F$ of $\by\mid\bX$, for some~$\bX$ fixed.
So $F$ has density $f_{\bmu,\bSigma}$ from~\eqref{def:tyler density}.
Tatsuoka and Tyler~\cite{tatsuoka&tyler2000} define location-scale M-functionals with auxiliary scale by means of a function
$\widetilde{\rho}:[0,\infty)\to[0,1]$.
It relates to our $\rho_1$-function as $\rho_1(d)=\widetilde{\rho}(d^2)$.
The M-functionals with auxiliary scale $\sigma(F)$ are defined 
to be $\balpha(F)$ and $\bA(F)=\sigma^2(F)\bG(F)$, where $(\balpha(F),\bG(F))$ minimizes
\begin{equation}
\label{eq:R1}
\int \rho_1
\left(
\frac{\displaystyle\sqrt{(\mathbf{y}-\balpha)^T\bG^{-1}(\mathbf{y}-\balpha)}}{\sigma(F)}
\right)
f_{\bmu,\bSigma}(\by)
\,\dd \mathbf{y},
\end{equation}
over all $\balpha\in\R^q$ and $\bG\in\text{PDS}(k)$ with $|\bG|=1$.
Our conditions (R1)-(R2) on $\rho_1$
imply Condition~2.1 on $\widetilde{\rho}$ imposed in Tatsuoka and Tyler~\cite{tatsuoka&tyler2000}.
It then follows from Theorem~4.2 in Tatsuoka and Tyler~\cite{tatsuoka&tyler2000}, together with the affine equivariance of the M-functional,
that
for any minimizer $(\balpha(F),\bG(F))$,
the multivariate M-functional $(\balpha(F),\bA(F))=(\balpha(F),\sigma^2(F)\bG(F))$ has a unique
solution 
\[
(\balpha(F),\bA(F))=(\bmu,\bSigma)=(\bX\bbeta^*,\bV(\btheta^*)).
\]
Since this solution is unique, any candidate minimizer of~\eqref{eq:R1} must
be of the form 
\[
(\balpha,\bG)
=
\left(
\bX\bbeta,\frac{\bV(\btheta)}{\sigma^2(F)}
\right)
=
\left(
\bX\bbeta,
\bV(\bgamma)
\right),
\]
for some $\bbeta\in\R^q$ and $\bgamma\in\R^\ell$,
where we use that $\bV$ satisfies (V2).
Furthermore, 
\[
|\bSigma|=|\bA(F)|=\sigma^{2k}(F)|\bG(F)|=\sigma^{2k}(F),
\]
so that $\sigma(F)=|\bSigma|^{1/(2k)}$.
Then, for $\bX$ fixed, minimizing~\eqref{eq:R1} over all $\balpha\in\R^q$ and  $\bG\in\text{PDS}(k)$ with $|\bG|=1$ is equivalent to 
minimizing 
\begin{equation}
\label{eq:R1 with X}
\int \rho_1
\left(
\frac{d(\by,\bX\bbeta,\bV(\bgamma))}{|\bSigma|^{1/(2k)}}
\right)
f_{\bmu,\bSigma}(\by)
\,\dd \mathbf{y},
\end{equation}
over $\bbeta\in\R^q$ and $\bgamma\in\R^\ell$, such that $|\bV(\bgamma)|=1$,
where $d$ is defined in~\eqref{def:Mahalanobis distance}.
As a consequence, for $\bX$ fixed, this minimization problem has a unique solution $(\widetilde\bbeta,\widetilde\bgamma)$,
which satisfies $\bX\widetilde\bbeta=\bmu=\bX\bbeta^*$ and 
\begin{equation}
\label{eq:prop Vgamma}
\bV(\widetilde\bgamma)
=
\frac{\bV(\btheta^*)}{\sigma^2(F)}
=
\frac{\bSigma}{|\bSigma|^{1/k}}.
\end{equation}
Since $\bX$ has full rank, with probability one, it follows that
$\widetilde\bbeta=\bbeta^*$.

We can transfer this result to the minimization of $R_P(\bbeta,\bV(\bgamma))$ as follows.
Because~$P$ is absolutely continuous, it satisfies condition~(C0).
Hence, according to Lemma~\ref{lem:existence sigma} the solution $\sigma(P)$ of~\eqref{def:sigma} is unique.
Because~$\bbeta_0(P)=\bbeta^*$ and $\bGamma(\btheta_0(P))=\bGamma(\btheta^*)=\bV(\btheta^*)/|\bV(\btheta^*)|^{1/k}=\bSigma/|\bSigma|^{1/k}$,
in~\eqref{def:sigma} we have that
\[
\int
\rho_0
\left(
\frac{d(\by,\bX\bbeta_0(P),\bGamma(\btheta_0(P)))}{\sigma}
\right)
\,
\text{d}P(\bs)
=
\E
\left[
\E
\left[
\rho_0
\left(
\frac{d(\by,\bX\bbeta^*,\bSigma)}{\sigma/|\bSigma|^{1/(2k)}}
\right)
\bigg|
\bX
\right]
\right].
\]
The inner expectation is the conditional expectation of $\by\mid\bX$,
which has the same distribution as $\bSigma^{1/2}\bz+\bX\bbeta^*$,
where $\bz$ has density $f_{\textbf{0},\bI_k}$.
This means that we must solve $\sigma(P)$ from
\begin{equation}
\label{eq:sigma elliptical}
\E_{\mathbf{0},\bI_k}
\rho_0
\left(
\frac{\|\bz\|}{\sigma/|\bSigma|^{1/(2k)}}
\right)
=
b_0.
\end{equation}
Furthermore, since $P$ is absolutely continuous, it satisfies 
satisfies $(\text{C1}_\epsilon)$ and $(\text{C2}_\epsilon)$, for some $0<\epsilon\leq 1-r_0$,
where $r_0=b_0/\sup\rho_0$. 
Hence, since $\bV$ satisfies (V1)-(V3), according to Theorem~\ref{th:existence} there exists $(\bbeta_1(P),\bgamma(P))\in \mathfrak{D}$ 
that minimizes $R_P(\bbeta,\bV(\bgamma))$ and $\btheta_1(P)\in\bTheta$ that is the unique solution of~\eqref{def:theta1}.
But then, there must be an $\bX$ such that
\[
\begin{split}
\int \rho_1
\left(
\frac{d(\by,\bX\bbeta_1(P),\bV(\bgamma(P)))}{\sigma(P)}
\right)
f_{\bmu,\bSigma}(\by)
\,\dd \mathbf{y}
\leq
\int \rho_1
\left(
\frac{d(\by,\bX\widetilde\bbeta,\bV(\widetilde\bgamma))}{\sigma(P)}
\right)
f_{\bmu,\bSigma}(\by)
\,\dd \mathbf{y}.
\end{split}
\]
But since for any $\bX$ fixed, $(\widetilde{\bbeta},\widetilde{\bgamma})$ is the unique minimizer of~\eqref{eq:R1 with X},
with probability one, we must have $\bbeta_1(P)=\widetilde{\bbeta}=\bbeta^*$ and $\bgamma(P)=\widetilde{\bgamma}$,
with probability one.
Together with~\eqref{eq:prop Vgamma} this proves part(i).

To prove part~(ii), note that $\btheta_1(P)$ satisfies
\[
\bV(\btheta_1(P))
=
\sigma^2(P)\bV(\bgamma(P))
=
\frac{\sigma^2(P)}{|\bSigma|^{1/k}}\bSigma
=
\frac{\sigma^2(P)}{|\bSigma|^{1/k}}
\bV(\btheta^*)
=
\bV\left(\frac{\btheta^*\sigma^2(P)}{|\bSigma|^{1/k}}\right),
\]
where we use that $\bV$ has a linear structure.
By identifiability
this means that $\btheta_1(P)=\btheta^*\sigma^2(P)/|\bSigma|^{1/k}$.

Finally, for part~(iii) suppose that $b_0=\E_{\mathbf{0},\bI_k}\rho_0(\|\bz\|)$.
It follows immediately from solving~\eqref{eq:sigma elliptical} that $\sigma(P)=|\bSigma|^{1/(2k)}$.
This finishes the proof.
\end{proof}

\subsection{Proofs for Section~\ref{sec:bdp}}

\paragraph*{Proof of Theorem~\ref{th:bdp}}
\begin{proof}
Suppose we replace $m$ points, where $m$ is such that
\[
\begin{split}
m&\leq 
\min\Big(
\lceil nr_0\rceil,
\lceil n-nr_0\rceil-\kappa(\mathcal{S}_n),
n\epsilon_n^*(\bbeta_{0,n},\btheta_{0,n},\mathcal{S}_n)
\Big)
-1
\end{split}
\]
Let $\mathcal{S}_m'$ be the corrupted collection of points.
Write $(\bbeta_{0,m},\btheta_{0,m})=(\bbeta_{0,n}(\mathcal{S}_m'),\btheta_{0,n}(\mathcal{S}_m'))$,
$\bV_{0,m}=\bV(\btheta_{0,n}(\mathcal{S}_m'))$, 
$\bGamma_{0,m}=\bGamma(\btheta_{0,n}(\mathcal{S}_m'))$.
Let $\mathbb{P}_m'$ be the empirical measure corresponding to the
corrupted collection~$\mathcal{S}_m'$ of~$n$ points.
We must show that there exists a pair
$(\bbeta_{1,m},\bgamma_m)=(\bbeta_{1,n}(\mathcal{S}_m'),\bgamma(\mathcal{S}_m'))\in \mathfrak{D}$ 
that satisfies~\eqref{eq:ineq Rn}, for the corrupted collection~$\mathcal{S}_m'$,
and $\btheta_{1,m}=\btheta_{1,n}(\mathcal{S}_m')\in\bTheta$ 
that satisfies $\bV(\btheta_{1,m})=\sigma_m^2\bV(\bgamma_m)$,
and that all pairs $(\bbeta_{1,m},\bgamma_m)$ satisfying~\eqref{eq:ineq Rn} and $\btheta_{1,m}$ do not break down.

We first show that a solution $\sigma_m=\sigma_n(\mathcal{S}_m')$ of
equation~\eqref{def:initial estimators} exists, for the collection~$\mathcal{S}_m'$.
Note that the maximum number of points of~$\mathcal{S}_m'$ that lie in the same hyperplane 
is $m+\kappa(\mathcal{S}_n)$.
Because $m\leq \lceil n-nr_0\rceil-\kappa(\mathcal{S}_n)-1$, it follows that
\[
m+\kappa(\mathcal{S}_n)
\leq
\lceil n-nr_0\rceil-1
<
n(1-r_0).
\]
This means that $\mathbb{P}_m'$ satisfies condition~(C0).
Since $\rho_0$ satisfies (R1)-(R3), according to Lemma~\ref{lem:existence sigma},
there exists a unique solution~$\sigma_m>0$ of~\eqref{def:initial estimators}.
Furthermore, from (R1)-(R3) we have that~$\rho_0(s)$, for $s\in[0,c_0]$, 
varies continuously between zero and $\sup\rho_0$.
Since $m\leq \lceil nr_0\rceil-1<nr_0$, there exists $\eta>0$, such that
\[
\frac{m}{n}\sup\rho_0+\eta<b_0.
\]
Because $\bbeta_{0,m}$ and $\btheta_{0,m}$ do not break down, there exist $M>0$ and $0<L_1\leq L_2<\infty$,
not depending on $\mathcal{S}_m'$,
such that $\|\bbeta_{0,m}\|\leq M$ and $L_1\leq \lambda_{k}(\bV_{0,m})\leq \lambda_{1}(\bV_{0,m})\leq L_2$.
This means that for all~$\bs_i\in \mathcal{S}_n$:
\[
\sqrt{(\by_i-\bX_i\bbeta_{0,m})^T\bGamma_{0,m}^{-1}(\by_i-\bX_i\bbeta_{0,m})}
\leq
\frac{\|\by_i-\bX_i\bbeta_{0,m}\|}{\sqrt{\lambda_k(\bGamma_{0,m})}}
\leq
\frac{\|\by_i\|+\|\bX_i\|\|\bbeta_{0,m}\|}{\sqrt{\lambda_k(\bGamma_{0,m})}},
\]
where
\[
\lambda_k(\bGamma_{0,m})
=
\frac{\lambda_k(\bV_{0,m})}{|\bV_{0,m}|^{1/k}}
\geq
\frac{L_1}{(L_2^k)^{1/k}}=\frac{L_1}{L_2}.
\]
Hence, for all $\bs_i\in \mathcal{S}_n$, we obtain
\[
d_{i,0}=\sqrt{(\by_i-\bX_i\bbeta_{0,m})^T\bGamma_{0,m}^{-1}(\by_i-\bX_i\bbeta_{0,m})}
\leq
\left(
\max_{\bs_i\in \mathcal{S}_n}\|\by_i\|
+
M
\max_{\bs_i\in \mathcal{S}_n}\|\bX_i\|
\right)
\sqrt{L_2/L_1}.
\]
This means, there exists $0<K<\infty$, only depending on $\mathcal{S}_n$,
such that $\max_{\bs_i\in \mathcal{S}_n}d_{i,0}\leq K$.
Since~$\rho_0$ is continuous and $0<\eta<b_0$, we can define $\delta>0$, such that $\rho_0(\delta)=\eta$.
Let $s_0=K/\delta$.
Then,
\[
\begin{split}
\frac1n
\sum_{(\by_i,\bX_i)\in \mathcal{S}_m'}
\rho_0
\left(
\frac{d_{i,0}}{s_0}
\right)
&=
\frac1n
\sum_{(\by_i,\bX_i)\in \mathcal{S}_m'\cap \mathcal{S}_n}
\rho_0
\left(
\frac{d_{i,0}}{s_0}
\right)
+
\frac1n
\sum_{(\by_i,\bX_i)\in \mathcal{S}_m'\setminus\mathcal{S}_n}
\rho_0
\left(
\frac{d_{i,0}}{s_0}
\right)\\
&\leq
\frac{n-m}{n}\rho_0
\left(
\frac{K}{s_0}
\right)
+
\frac{m}{n}\sup\rho_0\\
&\leq
\rho_0(\delta)+\frac{m}{n}\sup\rho_0\\
&=
\eta+\frac{m}{n}\sup\rho_0<b_0.
\end{split}
\]
As $\sigma_m$ is the solution of~\eqref{def:initial estimators}, 
we must have that $\sigma_m\leq s_0$.

Next, we show that there exists a pair
$(\bbeta_{1,m},\bgamma_m)=(\bbeta_{1,n}(\mathcal{S}_m'),\bgamma(\mathcal{S}_m'))\in \mathfrak{D}$ 
that minimizes
\[
R_m(\bbeta,\bV(\bgamma))
=
\frac1n
\sum_{\bs_i\in \mathcal{S}_m'}
\rho_1\left(
\frac{\displaystyle\sqrt{(\by_i-\bX_i\bbeta)^T\bV(\bgamma)^{-1}(\by_i-\bX_i\bbeta)}}{\sigma_m}
\right).
\]
Any minimizer $(\bbeta_{1,m},\bgamma_m)$ of $R_m(\bbeta,\bV(\bgamma))$ must satisfy~\eqref{eq:ineq Rn}.
Together with~\eqref{eq:ineq rho functions}
similar to~\eqref{eq:lower bound prob cylinder}, this implies
\[
\mathbb{P}_m'\left(\mathcal{C}(\bbeta_{1,m},\bgamma_m,c_1\sigma_m)\right)
\geq 
1-r_0.
\]
It then follows that the cylinder $\mathcal{C}(\bbeta_{1,m},\bgamma_m,c_1\sigma_m)$
contains at least $\lceil n-nr_0\rceil$ number of points from the corrupted collection~$\mathcal{S}_m'$.
Furthermore, any corrupted collection $\mathcal{S}'_m$ contains
\begin{equation}
\label{eq:points in cylinder}
\lceil n-nr_0\rceil-m
\geq
\kappa(\mathcal{S}_n)+1
\end{equation}
points of the original collection $\mathcal{S}_n$.
This means that the cylinder $\mathcal{C}(\bbeta_{1,m},\bgamma_m,c_1\sigma_m)$
must contain a non-empty simplex only depending on the original collection $\mathcal{S}_n$.
This implies that
\begin{equation}
\label{eq:lwb lamdak}
\lambda_k(\sigma_m^2\bV(\bgamma_m))\geq a_1>0,
\end{equation}
where~$a_1$ only depends on the original collection $\mathcal{S}_n$.
Hence,
$\lambda_k(\bV(\bgamma_m))\geq a_1/s_0^2>0$.
Since~$|\bV(\bgamma_m)|=1$, it immediately follows that
$\lambda_1(\bV(\bgamma_m))\leq (s_0^2/a_1)^{k-1}<\infty$.
Furthermore, recall that
$\mathcal{C}(\bbeta_{1,m},\bgamma_m,c_1\sigma_m)$ contains a subset~$J_0$ of
$\kappa(\mathcal{S}_n)+1$ points from the original collection~$\mathcal{S}_n$,
according to~\eqref{eq:points in cylinder}.
By definition, $\kappa(\mathcal{S}_n)+1$ original points cannot be on the same hyperplane, so that
\[
\alpha_n
=
\inf_{J\subset \mathcal{S}_n}
\inf_{\|\balpha\|=1}\max_{\bs\in J}
\|\bX\balpha\|>0.
\]
where the first infimum runs over all subsets $J\subset \mathcal{S}_n$ of $\kappa(\mathcal{S}_n)+1$ points.
By definition of~$\alpha_n$,
there exists an original point~$\bs_0\in J_0\subset \mathcal{S}_n\cap
\mathcal{C}(\bbeta_{1,m},\bgamma_m,c_1\sigma_m)$, such that
\[
\|\bbeta_{1,m}\|
=
\|\bX_0\bbeta_{1,m}\|
\times
\frac{\|\bbeta_{1,m}\|}{\|\bX_0\bbeta_{1,m}\|}
\leq
\frac{1}{\alpha_n}
\|\bX_0\bbeta_{1,m}\|.
\]
Because $\bs_0\in \mathcal{C}(\bbeta_{1,m},\bgamma_m,c_1\sigma_m)$,
it follows that
\[
\begin{split}
\|\by_0-\bX_0\bbeta_{1,m}\|^2
&\leq
(\by_0-\bX_0\bbeta_{1,m})^T
(\sigma_m^2\bV(\bgamma_m))^{-1}
(\by_0-\bX_0\bbeta_{1,m})\\
&\leq
c_1^2
\lambda_1(\sigma_m^2\bV(\bgamma_m))
\leq 
a_2,
\end{split}
\]
where~$a_2=c_1^2s_0^2(s_0^2/a_1)^{k-1}$ only depends on the original collection $\mathcal{S}_n$.
Because $\bs_0\in \mathcal{S}_n$, we have that
\[
\|\bX_0\bbeta_{1,m}\|
\leq
\sqrt{a_2}
+
\max_{(\by_i,\bX_i)\in \mathcal{S}_n}
\|\by_i\|
<\infty.
\]
We conclude that there exists a compact set of $\R^q$ that contains $\bbeta_{1,m}$
and a compact set $B\subset\R^{k\times k}$ that contains $\bV(\bgamma_m)$.
This means that $\bgamma_m$ is in the pre-image~$\bV^{-1}(B)$,
and with conditions~(V1) and (V3), it follows that~$\bV^{-1}(B)$ is a compact set in~$\bTheta$.
We conclude that for minimizing~$R_m(\bbeta,\bV(\bgamma))$ 
we can restrict ourselves to a compact set $B'\subset \mathfrak{D}$, only depending on the original collection~$\mathcal{S}_n$.
Because $\rho_1$ and $\bV$ are continuous, it follows that
$R_m(\bbeta,\bV(\bgamma))$ is a continuous function of $(\bbeta,\bgamma)$,
so it attains a minimum on $B'$.
This also means that there exists a pair $(\bbeta_{1,m},\bgamma_m)\in \mathfrak{D}$ 
that satisfies~\eqref{eq:ineq Rn}, for the corrupted collection~$\mathcal{S}_m'$.
Now, consider any pair~$(\bbeta_{1,m},\bgamma_m)$ that satisfies~\eqref{eq:ineq Rn}.
Then the reasoning above yields that there is a compact set~$B'$, 
only depending on the original collection~$\mathcal{S}_n$, that contains $(\bbeta_{1,m},\bgamma_m)$.
Hence, $\bbeta_{1,m}$ and~$\bgamma_m$ do not break down.

Finally, because $\bV$ satisfies (V2), there exists 
$\btheta_{1,m}=\btheta_{1,n}(\mathcal{S}_m')\in\bTheta$ that satisfies $\bV(\btheta_{1,m})=\sigma_m^2\bV(\bgamma_m)$.
From~\eqref{eq:lwb lamdak}, we have that 
$\lambda_k(\bV(\btheta_{1,m}))=\lambda_k(\sigma_m^2\bV(\bgamma_m))\geq a_1>0$,
where~$a_1$ only depends on the original collection $\mathcal{S}_n$.
Furthermore,
\[
\lambda_1(\bV(\btheta_{1,m}))
=
\sigma_m^2
\lambda_1(\bV(\bgamma_m))
\leq
s_0^2\left(\frac{s_0^2}{a_1}\right)^{k-1}
<\infty.
\]
This means that $\btheta_{1,m}$ does not break down.
\end{proof}

\subsection{Proofs for Section~\ref{sec:equations}}

\paragraph*{Proof of Proposition~\ref{prop:score equations}}
\begin{proof}
In STAGE 3 for the MM-functional, we are considering local minima
of $R_P(\bbeta,\bV(\bgamma))$ that satisfy $|\bV(\bgamma)|=1$,
or equivalently $\log|\bV(\bgamma)|=0$.
The Lagrangian corresponding to this constrained minimization problem
is given by
\[
L_P(\bxi,\lambda)
=
R_P(\bbeta,\bV(\bgamma))-
\lambda\log|\bV(\bgamma)|.
\]
Suppose $\bxi(P)=(\bbeta_1(P),\bgamma(P))$ is a local minimum of $R_P(\bbeta,\bV(\bgamma))$.
Then this is also a zero of the partial derivatives
$\partial L_P/\partial \bbeta$,
$\partial L_P/\partial \bgamma$,
and $\partial L_P/\partial \lambda$.
Let us write $\sigma_P=\sigma(P)$ and consider
\[
R_P(\bbeta,\bV(\bgamma))
=
\int
\rho_1\left(\frac{d}{\sigma_P}\right)\,\text{d}P(\bs)
\]
where $d=d(\by,\bX\bbeta,\bV(\bgamma))$, as defined in~\eqref{def:Mahalanobis distance}.
We find that
\[
\begin{split}
\frac{\partial\rho_1\left(d/\sigma_P\right)}{\partial\bbeta}
&=
\frac{1}{2\sigma_P}
u_1\left(\frac{d}{\sigma_P}\right)
\bX^T\bV^{-1}(\by-\bX\bbeta)\\
\frac{\partial\rho_1\left(d/\sigma_P\right)}{\partial\gamma_j}
&=
\frac{1}{2\sigma_P}
u_1\left(\frac{d}{\sigma_P}\right)
(\by-\bX\bbeta)^T\bV^{-1}
\frac{\partial \bV}{\partial \gamma_j}
\bV^{-1}(\by-\bX\bbeta),
\end{split}
\]
for $j=1,\ldots,l$.
Similar to the proof of Lemma~11.2 in Lopuha\"a \emph{et al}~\cite{lopuhaa-gares-ruizgazen2023},
using that~$\rho_1$ and $\bV$ satisfy (R2), (R4) and (V4), respectively,
we find that at $\bxi(P)$ it holds that
$\|\partial\rho_1\left(d_P/\sigma_P\right)/\partial\bbeta\|\leq C_1\|\bX\|$
and 
$\|\partial\rho_1\left(d_P/\sigma_P\right)/\partial\gamma_j\|\leq C_2$,
for universal constants $0<C_1,C_2<\infty$ only depending on $P$ and $\sigma_P$,
where $d_P=d(\by,\bX\bbeta_1(P),\bV(\bgamma(P)))$.
Since~$\E_P\|\bX\|<\infty$, this implies that under conditions~(R4) and~(V4), 
we may interchange the order of integration and differentiation
in $\partial L_P/\partial \bbeta$ and $\partial L_P/\partial \bgamma$, on a neighborhood of~$\bxi(P)$.
Similar to the derivation of equations~(21) in Lopuha\"a \textit{et al}~\cite{lopuhaa-gares-ruizgazen2023},
it follows that besides the constraint $\log|\bV(\bgamma)|=0$,
the pair $(\bxi(P),\lambda_P)$ satisfies
\begin{equation}
\label{eq:M-equations}
\begin{split}
\int
u_1\left(\frac{d}{\sigma_P}\right)
\bX^T\bV^{-1}(\by-\bX\bbeta)
\,\dd P(\bs)
&=
\mathbf{0}\\
\frac{1}{2\sigma_P}
\int
u_1\left(\frac{d}{\sigma_P}\right)
(\by-\bX\bbeta)^T\bV^{-1}
\frac{\partial \bV}{\partial \gamma_j}
\bV^{-1}(\by-\bX\bbeta)
\,\dd P(\mathbf{s})\\
+
\lambda
\,
\mathrm{tr}\left(\bV^{-1}\frac{\partial \bV}{\partial \gamma_j}\right)
&=0,
\end{split}
\end{equation}
for $j=1,\ldots,l$,
where $u_1(s)=\rho_1'(s)/s$ and $d=d(\by,\bX\bbeta,\bV(\bgamma))$, as defined by~\eqref{def:Mahalanobis distance},
and where we abbreviate~$\bV(\bgamma)$ by $\bV$.
To solve $\lambda_P$ from the second set of equations, we multiply the $j$-th equation by~$\gamma_j$ and then
sum over $j=1,\ldots,l$.
This leads to
\[
\begin{split}
\frac{1}{2\sigma_P}
\int
u_1\left(\frac{d}{\sigma_P}\right)
(\by-\bX\bbeta)^T\bV^{-1}
\left(
\sum_{j=1}^l\gamma_j\frac{\partial \bV}{\partial \gamma_j}
\right)
\bV^{-1}(\by-\bX\bbeta)
\,\dd P(\mathbf{s})\\
+
\lambda\,\mathrm{tr}\left(\bV^{-1}\sum_{j=1}^l\gamma_j\frac{\partial \bV}{\partial \gamma_j}\right)
&=0,
\end{split}
\]
which is solved by
\[
\lambda_P
=
\frac{\displaystyle
-\int
u_1(d/\sigma_P)
(\by-\bX\bbeta)^T\bV^{-1}
\left(
\sum_{t=1}^l\gamma_t(\partial \bV/\partial \gamma_t)
\right)
\bV^{-1}(\by-\bX\bbeta)
\,\dd P(\mathbf{s})}{\displaystyle 
2\sigma_P\mathrm{tr}\left(\bV^{-1}
\sum_{t=1}^l\gamma_t(\partial \bV/\partial \gamma_t)\right)}.
\]
When we insert this back into the second equation in~\eqref{eq:M-equations}, we find
\[
\begin{split}
&
-\mathrm{tr}\left(\bV^{-1}\sum_{t=1}^l\gamma_t\frac{\partial \bV}{\partial \gamma_t}\right)
\int
u_1\left(\frac{d}{\sigma_P}\right)
(\by-\bX\bbeta)^T\bV^{-1}
\frac{\partial \bV}{\partial \gamma_j}
\bV^{-1}(\by-\bX\bbeta)
\,\dd P(\mathbf{s})\\
&\quad+
\mathrm{tr}\left(\bV^{-1}\frac{\partial \bV}{\partial \gamma_j}\right)
\int
u_1\left(\frac{d}{\sigma_P}\right)
(\by-\bX\bbeta)^T\bV^{-1}
\left(
\sum_{t=1}^l\gamma_t\frac{\partial \bV}{\partial \gamma_t}
\right)
\bV^{-1}(\by-\bX\bbeta)
\,\dd P(\mathbf{s})
=0,
\end{split}
\]
or briefly
\begin{equation}
\label{eq:M-equations linear dependent}
\int
u_1\left(\frac{d}{\sigma_P}\right)
(\by-\bX\bbeta)^T\bV^{-1}
\bH_{1,j}
\bV^{-1}(\by-\bX\bbeta)
\,\dd P(\mathbf{s})
=0,
\quad
j=1,\ldots,l,
\end{equation}
where
\[
\bH_{1,j}
=
\mathrm{tr}\left(\bV^{-1}\frac{\partial \bV}{\partial \gamma_j}\right)
\left(
\sum_{t=1}^l\gamma_t\frac{\partial \bV}{\partial \gamma_t}
\right)
-
\mathrm{tr}\left(\bV^{-1}\sum_{t=1}^l\gamma_t\frac{\partial \bV}{\partial \gamma_t}\right)
\frac{\partial \bV}{\partial \gamma_j}.
\]
Because $\sum_{j=1}^l \gamma_j\bH_{1,j}=\mathbf{0}$, the system of equations~\eqref{eq:M-equations linear dependent} is linearly dependent.
Similar to Lopuha\"a \textit{et al}~\cite{lopuhaa-gares-ruizgazen2023}, we subtract the constraint from each equation.
Here, for each $j=1,\ldots,l$, we subtract the term
\[
\mathrm{tr}\left(\bV^{-1}\frac{\partial \bV}{\partial \gamma_j}\right)
\log|\bV|
\]
from the left hand side of equation~\eqref{eq:M-equations linear dependent}.
This finishes the proof.
\end{proof}

\paragraph*{Proof of Proposition~\ref{prop:score equations linear}}
\begin{proof}
When $\bV$ is of the form~\eqref{def:V linear}, then $\partial\bV/\partial\gamma_j=\bL_j$
and $\sum_{j=1}^l\gamma_j \partial\bV/\partial\gamma_j=\bV$.
In this case, 
$\bH_{1,j}=
\mathrm{tr}\left(\bV^{-1}\bL_j\right)\bV
-
k\bL_j$,
and $\Psi_{\bgamma,j}$ in Proposition~\ref{prop:score equations} becomes
\[
\begin{split}
\Psi_{\bgamma,j}(\bs,\bxi,\sigma)
&=
\mathrm{tr}\left(\bV^{-1}\bL_j\right)
\left\{
u_1\left(\frac{d}{\sigma}\right)d^2-\log|\bV|
\right\}\\
&\quad
-
k
u_1\left(\frac{d}{\sigma}\right)
(\by-\bX\bbeta)^T\bV^{-1}
\bL_j
\bV^{-1}(\by-\bX\bbeta).
\end{split}
\]
Using that 
\begin{equation}
\label{eq:trace vec}
\tr(\bA^T\bB)=\vc(\bA)^T\vc(\bB)
\end{equation}
the right hand side can be written as
\[
\begin{split}
-\vc\left(
k
u_1\left(\frac{d}{\sigma}\right)
(\by-\bX\bbeta)(\by-\bX\bbeta)^T
-
v_1\left(\frac{d}{\sigma}\right)\sigma^2\bV
-
\bV\log|\bV|
\right)^T
\vc\left(
\bV^{-1}
\bL_j
\bV^{-1}
\right),
\end{split}
\]
where $u_1(s)=\rho_1'(s)/s$ and
$v_1(s)=u_1(s)s^2=\rho_1'(s)s$.
The functions $\Psi_{\bgamma,j}$, for $j=1,\ldots,l$, can be combined in one expression for the
vector valued function $\Psi_{\bgamma}$ as follows.
First note that
\[
\vc\left(
\bV^{-1}
\bL_j
\bV^{-1}
\right)
=
\left(
\bV^{-1}
\otimes
\bV^{-1}
\right)
\vc\left(
\bL_j
\right)
\]
for $j=1,\ldots,l$.
Then, the column vector $\Psi_{\bgamma}=(\Psi_{\bgamma,1},\ldots,\Psi_{\gamma,l})$ can be written as
\[
\Psi_{\bgamma}(\bs,\bxi,\sigma)
=
-\bL^T
\left(
\bV^{-1}
\otimes
\bV^{-1}
\right)
\vc\left(
\Psi_\bV(\bs,\bxi,\sigma)
\right),
\]
with $\Psi_{\bV}$ defined in~\eqref{def:PsiV}.
\end{proof}

\subsection{Proofs for Section~\ref{sec:IF}}
\begin{lemma}
\label{lem:IF xi sigma}
Suppose that~$\rho_1$ satisfies~(R4) and that $\bV$ satisfies~(V4).
Let $\sigma(P)$ be the solution of equation~\eqref{def:sigma} and 
let $\bxi(P)\in \mathfrak{D}$ be a local minimum of $R_P(\bbeta,\bV(\bgamma))$. 
Suppose that $\text{\rm IF}(\bs;\sigma,P)$ exists.
Let~$\bxi(P_{h,\bs})\in \mathfrak{D}$ 
be a local minimum of $R_P(\bbeta,\bV(\bgamma))$ with~$P=P_{h,\bs}$,
and suppose that $\bxi(P_{h,\bs})\to\bxi(P)$, as~$h\downarrow0$.
Let $\Lambda$ be defined by~\eqref{def:Lambda} with~$\Psi$ from~\eqref{def:Psi}
and suppose $\Lambda$ is continuously differentiable with a non-singular derivative 
$\bD_{\bxi}=\partial\Lambda/\partial\bxi$ 
and derivative $\bD_{\sigma}=\partial\Lambda/\partial\sigma$ at $(\bxi(P),\sigma(P))$.
Then for $\bs\in\R^k\times\R^{kq}$,
\begin{itemize}
\item[(i)]
$\text{\rm IF}(\bs;\bxi,P)
=
-\bD_{\bxi}^{-1}
\Big\{
\Psi(\bs,\bxi(P),\sigma(P))
+
\bD_\sigma
\text{\rm IF}(\bs;\sigma,P)
\Big\}$.
\end{itemize}
Let $\btheta_1(P)$ and~$\btheta_1(P_{h,\bs})$ be solutions of equation~\eqref{def:theta1}
and equation~\eqref{def:theta1} with~$P=P_{h,\bs}$, respectively.
\begin{itemize}
\item[(ii)]
Then 
\[
\begin{split}
\mathrm{IF}(\bs;\vc\bV(\btheta_1),P)
&=
2\sigma(P)
\vc(\bV(\bgamma(P)))
\mathrm{IF}(\bs;\sigma,P)\\
&\qquad+
\sigma^2(P)
\frac{\partial\vc(\bV(\bgamma(P)))}{\partial\bgamma^T}
\text{\rm IF}(\bs;\bgamma,P).
\end{split}
\]
\end{itemize}
In addition, suppose that the $k^2\times l$ matrix $\bD_V=\partial\vc(\bV(\btheta_1(P)))/\partial\btheta^T$ has full rank.
Then 
\begin{itemize}
\item[(iii)]
$\mathrm{IF}(\bs;\btheta_1,P)
=
(\bD_V^T\bD_V)^{-1}
\bD_V^T
\mathrm{IF}(\bs;\vc\bV(\btheta_1),P)$.
\end{itemize}
\end{lemma}
\begin{proof}
Denote $\bxi_{h,\bs}=\bxi(P_{h,\bs})=(\bbeta_1(P_{h,\bs}),\bgamma(P_{h,\bs}))$
and $\sigma_{h,\bs}=\sigma(P_{h,\bs})$,
and write~$\bxi_P=\bxi(P)$ and $\sigma_P=\sigma(P)$.
Then $(\bxi_{h,\bs},\sigma_{h,\bs})$ satisfies the score equation~\eqref{eq:Psi=0} for $P$ equal to $P_{h,\bs}$.
We decompose as follows
\begin{equation}
\label{eq:decomposition IF}
\begin{split}
0
&=
\int \Psi(\bs,\bxi_{h,\bs},\sigma_{h,\bs})\,\dd P_{h,\bs}(\bs)\\
&=
(1-h)
\Lambda(\bxi_{h,\bs},\sigma_{h,\bs})
+
h\Big(
\Psi(\bs,\bxi_{h,\bs},\sigma_{h,\bs})-\Psi(\bs,\bxi_P,\sigma_P)
\Big)
+
h\Psi(\bs,\bxi_P,\sigma_P),
\end{split}
\end{equation}
where $\Psi$ and $\Lambda$ are defined by~\eqref{def:Psi} and~\eqref{def:Lambda}, respectively.
We first determine the order of $\bxi_{h,\bs}-\bxi_P$, as $h\downarrow0$.
Because $\rho_1$ and $\bV$ satisfy~(R4) and~(V4), respectively, it follows that
$\Psi(\bs,\bxi_{h,\bs},\sigma_{h,\bs})\to\Psi(\bs,\bxi_P,\sigma_P)$, as $h\downarrow0$.
Because $\Lambda$ is continuous differentiable at $(\bxi_P,\sigma_P)$,
we have that
\begin{equation}
\label{eq:term1}
\begin{split}
\Lambda(\bxi_{h,\bs},\sigma_{h,\bs})
&=
\Lambda(\bxi_P,\sigma_{h,\bs})
+
\frac{\partial\Lambda(\bxi_P,\sigma_{h,\bs})}{\partial\bxi}
(\bxi_{h,\bs}-\bxi_P)
+
o(\|\bxi_{h,\bs}-\bxi_P\|)\\
&=
\Lambda(\bxi_P,\sigma_{h,\bs})
+
\frac{\partial\Lambda(\bxi_P,\sigma_P)}{\partial\bxi}
(\bxi_{h,\bs}-\bxi_P)
+
o(\|\bxi_{h,\bs}-\bxi_P\|)\\
&=
\Lambda(\bxi_P,\sigma_{h,\bs})
+
\bD_{\bxi}
(\bxi_{h,\bs}-\bxi_P)
+
o(\|\bxi_{h,\bs}-\bxi_P\|).
\end{split}
\end{equation}
In general $\Lambda(\bxi_P,\sigma_{h,\bs})\ne \mathbf{0}$.
Since $\text{IF}(\bs;\sigma,P)$ exists, we find
\begin{equation}
\label{eq:term2}
\begin{split}
\Lambda(\bxi_P,\sigma_{h,\bs})
&=
\Lambda(\bxi_P,\sigma_P)
+
\frac{\partial\Lambda(\bxi_P,\sigma_P)}{\partial\sigma}(\sigma_{h,\bs}-\sigma_P)
+
o(|\sigma_{h,\bs}-\sigma_P|)\\
&=
h
\bD_{\sigma}
\text{IF}(\bs;\sigma,P)
+
o(h).
\end{split}
\end{equation}
Here we also use that $\Lambda(\bxi_P,\sigma_P)=\mathbf{0}$, because $\bxi_P$ is a solution of~\eqref{eq:Psi=0}.
Inserting~\eqref{eq:term1} and~\eqref{eq:term2} in~\eqref{eq:decomposition IF}, yields
\[
\mathbf{0}
=
h
\bD_{\sigma}
\text{IF}(\bs;\sigma,P)
+
\bD_{\bxi}
(\bxi_{h,\bs}-\bxi_P)
+
h\Psi(\bs,\bxi_P,\sigma_P)
+
o(\|\bxi_{h,\bs}-\bxi_P\|)
+
o(h).
\]
Since $\bD_{\bxi}$ is non-singular, this means that $\bxi_{h,\bs}-\bxi_P=O(h)$.
When we insert this in the previous equation and divide by~$h$, we obtain
\[
\frac{\bxi_{h,\bs}-\bxi_P}{h}
=
-\bD_{\bxi}^{-1}
\left\{
\bD_{\sigma}
\text{IF}(\bs;\sigma,P)
+
\Psi(\bs,\bxi_P,\sigma_P)
\right\}
+
o(1).
\]
After letting $h\downarrow0$, this proves the first part of the lemma.

Since $\bV$ satisfies~(V4), it holds that
\[
\vc(\bV(\bgamma_{h,\bs}))-\vc(\bV(\bgamma_P))
=
\frac{\partial\vc(\bV(\bgamma_P))}{\partial\bgamma^T}
(\bgamma_{h,\bs}-\bgamma_P)
+
o(\|\bgamma_{h,\bs})-\bgamma_P\|).
\]
From part~(i), this means that $\text{IF}(\bs;\vc\bV(\bgamma),P)$ exists, and is given by
\[
\vc(
\text{IF}(\bs;\bV(\bgamma),P))
=
\frac{\partial\vc(\bV(\bgamma_P))}{\partial\bgamma^T}
\text{IF}(\bs;\bgamma,P).
\]
Since $\btheta_1(P)$ satisfies~\eqref{def:theta1},
it follows that $\bV(\btheta_1(P))=\sigma^2_P\bV(\bgamma_P)$,
and similarly for $\btheta_1(P_{h,\bs})$.
This implies that
\[
\begin{split}
&
\vc\bV(\btheta_1(P_{h,\bs}))-\vc\bV(\btheta_1(P))\\
&=
\sigma_{h,\bs}^2
\vc\bV(\bgamma_{h,\bs})
-
\sigma_P^2
\vc\bV(\bgamma_P)\\
&=
\vc\bV(\bgamma_P)
\left(
\sigma_{h,\bs}^2-\sigma_P^2
\right)
+
\sigma_{h,\bs}^2
\left(
\vc\bV(\bgamma_{h,\bs})-\vc\bV(\bgamma_P)
\right)\\
&=
\vc\bV(\bgamma_P)(2\sigma_P+o(1))
\left(
\sigma_{h,\bs}-\sigma_P
\right)
+
(\sigma_P+o(1))
\left(
\vc\bV(\bgamma_{h,\bs})-\vc\bV(\bgamma_P)
\right).
\end{split}
\]
After dividing by $h$ and letting $h\downarrow0$, this proves the second part of the lemma.

Finally, since $\bV$ satisfies~(V4), as before we can write 
\[
\vc(\bV(\btheta_1(P_{h,\bs}))-\vc(\bV(\btheta_1(P))
=
\bD_V
(\btheta_1(P_{h,\bs})-\btheta_1(P))
+
o(\|\btheta_1(P_{h,\bs})-\btheta_1(P)\|).
\]
Because $\text{IF}(\bs;\vc\bV(\btheta_1),P)$ exists and $\bD_V$ has full rank,
it follows that $\btheta_1(P_{h,\bs})-\btheta_1(P)=O(h)$.
When insert this in the previous equation, then after dividing by $h$ and letting $h\downarrow0$,
the limit exists and we obtain
\[
\text{IF}(\bs;\vc\bV(\btheta_1),P)
=
\bD_V
\text{IF}(\bs;\btheta_1,P).
\]
Because $\bD_V$ has full rank, we can multiply from the left with $(\bD_V^T\bD_V)^{-1}\bD_V^T$,
which proves part three.
\end{proof}

For $\bzeta=(\bbeta,\btheta)\in\R^q\times\bTheta$, let
\begin{equation}
\label{def:Psi0}
\Psi_0(\bs,\bzeta,\sigma)
=
\rho_0
\left(
\frac{\displaystyle\sqrt{(\by-\bX\bbeta)^T\bGamma(\btheta)^{-1}(\by-\bX\bbeta)}}{\sigma}
\right)
-
b_0,
\end{equation}
where $\bGamma$ is defined in~\eqref{def:Gamma}, and define
\begin{equation}
\label{def:Lambda0}
\Lambda_0(\bzeta,\sigma)
=
\int \Psi_0(\bs,\bzeta,\sigma)\,\dd P(\bs).
\end{equation}

\begin{lemma}
\label{lem:IF sigma zeta}
Suppose that $\rho_0$ satisfies~(R4) and that~$\bV$ satisfies~(V4).
Let $\bzeta_0=(\bbeta_0,\btheta_0)$ be the pair of initial functionals and suppose that
$\text{\rm IF}(\bs,\bzeta_0,P)$ exists.
Let $\sigma(P)$ and $\sigma(P_{h,\bs})$ be solutions of equation~\eqref{def:sigma} 
and equation~\eqref{def:sigma} with $P=P_{h,\bs}$, respectively,
and suppose that $\sigma(P_{h,\bs})\to\sigma(P)$, as $h\downarrow0$.
Let $\Lambda_0$ be defined in~\eqref{def:Lambda0} and suppose  it is continuously differentiable with derivatives 
$D_{0,\sigma}=\partial\Lambda_0/\partial\sigma\ne0$ 
and $\bD_{0,\bzeta}=\partial\Lambda_0/\partial\bzeta\in\R^{q+l}$ at~$(\bzeta_0(P),\sigma(P))$. 
Then for $\bs\in\R^k\times\R^{kq}$,
\[
\text{\rm IF}(\bs;\sigma,P)
=
-D_{0,\sigma}^{-1}
\bigg\{
\Psi_0(\bs,\bzeta_0(P),\sigma(P))
+
\bD_{0,\bzeta}^T
\text{\rm IF}(\bs;\bzeta_0,P)
\bigg\}.
\]
\end{lemma}
\begin{proof}
Denote $\bzeta_{h,\bs}=\bzeta_0(P_{h,\bs})$, $\sigma_{h,\bs}=\sigma(P_{h,\bs})$,
and write $\bzeta_{0,P}=\bzeta_0(P)$ and $\sigma_P=\sigma(P)$.
By definition, $\sigma_{h,\bs}$ satisfies equation~\eqref{def:sigma}  for $P$ equal to $P_{h,\bs}$.
Similar to~\eqref{eq:decomposition IF}, we decompose as follows
\begin{equation}
\label{eq:decomposition IFsigma}
\begin{split}
0
=
(1-h)
\Lambda_0(\bzeta_{h,\bs},\sigma_{h,\bs})
&+
h\Big(
\Psi_0(\bs,\bzeta_{h,\bs},\sigma_{h,\bs})-\Psi_0(\bs,\bzeta_{0,P},\sigma_P)
\Big)\\
&+
h\Psi_0(\bs,\bzeta_{0,P},\sigma_P),
\end{split}
\end{equation}
where $\Psi_0$ and $\Lambda_0$ are defined by~\eqref{def:Psi0} and~\eqref{def:Lambda0}, respectively.
We first determine the order of $\sigma_{h,\bs}-\sigma_P$, as $h\downarrow0$.
Because $\rho_0$ satisfies~(R4) and $\bV$ satisfies~(V4), it follows that
$\Psi_0(\bs,\bzeta_{h,\bs},\sigma_{h,\bs})\to\Psi_0(\bs,\bzeta_{0,P},\sigma_P)$, as $h\downarrow0$.
Since~$\Lambda_0$ is continuous differentiable at~$(\bzeta_{0,P},\sigma_P)$,
similar to~\eqref{eq:term1} 
we have that
\[
\begin{split}
\Lambda_0(\bzeta_{h,\bs},\sigma_{h,\bs})
&=
\Lambda_0(\bzeta_{h,\bs},\sigma_P)
+
D_{0,\sigma}
(\sigma_{h,\bs}-\sigma_P)+
o(\sigma_{h,\bs}-\sigma_P).
\end{split}
\]
In general $\Lambda_0(\bzeta_{h,\bs},\sigma_P)\ne \mathbf{0}$, so that the behavior of $\sigma_{h,\bs}-\sigma_P$
depends on the behavior of
\[
\begin{split}
\Lambda_0(\bzeta_{h,\bs},\sigma_P)
&=
\Lambda_0(\bzeta_{0,P},\sigma_P)
+
\frac{\partial\Lambda_0(\bzeta_{0,P},\sigma_P)}{\partial\bzeta}
(\bzeta_{h,\bs}-\bzeta_{0,P})
+
o(\|\bzeta_{h,\bs}-\bzeta_{0,P}\|)\\
&=
\bD_{0,\bzeta}^T
(\bzeta_{h,\bs}-\bzeta_{0,P})
+
o(\|\bzeta_{h,\bs}-\bzeta_{0,P}\|).
\end{split}
\]
Here we also use that $\Lambda_0(\bzeta_{0,P},\sigma_P)=0$, because $\sigma_P$ is a solution of~\eqref{def:sigma}.
Because~$\text{IF}(\bs;\bzeta_0,P)$ exists, we conclude that
$\Lambda_0(\bzeta_{h,\bs},\sigma_P)=O(h)$, and therefore 
\[
\Lambda_0(\bzeta_{h,\bs},\sigma_{h,\bs})
=
D_{0,\sigma}
(\sigma_{h,\bs}-\sigma_P)+
o(\sigma_{h,\bs}-\sigma_P)
+O(h).
\]
When we insert this in the right hand side of~\eqref{eq:decomposition IFsigma}, if follows that
$\sigma_{h,\bs}-\sigma_P=O(h)$.
Then, again from~\eqref{eq:decomposition IFsigma}, we find that
\[
0
=
D_{0,\sigma}
(\sigma_{h,\bs}-\sigma_P)
+
\bD_{0,\bzeta}^T
(\bzeta_{h,\bs}-\bzeta_{0,P})
+
h\Psi_0(\bs,\bzeta_{0,P},\sigma_P)
+
o(h).
\]
After dividing by $h$, it follows that
\[
\frac{\sigma_{h,\bs}-\sigma_P}{h}
=
-D_{0,\sigma}^{-1}
\left\{
\bD_{0,\bzeta}^T
\frac{\bzeta_{h,\bs}-\bzeta_{0,P}}{h}
+
\Psi_0(\bs,\bzeta_{0,P},\sigma_P)
\right\}
+
o(1).
\]
When we let $h\downarrow0$ and use that $\text{IF}(\bs;\bzeta_0,P)$ exists, 
we conclude that the limit of the left hand side exists and is given by
\[
\text{IF}(\bs;\sigma,P)
=
-D_{0,\sigma}^{-1}
\Big\{
\bD_{0,\bzeta}^T
\text{IF}(\bs;\bzeta_0,P)
+
\Psi_0(\bs,\bzeta_{0,P},\sigma_P)
\Big\}.
\]
This proves the lemma.
\end{proof}

\paragraph*{Proof of Theorem~\ref{th:point of symmetry}}
\begin{proof}
Denote $\bxi_{h,\bs}=\bxi(P_{h,\bs})=(\bbeta_1(P_{h,\bs}),\bgamma(P_{h,\bs}))$,
$\sigma_{h,\bs}=\sigma(P_{h,\bs})$,
and write $\bxi_P=\bxi(P)$ and $\sigma_P=\sigma(P)$.
Then $(\bxi_{h,\bs},\sigma_{h,\bs})$ satisfies the regression score equation in~\eqref{eq:Psi=0} for $P$ equal to $P_{h,\bs}$.
Similar to~\eqref{eq:decomposition IF} we decompose the regression score equation~\eqref{eq:Psi=0} 
as follows
\begin{equation}
\label{eq:decom Psi beta}
\begin{split}
\mathbf{0}
&=
\int \Psi_{\bbeta}(\bs,\bxi_{h,\bs},\sigma_{h,\bs})\,\dd P_{h,\bs}(\bs)\\
&=
(1-h)
\Lambda_{\bbeta}(\bxi_{h,\bs},\sigma_{h,\bs})
+
h\Big(
\Psi_{\bbeta}(\bs,\bxi_{h,\bs},\sigma_{h,\bs})-\Psi_{\bbeta}(\bs,\bxi_P,\sigma_P)
\Big)\\
&\phantom{= (1-h) \Lambda_{\bbeta}(\bxi_{h,\bs},\sigma_{h,\bs})}
+
h\Psi_{\bbeta}(\bs,\bxi_P,\sigma_P),
\end{split}
\end{equation}
where $\Psi_{\bbeta}$ and $\Lambda_{\bbeta}$ are defined by~\eqref{def:Psi} and~\eqref{def:Lambda}, respectively.
Because $\rho_1$ and $\bV$ satisfy~(R4) and~(V1), respectively, it follows that
$\Psi_{\bbeta}(\bs,\bxi_{h,\bs},\sigma_{h,\bs})\to\Psi_{\bbeta}(\bs,\bxi_P,\sigma_P)$, as $h\downarrow0$.
Because~$\partial\Lambda_{\bbeta}/\partial\bbeta$ is continuous at $(\bxi_P,\sigma_P)$
and
$(\bgamma_{h,\bs},\sigma_{h,\bs})\to(\bgamma_P,\sigma_P)$, we find that
\[
\begin{split}
\Lambda_{\bbeta}(\bxi_{h,\bs},\sigma_{h,\bs})
&=
\Lambda_{\bbeta}(\bbeta_1(P),\bgamma_{h,\bs},\sigma_{h,\bs})
+
\left(
\bD_{\bbeta}+o(1)
\right)
(\bbeta_1(P_{h,\bs})-\bbeta_1(P)).
\end{split}
\]
Because $\bbeta_1(P)$ is a point of symmetry and $\Psi_{\bbeta}$ is an 
odd function of $\by-\bX\bbeta$, 
it follows that~$\Lambda_{\bbeta}(\bbeta_1(P),\bgamma_{h,\bs},\sigma_{h,\bs})=\mathbf{0}$.
This means that 
\[
\Lambda_{\bbeta}(\bxi_{h,\bs},\sigma_{h,\bs})
=
\bD_{\bbeta}
(\bbeta_1(P_{h,\bs})-\bbeta_1(P))
+
o(\|\bbeta_1(P_{h,\bs})-\bbeta_1(P)\|).
\]
Together with~\eqref{eq:decom Psi beta}, we find that 
\[
\mathbf{0}
=
\bD_{\bbeta}
(\bbeta_1(P_{h,\bs})-\bbeta_1(P))
+
h\Psi_{\bbeta}(\bs,\bxi_P,\sigma_P)
+
o(\|\bbeta_1(P_{h,\bs})-\bbeta_1(P)\|)
+
o(h).
\]
Since $\bD_{\bbeta}$ is non-singular, this implies that $\bbeta_1(P_{h,\bs})-\bbeta_1(P)=O(h)$.
When we insert this in the previous equation and divide by $h$, we obtain
\[
\frac{\bbeta_1(P_{h,\bs})-\bbeta_1(P)}{h}
=
-\bD_{\bbeta}^{-1}\Psi_{\bbeta}(\bs,\bxi_P,\sigma_P)
+
o(1).
\]
When we let $h\downarrow0$, this finishes the proof.
\end{proof}

\begin{lemma}
\label{lem:change order int diff Lambda0}
Let $\Lambda_0$ be defined by~\eqref{def:Lambda0} with~$\Psi_0$ defined in~\eqref{def:Psi0}
and suppose that~$\E\|\bX\|<\infty$.
Suppose that $\rho_0$ and~$\bV$ satisfy (R2), (R4) and (V4), respectively.
Let $\bzeta_0(P)=(\bbeta_0(P),\btheta_0(P))$ be the pair of initial functionals
and let~$\sigma(P)$ be a solution of~\eqref{def:sigma}.
Let $N\subset \R^q\times\bTheta\times(0,\infty)$ be an open neighborhood of~$(\bzeta_0(P),\sigma(P))$.
Then~$\Lambda_0$ is continuous differentiable at $(\bzeta_0(P),\sigma(P))$ and
for all $(\bzeta,\sigma)\in N$,
\[
\frac{\partial\Lambda_0(\bzeta,\sigma)}{\partial \bzeta}
=
\int
\frac{\partial\Psi_0(\bs,\bzeta,\sigma)}{\partial \bzeta}\,\dd P(\bs)
\quad\text{and}\quad
\frac{\partial\Lambda_0(\bzeta,\sigma)}{\partial \sigma}
=
\int
\frac{\partial\Psi_0(\bs,\bzeta,\sigma)}{\partial \sigma}\,\dd P(\bs).
\]
\end{lemma}
\begin{proof}
Let $(\bzeta,\sigma)\in N$.
Consider $\bzeta\mapsto\Lambda_0(\bzeta,\sigma)$ with $\sigma\in(0,\infty)$ fixed.
From~\eqref{def:Psi0} we find
\[
\frac{\partial\Psi_0(\bs,\bzeta,\sigma)}{\partial \bbeta}
=
\rho'_0\left(\frac{d_{\Gamma}}{\sigma}\right)
\frac{1}{2d_{\Gamma}\sigma}
\bX^T\bGamma^{-1}(\by-\bX\bbeta)
=
\frac{1}{2\sigma^2}
u_0\left(\frac{d_{\Gamma}}{\sigma}\right)
\bX^T\bGamma^{-1}(\by-\bX\bbeta)
\]
where $u_0(s)=\rho_0(s)/s$ and $d_\Gamma^2=(\by-\bX\bbeta)^T\bGamma^{-1}(\by-\bX\bbeta)$,
and where we write $\bGamma$ for $\bGamma(\btheta)$, as defined in~\eqref{def:Gamma}.
Similar to the proof of Lemma~11.2 in Lopuha\"a \textit{et al}~\cite{lopuhaa-gares-ruizgazen2023} we obtain
\[
\left\|
\bX^T\bGamma^{-1}(\by-\bX\bbeta)
\right\|^2
\leq
d_{\Gamma}^2\|\bX\|^2\lambda_1(\bGamma^{-1}).
\]
This means that
\[
\left\|
\frac{\partial\Psi_0(\bs,\bzeta,\sigma)}{\partial \bbeta}
\right\|
\leq
\frac{1}{2\sigma}
\left|
u_0\left(\frac{d_{\Gamma}}{\sigma}\right)
\frac{d_{\Gamma}}{\sigma}
\right|
\|\bX\|
\sqrt{\lambda_1(\bGamma^{-1})}.
\]
From (R2) and (R4) it follows that $u_0(s)s=\rho_0'(s)$ is bounded
and for $(\bzeta,\sigma)$ in the neighborhood $N$ of $(\bzeta_0(P),\sigma(P))$,
we have that $1/\sigma$ and $\lambda_1(\bGamma(\btheta)^{-1})$ are uniformly bounded.
This means there exists a universal constant $0<C_1<\infty$, such that
\[
\left\|
\frac{\partial\Psi_0(\bs,\bzeta,\sigma)}{\partial \bbeta}
\right\|
\leq C_1\|\bX\|.
\]
Since $\E\|\bX\|<\infty$, by dominated convergence, it follows that
for $(\bzeta,\sigma)$ in the neighborhood~$N$ of $(\bzeta_0(P),\sigma(P))$, it holds that
\[
\frac{\partial\Lambda_0(\bzeta,\sigma)}{\partial \bbeta}
=
\int
\frac{\partial\Psi_0(\bs,\bzeta,\sigma)}{\partial \bbeta}\,\dd P(\bs),
\]
and that $\partial\Lambda_0/\partial\bbeta$ is continous at $(\bzeta_0(P),\sigma(P))$.
Furthermore, from~\eqref{def:Psi0} we find
\[
\frac{\partial\Psi_0(\bs,\bzeta,\sigma)}{\partial \theta_j}
=
\frac{1}{2\sigma}
\rho'_0\left(\frac{d_{\Gamma}}{\sigma}\right)
(\by-\bX\bbeta)^T\bGamma^{-1}\frac{\partial\bGamma}{\partial\theta_j}\bGamma^{-1}(\by-\bX\bbeta).
\]
for any $j=1,\ldots,l$.
Similar to the proof of Lemma~11.2 in Lopuha\"a \textit{et al}~\cite{lopuhaa-gares-ruizgazen2023},
we find
\[
\left|
(\by-\bX\bbeta)^T\bGamma^{-1}
\frac{\partial \bGamma}{\partial \theta_j}
\bGamma^{-1}(\by-\bX\bbeta)
\right|
\leq
d_\Gamma^2
\left\|\frac{\partial \bGamma}{\partial \theta_j}\right\|
\lambda_1(\bGamma^{-1}).
\]
Furthermore, according to~(V4),
the mapping $\btheta\mapsto\bGamma(\btheta)=\bV(\btheta)/|\bV(\btheta)|^{1/k}$ is continuously differentiable.
This means that
there exists a universal constant $0<M_1<\infty$, such that
\begin{equation}
\label{eq:bound max dV}
\max_{1\leq j\leq l}
\sup_{(\bbeta,\btheta)\in N}
\left\|
\frac{\partial \bGamma(\btheta)}{\partial \theta_j}\right\|\leq M_1.
\end{equation}
We find that
\[
\left\|
\frac{\partial\Psi_0(\bs,\bzeta,\sigma)}{\partial \theta_j}
\right\|
\leq
\frac{\sigma M_1}{2}
\left|
\rho'_0\left(\frac{d_{\Gamma}}{\sigma}\right)
\right|
\frac{d_\Gamma^2}{\sigma^2}
\lambda_1(\bGamma^{-1}).
\]
From~(R2) and~(R4), it follows that $\rho_0'(s)s^2$ is bounded
and for $(\bzeta,\sigma)$ in the neighborhood $N$ of $(\bzeta_0(P),\sigma(P))$,
we have that $\sigma$ and $\lambda_1(\bGamma(\btheta)^{-1})$ are uniformly bounded.
This means there exists a universal constant $0<C_2<\infty$, such that
\[
\left\|
\frac{\partial\Psi_0(\bs,\bzeta,\sigma)}{\partial \theta_j}
\right\|
\leq C_2,
\]
for all $j=1,\ldots,l$.
By dominated convergence, it follows that
for $(\bzeta,\sigma)$ in the neighborhood~$N$ of $(\bzeta_0(P),\sigma(P))$, it holds that
\[
\frac{\partial\Lambda_0(\bzeta,\sigma)}{\partial \btheta}
=
\int
\frac{\partial\Psi_0(\bs,\bzeta,\sigma)}{\partial \btheta}\,\dd P(\bs),
\]
and that $\partial\Lambda_0/\partial\btheta$ is continous at $(\bzeta_0(P),\sigma(P))$.
Finally, from~\eqref{def:Psi0} we obtain
\[
\frac{\partial\Psi_0(\bs,\bzeta,\sigma)}{\partial \sigma}
=
-
\frac{1}{\sigma}
\rho_0'\left(\frac{d_{\Gamma}}{\sigma}\right)
\left(\frac{d_{\Gamma}}{\sigma}\right).
\]
From (R2) and (R4), it follows that $\rho_0'(s)s$ is bounded,
and for $(\bzeta,\sigma)$ in the neighborhood $N$ of $(\bzeta_0(P),\sigma(P))$,
we have that $1/\sigma$ is uniformly bounded.
This means there exists a universal constant $0<C_3<\infty$, such that
\[
\left\|
\frac{\partial\Psi_0(\bs,\bzeta,\sigma)}{\partial \sigma}
\right\|
\leq C_3.
\]
By dominated convergence, it follows that
for $(\bzeta,\sigma)$ in the neighborhood~$N$ of $(\bzeta_0(P),\sigma(P))$, it holds that
\[
\frac{\partial\Lambda_0(\bzeta,\sigma)}{\partial \sigma}
=
\int
\frac{\partial\Psi_0(\bs,\bzeta,\sigma)}{\partial \sigma}\,\dd P(\bs),
\]
and that $\partial\Lambda_0/\partial\sigma$ is continous at $(\bzeta_0(P),\sigma(P))$.
\end{proof}

For convenience we state the following result about spherically contoured densities,
e.g., see Lemma~5.1 in~\cite{lopuhaa1989}.
This lemma uses the commutation matrix $\mathbf{K}_{k,k}$, which
is the $k^2\times k^2$ block matrix with the $(i,j)$-block being equal to the $k\times k$ matrix $\mathbf{\Delta}_{ji}$
consisting of zero's except a 1 at entry $(j,i)$.
A useful property (e.g., see~\cite[Section 3.7]{magnus&neudecker1988}) is that for any $k\times k$ matrix $\bA$,
it holds that
\begin{equation}
\label{eq:prop K}
\bK_{k,k}\vc(\bA)=\vc(\bA^T).
\end{equation}
\begin{lemma}
\label{lem:Lemma 5.1}
Suppose that $\bz$ has a $k$-variate elliptical contoured density
defined in~\eqref{eq:elliptical}, with parameters $\bmu=\mathbf{0}$ and $\bSigma=\bI_k$.
Then $\bu=\bz/\|\bz\|$ is independent of $\|\bz\|$,
has mean zero and covariance matrix $(1/k)\bI_k$.
Furthermore, $\mathbb{E}_{\mathbf{0},\bI_k}[\bu\bu^T\bu]=\mathbf{0}$
and
\[
\mathbb{E}_{\mathbf{0},\bI_k}
\left[
\vc(\bu\bu^T)\vc(\bu\bu^T)^T
\right]
=
\sigma_1(\bI_{k^2}+\mathbf{K}_{k,k})+\sigma_2\vc(\bI_k)\vc(\bI_k)^T,
\]
where $\sigma_1=\sigma_2=(k(k+2))^{-1}$.
\end{lemma}
\begin{proof}
See e.g.~the proof of Lemma~5.1 in Lopuha\"a~\cite{lopuhaa1989}.
\end{proof}

\begin{lemma}
\label{lem:D0zeta=0}
Suppose that~$P$ satisfies~(E) for some $\bzeta^*=(\bbeta^*,\btheta^*)\in\R^q\times\bTheta$
and suppose that $\E\|\bX\|<\infty$.
Suppose that $\rho_0$ and $\bV$ satisfy (R2), (R4) and (V4), respectively.
Let $\bzeta_0=(\bbeta_0,\btheta_0)$ be the pair of initial functionals
satisfying $(\bbeta_0(P),\btheta_0(P))=(\bbeta^*,\btheta^*)$.
Let~$\sigma(P)$ be the unique solution of~\eqref{def:sigma}
and let $\Lambda_0$ be defined in~\eqref{def:Lambda0} with $\Psi_0$ from~\eqref{def:Psi0}.
Then, 
\[
\bD_{0,\bzeta}
=
\frac{\partial\Lambda_0(\bzeta_0(P),\sigma(P))}{\partial\bzeta}
=
\mathbf{0},
\]
and
\[
D_{0,\sigma}
=
\frac{\partial\Lambda_0(\bzeta_{0}(P),\sigma(P))}{\partial\sigma}
=
-\frac{1}{\sigma(P)}
\E_{\mathbf{0},\bI_k}
\left[
\rho_0'
\left(
c_\sigma\|\bz\|
\right)
c_\sigma\|\bz\|
\right],
\]
where $c_\sigma=|\bSigma|^{1/(2k)}/\sigma(P)$.
\end{lemma}
\begin{proof}
Write $\bzeta_{0,P}=(\bbeta_{0,P},\btheta_{0,P})=(\bbeta_0(P),\btheta_0(P))$ and $\sigma_P=\sigma(P)$.
Because $\rho_0$ and~$\bV$ satisfy (R2), (R4) and (V4), respectively, and $\E\|\bX\|<\infty$,
according to Lemma~\ref{lem:change order int diff Lambda0},
we have that~$\Lambda_0$ is continuously differentiable at~$(\bzeta_{0,P},\sigma_P)$
and that we may interchange integration and differentiation in~$\partial\Lambda_0/\partial\bbeta$
at $(\bzeta_{0,P},\sigma_P)$.
With~$u_0(s)=\rho_0'(s)/s$, we find that
\[
\bD_{0,\bbeta}
=
\frac{\partial\Lambda_0(\bzeta_{0,P},\sigma_P)}{\partial\bbeta}
=
-
\frac{1}{2\sigma_P}
\E
\left[
u_0\left(
\frac{d_{\Gamma,0}}{\sigma_P}
\right)
\bX^T\bGamma(\btheta_{0,P})^{-1}(\by-\bX\bbeta_{0,P})
\right],
\]
where $d_{\Gamma,0}=d(\by,\bX\bbeta_{0,P},\bGamma(\btheta_{0,P}))$,
as defined in~\eqref{def:Mahalanobis distance},
with $\bGamma$ defined in~\eqref{def:Gamma}.
Since~$\bbeta_{0,P}=\bbeta^*$ is a point of symmetry of $P$, it follows that
\begin{equation}
\label{eq:D0beta=0}
\bD_{0,\bbeta}=\mathbf{0}.
\end{equation}
According to Lemma~\ref{lem:change order int diff Lambda0},
we may also interchange integration and differentiation in~$\partial\Lambda_0/\partial\btheta$
at $(\bzeta_{0,P},\sigma_P)$.
For any $j=1,\ldots,l$,
we find that
\begin{equation}
\label{eq:D0j}
\begin{split}
D_{0,j}
&=
\frac{\partial\Lambda_0(\bzeta_{0,P},\sigma_P)}{\partial\theta_{j}}\\
&=
-
\frac{1}{2\sigma_P}
\E\left[
u_0\left(\frac{d_{\Gamma,0}}{\sigma_P}\right)
\be_{0,P}^T\bGamma(\btheta_{0,P})^{-1}
\frac{\partial \bGamma(\btheta_{0,P})}{\partial \theta_{j}}
\bGamma(\btheta_{0,P})^{-1}
\be_{0,P}
\right],
\end{split}
\end{equation}
where $\be_{0,P}=\by-\bX\bbeta_{0,P}$.
Since $(\bbeta_{0,P},\btheta_{0,P})=(\bbeta^*,\btheta^*)$ and 
$\bGamma(\btheta_{0,P})=\bV(\btheta^*)/|\bV(\btheta^*)|^{1/k}=\bSigma/|\bSigma|^{1/k}$,
it follows that $d_{\Gamma,0}=d_{\Gamma}^*$, where
\begin{equation}
\label{eq:dGamma*}
(d_{\Gamma}^*)^2
=
|\bSigma|^{1/k}
(\by-\bX\bbeta^*)^T
\bSigma^{-1}
(\by-\bX\bbeta^*).
\end{equation}
Hence, the expectation on the right hand side of~\eqref{eq:D0j} can be written as
\[
|\bSigma|^{2/k}
\E\left[
\E\left[
u_0\left(\frac{d_{\Gamma}^*}{\sigma_P}\right)
(\be^*)^T\bSigma^{-1}
\frac{\partial \bGamma(\btheta^*)}{\partial \theta_{j}}
\bSigma^{-1}
\be^*
\bigg|\bX
\right]
\right]
\]
where $\be^*=\by-\bX\bbeta^*$.
The inner expectation is the conditional expectation
of $\by\mid\bX$, which has the same distribution as $\bSigma^{1/2}\bz+\bX\bbeta^*$,
where $\bz$ has spherical density $f_{\mathbf{0},\bI_k}$.
This means that the inner expection can be written as
\[
\E_{\mathbf{0},\bI_k}
\left[
u_0\left(c_\sigma\|\bz\|\right)
\bz^T\bSigma^{-1/2}
\frac{\partial \bGamma(\btheta^*)}{\partial \theta_{j}}
\bSigma^{-1/2}\bz
\right],
\]
where $c_\sigma=|\bSigma|^{1/(2k)}/\sigma_P$.
Next, let $\bu=\bz/\|\bz\|$ and apply Lemma~\ref{lem:Lemma 5.1}.
It follows that this expectation is equal to
\[
\begin{split}
&
\E_{\mathbf{0},\bI_k}
\left[
u_0\left(c_\sigma\|\bz\|\right)\|\bz\|^2
\right]
\text{tr}
\left(
\E_{\mathbf{0},\bI_k}
\left[
\bu\bu^T
\right]
\bSigma^{-1/2}
\frac{\partial \bGamma(\btheta^*)}{\partial \theta_{j}}
\bSigma^{-1/2}
\right)\\
&=
\E_{\mathbf{0},\bI_k}
\left[
u_0\left(c_\sigma\|\bz\|\right)\|\bz\|^2
\right]
\frac{1}{k}
\text{tr}
\left(
\bSigma^{-1/2}
\frac{\partial \bGamma(\btheta^*)}{\partial \theta_{j}}
\bSigma^{-1/2}
\right).
\end{split}
\]
Because for each $j=1,\ldots,l$, 
\[
\frac{\partial\bGamma}{\partial\theta_j}
=
\frac{\partial\bV/|\bV|^{1/k}}{\partial\theta_j}
=
-\frac{1}{k}
|\bV|^{-1/k}
\text{tr}\left(\bSigma^{-1}\frac{\partial\bV}{\partial\btheta_j}\right)\bV
+
|\bV|^{-1/k}
\frac{\partial\bV}{\partial\btheta_j},
\]
together with $\bV(\btheta^*)=\bSigma$, it follows that
\[
\text{tr}
\left(
\bSigma^{-1/2}
\frac{\partial \bGamma(\btheta^*)}{\partial \theta_{j}}
\bSigma^{-1/2}
\right)
=0,
\]
for all $j=1,2,\ldots,l$.
This means that
\begin{equation}
\label{eq:D0theta=0}
\bD_{0,\btheta}
=
\frac{\partial\Lambda_0(\bzeta_{0,P},\sigma_P)}{\partial\btheta}
=
\mathbf{0}.
\end{equation}
Together with~\eqref{eq:D0beta=0} this proves part one.

According to Lemma~\ref{lem:change order int diff Lambda0},
we may also interchange integration and differentiation in~$\partial\Lambda_0/\partial\sigma$.
We find that
\[
D_{0,\sigma}
=
\frac{\partial\Lambda_0(\bzeta_{0,P},\sigma_P)}{\partial\sigma}
=
-
\frac{1}{\sigma_P}
\E\left[
\rho_0'\left(\frac{d_{\Gamma,0}}{\sigma_P}\right)\frac{d_{\Gamma,0}}{\sigma_P}
\right],
\]
where $d_{\Gamma,0}=d(\by,\bX\bbeta_{0,P},\bGamma(\btheta_{0,P}))$,
as defined in~\eqref{def:Mahalanobis distance},
with $\bGamma$ defined in~\eqref{def:Gamma}.
As before, it follows that $d_{\Gamma,0}=d_{\Gamma}^*$, where $d_{\Gamma}^*$ is defined in~\eqref{eq:dGamma*},
so that 
\[
D_{0,\sigma}
=
-
\frac{1}{\sigma_P}
\E\left[
\rho_0'\left(\frac{d_{\Gamma}^*}{\sigma_P}\right)\frac{d_{\Gamma}^*}{\sigma_P}
\right]
=
-
\frac{1}{\sigma_P}
\E\left[
\E\left[
\rho_0'\left(\frac{d_{\Gamma}^*}{\sigma_P}\right)\frac{d_{\Gamma}^*}{\sigma_P}
\bigg|
\bX
\right]
\right].
\]
Then, the inner expectation on the right hand side is the conditional expectation
of $\by\mid\bX$, which has the same distribution as $\bSigma^{1/2}\bz+\bX\bbeta^*$,
where $\bz$ has spherical density $f_{\mathbf{0},\bI_k}$.
This means that 
\begin{equation}
\label{eq:D0sigma}
D_{0,\sigma}=
-\frac{1}{\sigma_P}
\E_{\mathbf{0},\bI_k}
\left[
\rho_0'
\left(
c_\sigma\|\bz\|
\right)
c_\sigma\|\bz\|
\right]
<0,
\end{equation}
where $c_\sigma=|\bSigma|^{1/(2k)}/\sigma_P$.
\end{proof}

\begin{lemma}
\label{lem:IF sigma elliptical}
Suppose that~$P$ satisfies~(E) for some $(\bbeta^*,\btheta^*)\in\R^q\times\bTheta$
and suppose that $\E\|\bX\|<\infty$.
Suppose that $\rho_0$ satisfies~(R2) and~(R4), and that~$\bV$ satisfies~(V4).
Let~$\bzeta_0=(\bbeta_0,\btheta_0)$ be the pair of initial functionals
satisfying $(\bbeta_0(P),\btheta_0(P))=(\bbeta^*,\btheta^*)$,
and suppose that $\text{\rm IF}(\bs,\bzeta_0,P)$ exists.
Let~$\sigma(P_{h,\bs})$ be the solution of equation~\eqref{def:sigma} with $P=P_{h,\bs}$, 
and suppose that for all $\bs\in\R^k\times\R^{kq}$, $\sigma(P_{h,\bs})\to\sigma(P)$, as $h\downarrow0$,
where~$\sigma(P)$ is a solution of~\eqref{def:sigma}.
Suppose that 
$\E_{\mathbf{0},\bI_k}[
\rho_0'
\left(
c_\sigma\|\bz\|
\right)
c_\sigma\|\bz\|]
>0$,
where $c_\sigma=|\bSigma|^{1/(2k)}/\sigma(P)$.
Then, for~$\bs_0\in\R^k\times\R^{kq}$,
\[
\text{\rm IF}(\bs_0;\sigma,P)
=
\frac{\sigma(P)}{\E_{\mathbf{0},\bI_k}[
\rho_0'
\left(
c_\sigma\|\bz\|
\right)
c_\sigma\|\bz\|]}
\Big\{
\rho_0
\left(
c_\sigma\|\bz_0\|
\right)
-
b_0
\Big\},
\]
where $\bz_0=\bSigma^{-1/2}(\by_0-\bX_0\bbeta^*)$.
\end{lemma}
\begin{proof}
From Lemmas~\ref{lem:change order int diff Lambda0} and~\ref{lem:D0zeta=0}, 
we have that~$\Lambda_0$ is continuously differentiable 
at $(\bzeta_0(P),\sigma(P))$ with
$\bD_{0,\bzeta}=\mathbf{0}$ and
$D_{0,\sigma}=-\E_{\mathbf{0},\bI_k}[\rho_0'(c_\sigma\|\bz\|)c_\sigma\|\bz\|]/\sigma(P)<0$.
Since $(\bbeta_0(P),\btheta_0(P))=(\bbeta^*,\btheta^*)$ and 
$\bGamma(\btheta_0(P))=\bV(\btheta^*)/|\bV(\btheta^*)|^{1/k}=\bSigma/|\bSigma|^{1/k}$,
it follows that
\begin{equation}
\label{eq:Psi0 elliptical}
\Psi_0(\bs_0,\bzeta_0(P),\sigma(P))
=
\rho_0\left(
\frac{d(\by_0,\bX_0\bbeta^*,\bSigma/|\bSigma|^{1/k})}{\sigma(P)}
\right)-b_0
=
\rho_0
\left(
c_\sigma\|\bz_0\|
\right)
-
b_0,
\end{equation}
where $\bz_0=\bSigma^{-1/2}(\by_0-\bX_0\bbeta^*)$.
The lemma now follows immediately from Lemma~\ref{lem:IF sigma zeta}.
\end{proof}

\begin{lemma}
\label{lem:Psi bounded}
Suppose that $\rho_1$ satisfies (R2) and (R4).
Let $\sigma(P)$ be a solution of~\eqref{def:sigma} and let $\Psi=(\Psi_{\bbeta},\Psi_{\bgamma})$, as defined in~\eqref{def:Psi}.
Then there exist $0<C_1<\infty$, only depending on $P$ and $\sigma(P)$, such that
$\|\Psi_{\bbeta}(\bs,\bxi(P),\sigma(P))\|\leq C_1\|\bX\|$.
If in addition, $\bV$ satisfies~(V4), then there exist $0<C_2<\infty$, 
only depending on $P$ and $\sigma(P)$, such that~$\|\Psi_{\bgamma}(\bs,\bxi(P),\sigma(P))\|\leq C_2$.
\end{lemma}
\begin{proof}
The proof is completely similar to that of Lemma~11.2 in Lopuha\"a \emph{et al}~\cite{lopuhaa-gares-ruizgazen2023}.
\end{proof}

\begin{lemma}
\label{lem:change order int diff Lambda}
Let $\Lambda$ be defined by~\eqref{def:Lambda} with~$\Psi$ defined in~\eqref{def:Psi}
and let~$\E\|\bX\|^2<\infty$.
Suppose that $\rho_1$ satisfies (R2) and (R5) and $\bV$ satisfies (V5).
Let $\sigma(P)$ be a solution of~\eqref{def:sigma} and let~$\bxi(P)\in \mathfrak{D}$ be a local minimum of~$R_P(\bbeta,\bV(\bgamma))$.
Let $N\subset \R^k\times \bTheta\times(0,\infty)$ be an open neighborhood of~$(\bxi(P),\sigma(P))$.
Then~$\Lambda$ is continuous differentiable at $(\bxi(P),\sigma(P)$ and
for all $(\bxi,\sigma)\in N$,
\[
\frac{\partial\Lambda(\bxi,\sigma)}{\partial \bxi}
=
\int
\frac{\partial\Psi(\bs,\bxi,\sigma)}{\partial \bxi}\,\dd P(\bs)
\quad\text{and}\quad
\frac{\partial\Lambda(\bxi,\sigma)}{\partial \sigma}
=
\int
\frac{\partial\Psi(\bs,\bxi,\sigma)}{\partial \sigma}\,\dd P(\bs).
\]
\end{lemma}
\begin{proof}
Let $(\bxi,\sigma)\in N$.
Consider $\bxi\mapsto\Lambda(\bxi,\sigma)$ with $\sigma\in(0,\infty)$ fixed.
The proof of
\[
\frac{\partial\Lambda(\bxi,\sigma)}{\partial \bxi}
=
\int
\frac{\partial\Psi(\bs,\bxi,\sigma)}{\partial \bxi}\,\dd P(\bs),
\]
and that $\partial\Lambda/\partial\bxi$ is continuous at~$(\bxi(P),\sigma(P))$, 
is completely similar to that of Lemma~11.3 in Lopuha\"a \emph{et al}~\cite{lopuhaa-gares-ruizgazen2023},
taking into account that $\sigma$ is uniformly bounded away from zero and infinity.

Next consider $\sigma\mapsto\Lambda(\bxi,\sigma)$ with $\bxi\in\R^k\times\bTheta$ fixed. 
From~\eqref{def:Psi}, we find that
\[
\begin{split}
\frac{\partial\Psi_{\bbeta}(\bs,\bxi,\sigma)}{\partial \sigma}
&=
-u_1'\left(\frac{d}{\sigma}\right)\frac{d}{\sigma^2}\bX^T\bV^{-1}(\by-\bX\bbeta)\\
\frac{\partial\Psi_{\bgamma,j}(\bs,\bxi,\sigma)}{\partial \sigma}
&=
-u_1'\left(\frac{d}{\sigma}\right)\frac{d}{\sigma^2}
(\by-\bX\bbeta)^T\bV^{-1}
\bH_{1,j}
\bV^{-1}(\by-\bX\bbeta),
\end{split}
\]
for $j=1,\ldots,l$, where $d^2=(\by-\bX\bbeta)^T\bV^{-1}(\by-\bX\bbeta)$
and $\bH_j$ is defined in~\eqref{def:Hj}.
Taking into account that $\sigma$ is uniformly bounded away from zero and infinity,
similar to the proof of Lemma~11.3 in Lopuha\"a \emph{et al}~\cite{lopuhaa-gares-ruizgazen2023},
we obtain
\[
\left\|
\frac{\partial\Psi_{\bbeta}(\bs,\bxi,\sigma)}{\partial \sigma}
\right\|
\leq
C_1\|\bX\|^2
\quad\text{and}\quad
\left\|
\frac{\partial\Psi_{\bgamma,j}(\bs,\bxi,\sigma)}{\partial \sigma}
\right\|
\leq C_2,
\]
for constants $0<C_1,C_2<\infty$ only depending on $P$.
Hence, it follows by dominated convergence that for $(\bxi,\sigma)$ in the neighborhood $N$
of $(\bxi(P),\sigma(P))$, it holds that
\[
\frac{\partial\Lambda(\bxi,\sigma)}{\partial \sigma}
=
\int
\frac{\partial\Psi(\bs,\bxi,\sigma)}{\partial \sigma}\,\dd P(\bs),
\]
and that $\partial\Lambda/\partial\sigma$ is continuous at~$(\bxi(P),\sigma(P))$.
\end{proof}

\begin{lemma}
\label{lem:Lambda derivative}
Suppose that $P$ satisfies~(E) for some $(\bbeta^*,\btheta^*)\in\R^q\times\bTheta$
and that $\E\|\bX\|^2<\infty$.
Suppose that $\rho_1$ satisfies (R2) and (R5) and that $\bV$ satisfies (V5) and has
a linear structure~\eqref{def:V linear}.
Let $\sigma(P)$ be the solution of~\eqref{def:sigma} and let $\bxi(P)=(\bbeta_1(P),\bgamma(P))\in \mathfrak{D}$ be
a local minimum of $R_P(\bbeta,\bV(\bgamma))$.
Suppose that $\bbeta_1(P)=\bbeta^*$ and $\bV(\bgamma(P))=\bSigma/|\bSigma|^{1/k}$.
Let $\Lambda$ be defined by~\eqref{def:Lambda} and~\eqref{def:Lambda-beta-gamma} with $\Psi$ 
from~\eqref{def:Psi linear}.
Then 
\[
\bD_{\bxi}
=
\frac{\partial \Lambda(\bxi(P),\sigma(P))}{\partial \bxi}
=
\left(
  \begin{array}{cc}
\bD_{\bbeta} & \mathbf{0}\\
    \\
\mathbf{0} & \bD_{\bgamma} \\
  \end{array}
\right),
\]
where
\begin{equation}
\label{def:derivative Lambda beta}
\bD_{\bbeta}=
\frac{\partial\Lambda_{\bbeta}(\bxi(P),\sigma(P))}{\partial \bbeta}
=
-\alpha_1
|\bSigma|^{1/k}
\E\left[
\bX^T\bSigma^{-1}\bX
\right],
\end{equation}
with $\alpha_1$ defined in~\eqref{def:alpha1-gamma1}, 
and
\begin{equation}
\label{def:derivative Lambda gamma}
\bD_{\bgamma}=
\frac{\partial\Lambda_{\bgamma}(\bxi(P),\sigma(P))}{\partial \bgamma}
=
\omega_1\bL^T\left(\bSigma^{-1}\otimes\bSigma^{-1}\right)\bL
-
\omega_2\bL^T
\vc(\bSigma^{-1})
\vc(\bSigma^{-1})^T
\bL,
\end{equation}
where $\bL=\partial\vc(\bV(\bgamma(P)))/\partial\bgamma^T$ is the $k^2\times l$ matrix given in~\eqref{def:L},
$\omega_1=\sigma^2(P)|\bSigma|^{2/k}\gamma_1$ and $\omega_2=\omega_1/k+|\bSigma|^{2/k}$, 
with $\gamma_1$ defined in~\eqref{def:alpha1-gamma1}.
\end{lemma}
\begin{proof}
For convenience, write $\bxi_P=(\bbeta_{1,P},\bgamma_P)=(\bbeta_1(P),\bgamma(P))$,
$\bV_P=\bV(\bgamma(P))$, and $\sigma_P=\sigma(P)$.
Write $\partial \Lambda/\partial\bxi$ as the block matrix
\begin{equation}
\label{def:Lambda derivative}
\frac{\partial \Lambda(\bxi_P,\sigma_P)}{\partial \bxi}
=
\left(
  \begin{array}{cc}
\dfrac{\partial \Lambda_{\bbeta}(\bxi_P,\sigma_P)}{\partial \bbeta} & \dfrac{\partial \Lambda_{\bbeta}(\bxi_P,\sigma_P)}{\partial \bgamma} \\
    \\
\dfrac{\partial \Lambda_{\bgamma}(\bxi_P,\sigma_P)}{\partial \bbeta} & \dfrac{\partial \Lambda_{\bgamma}(\bxi_P,\sigma_P)}{\partial \bgamma} \\
  \end{array}
\right),
\end{equation}
where $\Lambda_{\bbeta}$ and $\Lambda_{\bgamma}$ are defined in~\eqref{def:Lambda-beta-gamma}
with $\Psi_{\bbeta}$ and $\Psi_{\bgamma}$ from~\eqref{def:Psi linear}.
Because $\rho_1$ and $\bV$ satisfy~(R2), (R5), and~(V5), and $\E\|\bX\|^2<\infty$,
according to Lemma~\ref{lem:change order int diff Lambda} we may
we may interchange integration and differentiation  
in~$\partial\Lambda_{\bbeta}/\partial\bgamma$ and $\partial\Lambda_{\bgamma}/\partial\bbeta$.
It can be seen that these are expectations of an odd function of $\by-\bX\bbeta_{1,P}$,
which means that they are equal to zero, as~$\bbeta_{1,P}=\bbeta^*$ is a point of symmetry of $P$.
Therefore
\begin{equation}
\label{eq:block matrix}
\bD_{\bxi}
=
\frac{\partial \Lambda(\bxi_P,\sigma_P)}{\partial \bxi}
=
\left(
  \begin{array}{cc}
\bD_{\bbeta} & \mathbf{0} \\
    \\
\mathbf{0} & \bD_{\bgamma} \\
  \end{array}
\right).
\end{equation}
It remains to determine $\bD_{\bbeta}=\partial\Lambda_{\bbeta}(\bxi_P,\sigma_P)/\partial\bbeta$
and $\bD_{\bgamma}=\partial\Lambda_{\bgamma}(\bxi_P,\sigma_P)/\partial\bgamma$.
According to Lemma~\ref{lem:change order int diff Lambda}, we have that $\Lambda$ is continuous differentiable at $(\bxi_P,\sigma_P)$ and
that we may interchange
integration and differentiation in $\partial\Lambda_{\bbeta}/\partial\bbeta$, where $\Lambda_{\bbeta}$ is defined in~\eqref{def:Lambda-beta-gamma}
with $\Psi_{\bbeta}$ from~\eqref{def:Psi linear}.
We obtain
\begin{equation}
\label{eq:Dbeta elliptical}
\begin{split}
\bD_{\bbeta}
&=
\int
\frac{\partial}{\partial \bbeta}
\Psi_{\bbeta}(\bs,\bxi_P,\sigma_P)\,\dd P(\bs)\\
&=
-\E
\left[
u_1'\left(\frac{d_P}{\sigma_P}\right)
\frac{\bX^T\bV_P^{-1}\be_P
\be_P^T
\bV_P^{-1}\bX}{\sigma_P d_P}
+
u_1\left(\frac{d_P}{\sigma_P}\right)\bX^T\bV_P^{-1}\bX
\right]\\
&=
-\E
\left[
\E
\left[
u_1'\left(\frac{d_P}{\sigma_P}\right)
\frac{\bX^T\bV_P^{-1}\be_P
\be_P^T
\bV_P^{-1}\bX}{\sigma_P d_P}
+
u_1\left(\frac{d_P}{\sigma_P}\right)\bX^T\bV_P^{-1}\bX
\bigg|
\bX
\right]
\right],
\end{split}
\end{equation}
where $d_P^2=\be_P^T\bV_P^{-1}\be_P$ and $\be_P=\by-\bX\bbeta_{1,P}$.
The inner expectation on the right hand side is the conditional expectation of $\by\mid\bX$,
which can be written as
\[
\bX^T\bV_P^{-1/2}
\E
\left[
u_1'\left(\frac{d_P}{\sigma_P}\right)
\frac{\bV_P^{-1/2}\be_P
\be_P^T
\bV_P^{-1/2}}{\sigma_P d_P}
+
u_1\left(\frac{d_P}{\sigma_P}\right)\bI_k
\,\Bigg|\,
\bX
\right]
\bV_P^{-1/2}\bX.
\]
Because $\bbeta_{1,P}=\bbeta^*$ and $\bV_P=\bSigma/|\bSigma|^{1/k}$, the previous expression is equal to
\[
|\bSigma|^{1/k}
\bX^T\bSigma^{-1/2}
\E
\left[
u_1'\left(\frac{d^*}{\sigma_P}\right)
\frac{|\bSigma|^{1/k}\bSigma^{-1/2}\be^*(\be^*)^T\bSigma^{-1/2}}{\sigma_Pd^*}
+
u_1\left(\frac{d^*}{\sigma_P}\right)\bI_k
\,\Bigg|\,
\bX
\right]
\bSigma^{-1/2}\bX,
\]
where $(d^*)^2=|\bSigma|^{1/k}(\by-\bX\bbeta^*)^T\bSigma^{-1}(\by-\bX\bbeta^*)$ and $\be^*=\by-\bX\bbeta^*$.
Note that $\by\mid\bX$ has the same distribution as $\bSigma^{1/2}\bz+\bX\bbeta^*$,
where~$\bz$ has a spherical density $f_{\mathbf{0},\bI_k}$,
so that the expression in the previous display is equal to
\[
|\bSigma|^{1/k}
\bX^T
\bSigma^{-1/2}
\mathbb{E}_{\mathbf{0},\bI_k}
\left[
u_1'\left(c_\sigma\|\bz\|\right)
\frac{c_\sigma}{\|\bz\|}
\bz\bz^T
+
u_1\left(c_\sigma\|\bz\|\right)
\bI_k
\right]
\bSigma^{-1/2}
\bX,
\]
where $c_\sigma=|\bSigma|^{1/(2k)}/\sigma_P$.
Let $\bu=\bz/\|\bz\|$.
Then with Lemma~\ref{lem:Lemma 5.1} we find
\[
\begin{split}
&
\mathbb{E}_{\mathbf{0},\bI_k}
\left[
u_1'\left(c_\sigma\|\bz\|\right)
\frac{c_\sigma}{\|\bz\|}
\bz\bz^T
+
u_1\left(c_\sigma\|\bz\|\right)
\bI_k
\right]\\
&=
\mathbb{E}_{\mathbf{0},\bI_k}
\left[
u_1'\left(c_\sigma\|\bz\|\right)c_\sigma\|\bz\|
\right]
\mathbb{E}_{\mathbf{0},\bI_k}
\left[
\bu\bu^T
\right]
+
\mathbb{E}_{\mathbf{0},\bI_k}
\left[
u_1\left(c_\sigma\|\bz\|\right)
\right]
\bI_k
=
\alpha_1\bI_k,
\end{split}
\]
where
\[
\begin{split}
\alpha_1
&=
\mathbb{E}_{\mathbf{0},\bI_k}
\left[
\frac{1}{k}
u_1'\left(c_\sigma\|\bz\|\right)c_\sigma\|\bz\|
+
u_1\left(c_\sigma\|\bz\|\right)
\right]\\
&=
\mathbb{E}_{\mathbf{0},\bI_k}
\left[
\left(
1-\frac{1}{k}
\right)
\frac{\rho_1'\left(c_\sigma\|\bz\|\right)}{c_\sigma\|\bz\|}
+
\frac1k
\rho_1''\left(c_\sigma\|\bz\|\right)
\right].
\end{split}
\]
We conclude that
\[
\dfrac{\partial \Lambda_{\bbeta}(\bxi_P,\sigma_P)}{\partial \bbeta}
=
-
\alpha_1
|\bSigma|^{1/k}
\E\left[
\bX^T\bSigma^{-1}\bX
\right].
\]
Next, we determine $\partial\Lambda_{\bgamma}(\bxi_P,\sigma_P)/\partial\bgamma$.
From~\eqref{def:Psi linear} we have
\[
\Psi_{\bgamma,j}
=
-
\vc(\bV^{-1}\bL_j\bV^{-1})^T
\vc\left(
\Psi_\bV
\right),
\]
for all $j=1,2,\ldots,l$,
where $\Psi_\bV$ is defined in~\eqref{def:PsiV}. 
Because $|\bV_P|=1$, we have
\[
\begin{split}
\int \Psi_{\bV}(\bs,\bxi_P,\sigma_P)\,\text{d}P(\bs)
&=
\E\left[
k
u_1\left(\frac{d_P}{\sigma_P}\right)
\be_P\be_P^T
-
v_1\left(\frac{d_P}{\sigma_P}\right)
\sigma_P^2\bV_P
\right]\\
&=
\E\left[
\E\left[
k
u_1\left(\frac{d_P}{\sigma_P}\right)
\be_P\be_P^T
-
v_1\left(\frac{d_P}{\sigma_P}\right)
\sigma_P^2\bV_P
\bigg|
\bX
\right]
\right].
\end{split}
\]
Similar to the reasoning before, the inner expectation can be written as
\[
\begin{split}
&
\E_{\mathbf{0},\bI_k}
\left[
ku_1(c_\sigma\|\bz\|)
\bSigma^{1/2}
\bz\bz^T
\bSigma^{1/2}
-
v_1(c_\sigma\|\bz\|)
\sigma_P^2
\bSigma/|\bSigma|^{1/k}
\right]\\
&=
k
\bSigma^{1/2}
\E_{\mathbf{0},\bI_k}
\left[
u_1(c_\sigma\|\bz\|)
\bz\bz^T
\right]
\bSigma^{1/2}
-
\E_{\mathbf{0},\bI_k}
\left[
u_1(c_\sigma\|\bz\|)\|\bz\|^2
\right]
\bSigma
\\
&=
k
\bSigma^{1/2}
\E_{\mathbf{0},\bI_k}
\left[
u_1(c_\sigma\|\bz\|)\|\bz\|^2
\right]
\E_{\mathbf{0},\bI_k}
\left[
\bu\bu^T
\right]
\bSigma^{1/2}
-
\E_{\mathbf{0},\bI_k}
\left[
u_1(c_\sigma\|\bz\|)\|\bz\|^2
\right]
\bSigma\\
&=
\mathbf{0},
\end{split}
\]
since $\E_{\mathbf{0},\bI_k}[\bu\bu^T]=(1/k)\bI_k$, according to Lemma~\ref{lem:Lemma 5.1}.
Hence we conclude that 
\[
\int \Psi_{\bV}(\bs,\bxi_P,\sigma_P)\,\text{d}P(\bs)=\mathbf{0}.
\]
Since, we may interchange integration and differentiation in $\partial\Lambda_{\bgamma}/\partial\bgamma$,
according to Lemma~\ref{lem:change order int diff Lambda},
this means that
for each $j,s=1,\ldots,l$,
\begin{equation}
\label{eq:Dlambda-gamma}
\begin{split}
\frac{\partial\Lambda_{\bgamma,j}(\bxi_P,\sigma_P)}{\partial \gamma_s}
&=
-
\vc\left(\bV_P^{-1}\bL_j\bV_P^{-1}\right)^T\vc\left(
\int
\frac{\partial\Psi_{\bV}(\bs,\bxi_P,\sigma_P)}{\partial \gamma_s}
\,\dd P(\bs)
\right),
\end{split}
\end{equation}
where $\Psi_\bV$ is defined in~\eqref{def:PsiV}.
We have
\begin{equation}
\label{eq:decomp PsiV}
\begin{split}
\frac{\partial \Psi_\bV}{\partial\gamma_s}
&=
\frac{\partial}{\partial\gamma_s}
k
u_1\left(\frac{d}{\sigma}\right)
(\by-\bX\bbeta)(\by-\bX\bbeta)^T
-
\frac{\partial}{\partial\gamma_s}
v_1\left(\frac{d}{\sigma}\right)\sigma^2\mathbf{V}
+
\frac{\partial}{\partial\gamma_s}
(\log|\bV|)\mathbf{V}.
\end{split}
\end{equation}
Because $\bV$ satisfies~\eqref{def:V linear}, if follows that $\partial\bV/\partial\gamma_s=\bL_s$.
Similar to~\eqref{eq:Dbeta elliptical}, for the first term in~\eqref{eq:decomp PsiV} we have at~$(\bxi_P,\sigma_P)$:
\[
\begin{split}
&\int
\frac{\partial}{\partial\gamma_s}
k
u_1\left(\frac{d_P}{\sigma_P}\right)
\be_P\be_P^T
\,\text{d}P(\bs)\\
&\quad=
\E
\left[
\E
\left[
-k
u_1'\left(\frac{d_P}{\sigma_P}\right)
\frac{1}{2\sigma_P d_P}
\be_P^T\bV_P^{-1}\frac{\partial\bV_P}{\partial\gamma_s}\bV_P^{-1}\be_P
\cdot
\be_P\be_P^T
\bigg|
\bX
\right]
\right],
\end{split}
\]
where $d_P^2=\be_P^T\bV_P^{-1}\be_P$ and $\be_P=\by-\bX\bbeta_{1,P}$.
Because $\bbeta_{1,P}=\bbeta^*$ and $\bV_P=\bSigma/|\bSigma|^{1/k}$, 
the inner expectation on the right hand side can be written as
\[
\begin{split}
&
-k
\E
\left[
u_1'\left(\frac{d^*}{\sigma_P}\right)
\frac{|\bSigma|^{2/k}}{2\sigma_Pd^*}
(\be^*)^T\bSigma^{-1}\bL_s\bSigma^{-1}\be^*
\cdot
\be^*(\be^*)^T
\right]\\
&=
-
\sigma_P^2
\E_{\mathbf{0},\bI_k}
\left[
\frac{ku_1'\left(c_\sigma\|\bz\|\right)(c_\sigma\|\bz\|)^3}2
\bu^T\bSigma^{-1/2}\bL_s\bSigma^{-1/2}\bu
\cdot
\bSigma^{1/2}\bu\bu^T\bSigma^{1/2}
\right],
\end{split}
\]
where $\bz=\bSigma^{-1/2}(\by-\bX\bbeta^*)$ and $\bu=\bz/\|\bz\|$.
According to Lemma~\ref{lem:Lemma 5.1}, 
the expectation of the right hand side is equal to
\begin{equation}
\label{eq:term1 decomp PsiV}
-
\sigma_P^2
\E_{\mathbf{0},\bI_k}
\left[
\frac{ku_1'\left(c_\sigma\|\bz\|\right)(c_\sigma\|\bz\|)^3}{2}
\right]
\E_{\mathbf{0},\bI_k}
\left[
\bu^T\bSigma^{-1/2}\bL_s\bSigma^{-1/2}\bu
\bSigma^{1/2}\bu\bu^T\bSigma^{1/2}
\right].
\end{equation}
For the second term in~\eqref{eq:decomp PsiV} we get at $(\bxi_P,\sigma_P)$:
\[
\begin{split}
&
-
\int\frac{\partial}{\partial\gamma_s}
v_1\left(\frac{d_P}{\sigma_P}\right)\sigma_P^2\mathbf{V}_P
\,\text{d}P(\bs)\\
&=
\E\left[
\E\left[
v_1'\left(\frac{d_P}{\sigma_P}\right)
\frac{1}{2\sigma_P d_P}
\be_P^T\bV_P^{-1}\bL_s\bV_P^{-1}\be_P
\cdot\sigma_P^2\mathbf{V}_P
-
v_1\left(\frac{d_P}{\sigma_P}\right)\sigma_P^2\bL_s
\bigg|\bX
\right]
\right],
\end{split}
\]
where $d_P^2=\be_P^T\bV_P^{-1}\be_P$ and $\be_P=\by-\bX\bbeta_{1,P}$.
As before, the inner expectation can be written as
\[
\begin{split}
&
\sigma_P^2
\E\left[
v_1'\left(\frac{d^*}{\sigma_P}\right)
\frac{|\bSigma|^{1/k}}{2\sigma_Pd^*}
(\be^*)^T\bSigma^{-1}\bL_s\bSigma^{-1}\be^*
\cdot
\bSigma
-
v_1\left(\frac{d^*}{\sigma_P}\right)\sigma_P^2
\bL_s
\right]\\
&=
\sigma_P^2
\E_{\mathbf{0},\bI_k}
\left[
\frac{v_1'\left(c_\sigma\|\bz\|\right)c_\sigma\|\bz\|}{2}
\bu^T\bSigma^{-1/2}\bL_s\bSigma^{-1/2}\bu
\cdot
\bSigma
-
\sigma_P^2
v_1\left(c_\sigma\|\bz\|\right)\bL_s
\right].
\end{split}
\]
With Lemma~\ref{lem:Lemma 5.1}, 
for the expectation of the second term in~\eqref{eq:decomp PsiV} we get 
\begin{equation}
\label{eq:term2 decomp PsiV}
\begin{split}
\sigma_P^2
\E_{\mathbf{0},\bI_k}
\left[
\frac{v_1'\left(c_\sigma\|\bz\|\right)c_\sigma\|\bz\|}{2}
\right]
\E_{\mathbf{0},\bI_k}
\left[
\bu^T\bSigma^{-1/2}\bL_s\bSigma^{-1/2}\bu
\right]
\bSigma\\
-
\sigma_P^2
\E_{\mathbf{0},\bI_k}
\left[
v_1\left(c_\sigma\|\bz\|\right)
\right]
\bL_s.
\end{split}
\end{equation}
For the third term in~\eqref{eq:decomp PsiV} we get at $(\bxi_P,\sigma_P)$:
\begin{equation}
\label{eq:term3 decomp PsiV}
\begin{split}
\frac{\partial}{\partial\gamma_s}
(\log|\bV_P|)\mathbf{V}_P
&=
\left(
\frac{\partial\log|\bV_P|}{\partial\gamma_s}
\right)
\bV_P
+
\left(
\log|\bV_P|
\right)
\frac{\partial\bV}{\partial\gamma_s}\\
&=
\text{tr}
\left(
\bV_P^{-1}\bL_s
\right)
\bV_P
=
\text{tr}
\left(
\bSigma^{-1}\bL_s
\right)
\bSigma,
\end{split}
\end{equation}
using that $|\bV_P|=1$.
It follows that 
\begin{equation}
\label{eq:Dlambda-gamma-js}
\begin{split}
&
\int
\frac{\partial\Psi_{\bV}(\bs,\bxi_P,\sigma_P)}{\partial \gamma_s}
\,\dd P(\bs)\\
&=
-
\sigma_P^2
\E_{\mathbf{0},\bI_k}
\left[
\frac{ku_1'\left(c_\sigma\|\bz\|\right)(c_\sigma\|\bz\|)^3}{2}
\right]
\E_{\mathbf{0},\bI_k}
\left[
\bu^T\bSigma^{-1/2}\bL_s\bSigma^{-1/2}\bu
\bSigma^{1/2}\bu\bu^T\bSigma^{1/2}
\right]\\
&\qquad+
\sigma_P^2
\E_{\mathbf{0},\bI_k}
\left[
\frac{v_1'\left(c_\sigma\|\bz\|\right)c_\sigma\|\bz\|}{2}
\right]
\E_{\mathbf{0},\bI_k}
\left[
\bu^T\bSigma^{-1/2}\bL_s\bSigma^{-1/2}\bu
\right]
\bSigma\\
&\qquad\qquad-
\sigma_P^2
\E_{\mathbf{0},\bI_k}
\left[
v_1\left(c_\sigma\|\bz\|\right)
\right]
\bL_s
+
\text{tr}
\left(
\bSigma^{-1}\bL_s
\right)
\bSigma.
\end{split}
\end{equation}
In view of~\eqref{eq:Dlambda-gamma} and~\eqref{eq:Dlambda-gamma-js}, for the first term in $\partial\Lambda_{\bgamma,j}/\partial\gamma_s$
we obtain
\[
\begin{split}
&
\vc(\bV_P^{-1}\bL_j\bV_P^{-1})^T
\vc\left(
\E_{\mathbf{0},\bI_k}
\left[
\bu^T\bSigma^{-1/2}\bL_s\bSigma^{-1/2}\bu
\bSigma^{1/2}\bu\bu^T\bSigma^{1/2}
\right]
\right)\\
&=
|\bSigma|^{2/k}
\vc(\bSigma^{-1}\bL_j\bSigma^{-1})^T
\E_{\mathbf{0},\bI_k}
\left[
\vc\left(
\bSigma^{1/2}\bu\bu^T\bSigma^{1/2}
\right)
\bu^T\bSigma^{-1/2}\bL_s\bSigma^{-1/2}\bu
\right]\\
&=
|\bSigma|^{2/k}
\vc(\bSigma^{-1}\bL_j\bSigma^{-1})^T
(\bSigma^{1/2}\otimes\bSigma^{1/2})
\E_{\mathbf{0},\bI_k}
\left[
\vc\left(
\bu\bu^T
\right)
\vc(\bu\bu^T)^T
\right]\\
&\qquad\qquad\qquad\qquad\qquad\qquad\qquad\qquad\qquad\qquad\qquad
\vc\left(\bSigma^{-1/2}\bL_s\bSigma^{-1/2}\right)
\\
&=
|\bSigma|^{2/k}
\vc(\bSigma^{-1/2}\bL_j\bSigma^{-1/2})^T
\frac{1}{k(k+2)}
\Big(
\bI_{k^2}+\bK_{k,k}
+
\vc(\bI_k)\vc(\bI_k)^T
\Big)\\
&\qquad\qquad\qquad\qquad\qquad\qquad\qquad\qquad\qquad\qquad\qquad
\vc\left(\bSigma^{-1/2}\bL_s\bSigma^{-1/2}\right),
\end{split}
\]
using Lemma~\ref{lem:Lemma 5.1}.
Application of properties~\eqref{eq:prop K} and~\eqref{eq:trace vec}, yields
\[
\begin{split}
&
\vc(\bV_P^{-1}\bL_j\bV_P^{-1})^T
\vc\left(
\E_{\mathbf{0},\bI_k}
\left[
\bu^T\bSigma^{-1/2}\bL_s\bSigma^{-1/2}\bu
\bSigma^{1/2}\bu\bu^T\bSigma^{1/2}
\right]
\right)\\
&=
\frac{|\bSigma|^{2/k}}{k(k+2)}
\Big(
2\text{tr}(\bSigma^{-1}\bL_j\bSigma^{-1}\bL_s)
+
\text{tr}(\bSigma^{-1}\bL_j)\text{tr}(\bSigma^{-1}\bL_s)
\Big).
\end{split}
\]
It follows that the first term in $\partial\Lambda_{\bgamma,j}/\partial\gamma_s$ is equal to
\begin{equation}
\label{eq:first term}
\sigma_P^2
|\bSigma|^{2/k}
\frac{\E_{\mathbf{0},\bI_k}
\left[
u_1'(c_\sigma\|\bz\|)(c_\sigma\|\bz\|)^3
\right]}{2(k+2)}
\Big(
2\text{tr}(\bSigma^{-1}\bL_j\bSigma^{-1}\bL_s)+\text{tr}(\bSigma^{-1}\bL_j)\text{tr}(\bSigma^{-1}\bL_s)
\Big).
\end{equation}
Similarly, for the second term in $\partial\Lambda_{\bgamma,j}/\partial\gamma_s$
we obtain
\[
\begin{split}
&
\vc\left(\bV_P^{-1}\bL_j\bV_P^{-1}\right)^T
\vc\left(\E_{\mathbf{0},\bI_k}
\left[
\bu^T\bSigma^{-1/2}\bL_s\bSigma^{-1/2}\bu
\right]
\bSigma
\right)\\
&=
|\bSigma|^{2/k}
\vc\left(\bSigma^{-1}\bL_j\bSigma^{-1}\right)^T
\text{tr}
\left(
\E_{\mathbf{0},\bI_k}
\left[
\bu\bu^T
\right]
\bSigma^{-1/2}\bL_s\bSigma^{-1/2}
\right)
\vc\left(\bSigma\right)
\\
&=
|\bSigma|^{2/k}
\text{tr}(\bSigma^{-1}\bL_s)
\vc\left(\bSigma^{-1}\bL_j\bSigma^{-1}\right)^T
\vc\left(\bSigma\right)\\
&=
|\bSigma|^{2/k}
\text{tr}(\bSigma^{-1}\bL_s)
\text{tr}(\bSigma^{-1}\bL_j).
\end{split}
\]
It follows that the second term in $\partial\Lambda_{\bgamma,j}/\partial\gamma_s$ is equal to
\begin{equation}
\label{eq:second term}
-\sigma_P^2
|\bSigma|^{2/k}
\frac{\E_{\mathbf{0},\bI_k}
\left[
v_1'(c_\sigma\|\bz\|)c_\sigma\|\bz\|
\right]}{2k}
\text{tr}\left(\bSigma^{-1}\bL_s\right)
\text{tr}\left(\bSigma^{-1}\bL_j\right).
\end{equation}
Finally, the third term in $\partial\Lambda_{\bgamma,j}/\partial\gamma_s$ is equal to
\begin{equation}
\label{eq:third term}
\begin{split}
&
\sigma_P^2
|\bSigma|^{2/k}
\E_{\mathbf{0},\bI_k}
\left[
v_1(c_\sigma\|\bz\|)
\right]
\vc(\bSigma^{-1}\bL_j\bSigma^{-1})^T
\vc(\bL_s)\\
&=
\sigma_P^2
|\bSigma|^{2/k}
\E_{\mathbf{0},\bI_k}
\left[
v_1(c_\sigma\|\bz\|)
\right]
\text{tr}(\bSigma^{-1}\bL_j\bSigma^{-1}\bL_s),
\end{split}
\end{equation}
and the fourth term in $\partial\Lambda_{\bgamma,j}/\partial\gamma_s$ is equal to
\begin{equation}
\label{eq:fourth term}
-|\bSigma|^{2/k}
\text{tr}
\left(
\bSigma^{-1}\bL_s
\right)
\vc(\bSigma^{-1}\bL_j\bSigma^{-1})^T(\vc\bSigma)
=
-|\bSigma|^{2/k}
\text{tr}
\left(
\bSigma^{-1}\bL_s
\right)
\text{tr}
\left(
\bSigma^{-1}\bL_j
\right).
\end{equation}
We conclude that $\partial\Lambda_{\bgamma,j}(\bxi_P)/\partial \gamma_s$ consists of
a term
$\text{tr}(\bSigma^{-1}\bL_j\bSigma^{-1}\bL_s)$
from~\eqref{eq:first term} and~\eqref{eq:third term}
with coefficient
\[
\begin{split}
\omega_1
&=
\sigma_P^2
|\bSigma|^{2/k}
\left(
\frac{\E_{\mathbf{0},\bI_k}
\left[
u_1'(c_\sigma\|\bz\|)(c_\sigma\|\bz\|)^3
\right]}{k+2}
+
\E_{\mathbf{0},\bI_k}
\left[
v_1(c_\sigma\|\bz\|)
\right]
\right)\\
&=
\sigma_P^2
|\bSigma|^{2/k}
\frac{\mathbb{E}_{0,\mathbf{I}_k}
\left[
\rho''(c_\sigma\|\bz\|)(c_\sigma\|\bz\|)^2+(k+1)\rho'(c_\sigma\|\bz\|)c_\sigma\|\bz\|
\right]}{k+2},
\end{split}
\]
and a term $-\text{tr}(\bSigma^{-1}\bL_s)\text{tr}(\bSigma^{-1}\bL_j)$ 
from~\eqref{eq:first term}, \eqref{eq:second term}, and~\eqref{eq:fourth term}
with coefficient
\[
\begin{split}
\omega_2
&=
\sigma_P^2
|\bSigma|^{2/k}
\left(
\frac{\E_{\mathbf{0},\bI_k}
\left[
u_1'(c_\sigma\|\bz\|)(c_\sigma\|\bz\|)^3
\right]}{2(k+2)}
-
\frac{\E_{\mathbf{0},\bI_k}
\left[
v_1'(c_\sigma\|\bz\|)c_\sigma\|\bz\|
\right]}{2k}
\right)
+
|\bSigma|^{2/k}\\
&=
\sigma_P^2
|\bSigma|^{2/k}
\frac{\mathbb{E}_{0,\mathbf{I}_k}
\left[
\rho''(c_\sigma\|\bz\|)(c_\sigma\|\bz\|)^2+(k+1)\rho'(c_\sigma\|\bz\|)c_\sigma\|\bz\|
\right]}{k(k+2)}
+
|\bSigma|^{2/k}\\
&=
\frac{\omega_1}{k}+|\bSigma|^{2/k}.
\end{split}
\]
From the definition of $\bL$ in~\eqref{def:L} it follows that the $l\times l$ matrix with entries
\[
\begin{split}
\frac{\partial\Lambda_{\bgamma,j}(\bxi_P,\sigma_P)}{\partial \gamma_s}
&=
\omega_1\text{tr}(\bSigma^{-1}\bL_j\bSigma^{-1}\bL_s)
-
\omega_2\text{tr}(\bSigma^{-1}\bL_s)\text{tr}(\bSigma^{-1}\bL_j),
\end{split}
\]
is the matrix
\[
\bD_{\bgamma}
=
\frac{\partial\Lambda_{\bgamma}(\bxi_P,\sigma_P)}{\partial \bgamma}
=
\omega_1\mathbf{L}^T\left(\bSigma^{-1}\otimes\bSigma^{-1}\right)\mathbf{L}
-
\omega_2\mathbf{L}^T
\vc(\bSigma^{-1})
\vc(\bSigma^{-1})^T
\mathbf{L}.
\]
This proves the lemma.
\end{proof}

\begin{lemma}
\label{lem:inverse}
Suppose the conditions of Lemma~\ref{lem:Lambda derivative} hold.
Let $\bD_{\bbeta}$ be defined in~\eqref{def:derivative Lambda beta} with~$\alpha_1$ from~\eqref{def:alpha1-gamma1}.
When $\alpha_1\ne0$ and $\bX$ has full rank with probability one,
then $\bD_{\bbeta}$ is non-singular with inverse
\[
\bD_{\bbeta}^{-1}
=
-\frac{1}{\alpha_1|\bSigma|^{1/k}}
\left(
\E\left[
\bX^T\bSigma^{-1}\bX
\right]
\right)^{-1}.
\] 
Let $\bD_{\bgamma}$ be defined in~\eqref{def:derivative Lambda gamma}
with $\omega_1=\sigma^2(P)|\bSigma|^{2/k}\gamma_1$ and $\omega_2=\omega_1/k+|\bSigma|^{2/k}$, 
where $\gamma_1$ is defined in~\eqref{def:alpha1-gamma1} and $\bL$ is defined in~\eqref{def:L}.
When $\gamma_1\ne0$ and~$\bL$ has full rank,
then~$\bD_{\bgamma}$ is non-singular with inverse
\[
\bD_{\bgamma}^{-1}
=
a(\bE^T\bE)^{-1}
+
b
(\bE^T\bE)^{-1}
\bE^T\vc(\bI_k)\vc(\bI_k)^T\bE
(\bE^T\bE)^{-1},
\]
where $\bE=(\bSigma^{-1/2}\otimes\bSigma^{-1/2})\bL$,
$a=1/\omega_1$, and $b=\omega_2/(\omega_1(\omega_1-k\omega_2))$.
\end{lemma}
\begin{proof}
Because~$\bX$ has full rank with probability one, it follows that $\E[\bX^T\bSigma^{-1}\bX]$ is non-singular.
Since $\alpha_1\ne0$, this proves part one.
Then consider $\bD_{\bgamma}$ as given in~\eqref{def:derivative Lambda gamma}.
With $\bE=(\bSigma^{-1/2}\otimes\bSigma^{-1/2})\bL$,
we can write
\[
\bD_{\bgamma}
=
\omega_1\bE^T\bE
-
\omega_2\bE^T\vc(\bI_k)\vc(\bI_k)^T\bE,
\]
which is of the form $\bA+\bu\bv^T$.
Because $\bL$ has full rank and $\omega_1=\sigma^2(P)|\bSigma|^{2/k}\gamma_1\ne0$, it follows that 
$\bA=\omega_1\bE^T\bE=\omega_1\bL^T\left(\bSigma^{-1}\otimes\bSigma^{-1}\right)\bL$ is non-singular.
Furthermore, since~$\bV$ has a linear structure, we have that $\vc(\bSigma)=\bL\btheta^*$,
which implies $\bE\btheta^*=\vc(\bI_k)$ and $(\bE^T\bE)^{-1}\bE^T\vc(\bI_k)=\btheta^*$.
This means that $1+\bv^T\bA^{-1}\bu=(\omega_1-k\omega_2)/\omega_1=k|\bSigma|^{2/k}/\omega_1\ne0$.
It then follows from the Sherman-Morisson formula that $\bD_{\bgamma}$ is non-singular
and has inverse
\[
\bD_{\bgamma}^{-1}
=
a(\bE^T\bE)^{-1}
+
b
(\bE^T\bE)^{-1}
\bE^T\vc(\bI_k)\vc(\bI_k)^T\bE
(\bE^T\bE)^{-1},
\]
where 
$a=1/\omega_1$ and $b=\omega_2/(\omega_1(\omega_1-k\omega_2))$.
\end{proof}

\begin{lemma}
\label{lem:Dsigma=0}
Suppose $P$ satisfies~(E) for some $(\bbeta^*,\btheta^*)\in\R^q\times\bTheta$
and~$\E\|\bX\|^2<\infty$.
Suppose that $\rho_1$ satisfy (R2) and (R5), and suppose that $\bV$ satisfy and~(V5).
Let $\sigma(P)$ be the solution of~\eqref{def:sigma} and let $\bxi(P)=(\bbeta_1(P),\bgamma(P))\in \mathfrak{D}$ be
a local minimum of $R_P(\bbeta,\bV(\bgamma))$ that
satisfies~$\bbeta_1(P)=\bbeta^*$ and $\bV(\bgamma(P))=\bSigma/|\bSigma|^{1/k}$.
Then 
\[
\bD_{\sigma}
=
\frac{\partial\Lambda(\bxi(P),\sigma(P))}{\partial\sigma}=\mathbf{0}.
\]
\end{lemma}

\begin{proof}
For convenience, write $\bxi_P=(\bbeta_{1,P},\bgamma_P)=(\bbeta_1(P),\bgamma(P))$,
$\bV_P=\bV(\bgamma(P))$, and $\sigma_P=\sigma(P)$.
Consider $\Lambda_{\bbeta}$ as defined in~\eqref{def:Lambda-beta-gamma} with $\Psi_{\bbeta}$ 
from~\eqref{def:Psi linear}.
Because $\rho_1$ and $\bV$ satisfy~(R2), (R5), and~(V5), and $\E\|\bX\|^2<\infty$,
according to Lemma~\ref{lem:change order int diff Lambda},
we may interchange differentiation and integration in $\partial\Lambda_{\bbeta}/\partial\sigma$.
We find
\[
\begin{split}
\frac{\partial\Lambda_{\bbeta}(\bxi_P,\sigma_P)}{\partial\sigma}
&=
\int 
\frac{\partial\Psi_{\bbeta}(\bs,\bxi_P,\sigma_P)}{\partial\sigma}
\,\text{d}P(\bs)\\
&=
-
\E\left[
u_1'\left(\frac{d_P}{\sigma_P}\right)\frac{d_P}{\sigma_P^2}
\bX^T\bV_P^{-1}(\by-\bX\bbeta_{1,P})
\right].
\end{split}
\]
where $d_P^2=(\by-\bX\bbeta_{1,P})^T\bV_P^{-1}(\by-\bX\bbeta_{1,P})$.
Because $\bbeta_{1,P}=\bbeta^*$ is a point of symmetry,
it follows that 
\begin{equation}
\label{eq:Dbeta-sigma=0}
\bD_{\bbeta,\sigma}
=
\frac{\partial\Lambda_{\bbeta}(\bxi_P,\sigma_P)}{\partial\sigma}
=
\mathbf{0}.
\end{equation}
According to Lemma~\ref{lem:change order int diff Lambda},
we may also interchange differentiation and integration
in~$\partial\Lambda_{\bgamma}/\partial \sigma$,
where $\Lambda_{\bgamma}$ is defined by~\eqref{def:Lambda-beta-gamma} with $\Psi_{\bgamma}$
from~\eqref{def:Psi linear}.
For all $j=1,\ldots,l$ we obtain
\[
\frac{\partial\Lambda_{\bgamma,j}(\bxi_P,\sigma_P)}{\partial \sigma}
=
-
\vc\left(\bV_P^{-1}\bL_j\bV_P^{-1}\right)^T\vc\left(
\int
\frac{\partial\Psi_{\bV}(\bs,\bxi_P,\sigma_P)}{\partial \sigma}
\,\dd P(\bs)
\right),
\]
where $\Psi_\bV$ is defined in~\eqref{def:PsiV}.
Since $v_1(s)=u_1(s)s^2$, we have 
\[
\begin{split}
\int
\frac{\partial\Psi_{\bV}(\bs,\bxi_P,\sigma_P)}{\partial \sigma}\,\dd P(\bs)
&=
-\E
\left[
\E
\left[
ku_1'\left(\frac{d_P}{\sigma_P}\right)
\frac{d_P}{\sigma_P^2}
\be_P\be_P^T
-
u_1'\left(\frac{d_P}{\sigma_P}\right)
\frac{d_P^3}{\sigma_P^2}\bV_P
\,\Bigg|\,
\bX
\right]
\right],
\end{split}
\]
where $d^2_P=\be_P^T\bV_P^{-1}\be_P$ and $\be_P=\by-\bX\bbeta_{1,P}$.
The inner expectation on the right hand side is the conditional expectation of $\by\mid\bX$.
Because $\bbeta_{1,P}=\bbeta^*$ and $\bV_P=\bSigma/|\bSigma|^{1/k}$,
it follows that 
$d^2_P=|\bSigma|^{1/k}(\by-\bX\bbeta^*)^T\bSigma^{-1}(\by-\bX\bbeta^*)$.
Furthermore, $\by\mid\bX$ has the same distribution as $\bSigma^{1/2}\bz+\bX\bbeta^*$,
where~$\bz$ has a spherical density $f_{\mathbf{0},\bI_k}$.
This means that the inner expectation is equal to
\[
\begin{split}
&\E
\left[
ku_1'\left(\frac{d_P}{\sigma_P}\right)
\frac{d_P}{\sigma_P^2}
\be_P\be_P^T
-
u_1'\left(\frac{d_P}{\sigma_P}\right)
\frac{d_P^3}{\sigma_P^2}\bV_P
\,\Bigg|\,
\bX
\right]\\
&=
\E_{\mathbf{0},\bI_k}
\left[
ku_1'\left(c_\sigma\|\bz\|\right)
\frac{c_\sigma\|\bz\|}{\sigma_P}
\bSigma^{1/2}\bz\bz^T\bSigma^{1/2}
-
u_1'\left(c_\sigma\|\bz\|\right)
(c_\sigma\|\bz\|)^3
\frac{\sigma_P}{|\bSigma|^{1/k}}\bSigma
\right]\\
&=
\frac{\sigma_P}{|\bSigma|^{1/k}}
\E_{\mathbf{0},\bI_k}
\left[
ku_1'\left(c_\sigma\|\bz\|\right)
(c_\sigma\|\bz\|)^3
\bSigma^{1/2}
\bu\bu^T
\bSigma^{1/2}
-
u_1'\left(c_\sigma\|\bz\|\right)
(c_\sigma\|\bz\|)^3\bSigma
\right]
\end{split}
\]
where $c_\sigma=|\bSigma|^{1/(2k)}/\sigma_P$ and $\bu=\bz/\|\bz\|$.
Because $\E_{\mathbf{0},\bI_k}\left[\bu\bu^T\right]=(1/k)\bI_k$, according to Lemma~\ref{lem:Lemma 5.1},
the right hand side is equal to
\[
\begin{split}
\frac{k\sigma_P}{|\bSigma|^{1/k}}
\E_{\mathbf{0},\bI_k}
\left[
u_1'\left(c_\sigma\|\bz\|\right)
(c_\sigma\|\bz\|)^3
\right]
\bSigma^{1/2}
\E_{\mathbf{0},\bI_k}
\left[
\bu\bu^T
\right]
\bSigma^{1/2}&\\
-
\frac{\sigma_P}{|\bSigma|^{1/k}}
\E_{\mathbf{0},\bI_k}
\left[
u_1'\left(c_\sigma\|\bz\|\right)
(c_\sigma\|\bz\|)^3
\right]
\bSigma
&=
\mathbf{0}.
\end{split}
\]
We conclude that 
\begin{equation}
\label{eq:Dgamma-sigma=0}
\bD_{\bgamma,\sigma}
=
\frac{\partial\Lambda_{\bgamma}(\bxi_P,\sigma_P)}{\partial\sigma}
=
\textbf{0}.
\end{equation}
Together with~\eqref{eq:Dbeta-sigma=0} this proves the lemma.
\end{proof}

\paragraph*{Proof of Theorem~\ref{th:IF elliptical}}
\begin{proof}
Write $\bxi_P=(\bbeta_{1,P},\bgamma_P)=(\bbeta_1(P),\bgamma(P))$,
$\bV_P=\bV(\bgamma(P))$, and $\sigma_P=\sigma(P)$.
Because $(\bbeta_{1,P},\bV_P)=(\bbeta^*,\bSigma/|\bSigma|^{1/k})$,
we have
\[
d_{0,P}^2
=
(\by_0-\bX_0\bbeta_{1,P})^T\bV_P^{-1}(\by_0-\bX_0\bbeta_{1,P})
=
|\bSigma|^{1/k}\|\bz_0\|^2,
\]
where $\bz_0=\bSigma^{-1/2}(\by_0-\bX_0\bbeta^*)$.
This means that
\[
\Psi_{\bbeta}(\bs_0,\bxi_P,\sigma_P)
=
u_1\left(\frac{d_{0,P}}{\sigma_P}\right)
\bX_0^T\bV_P^{-1}(\by_0-\bX_0\bbeta^*)
=
|\bSigma|^{1/k}
u_1\left(c_\sigma\|\bz_0\|\right)
\bX_0^T\bSigma^{-1/2}\bz_0,
\]
where $c_\sigma=|\bSigma|^{1/(2k)}/\sigma_P$.
From Lemmas~\ref{lem:change order int diff Lambda} and~\ref{lem:Lambda derivative}, 
we have that~$\Lambda_{\bbeta}$ is continuously differentiable at~$(\bxi_P,\sigma_P)$, 
with a derivative given by
$\bD_{\bbeta}
=
-\alpha_1
|\bSigma|^{1/k}
\mathbb{E}\left[\mathbf{X}^T\bSigma^{-1}\mathbf{X}\right]$,
which is non-singular according to Lemma~\ref{lem:inverse}.
Because $\bbeta_{1,P}=\bbeta^*$ is a point of symmetry, from Theorem~\ref{th:point of symmetry} 
we obtain
\[
\begin{split}
\text{\rm IF}(\bs_0;\bbeta_1,P)
&=
-\bD_{\bbeta}^{-1}\Psi_{\bbeta}(\bs_0;\bxi_P,\sigma_P)\\
&=
\frac{u_1\left(c_\sigma\|\bz_0\|\right)}{\alpha_1}
\left(
\E\left[
\bX^T\bSigma^{-1}\bX
\right]
\right)^{-1}
\bX_0^T\bSigma^{-1/2}\bz_0.
\end{split}
\]
This proves part one.

For part two, we apply Lemma~\ref{lem:IF xi sigma}(i).
Consider $\Lambda$ as defined in~\eqref{def:Lambda} with~$\Psi$ defined in~\eqref{def:Psi linear}.
From Lemmas~\ref{lem:change order int diff Lambda} and~\ref{lem:Lambda derivative}, 
we have that~$\Lambda$ is continuously differentiable at~$(\bxi_P,\sigma_P)$, 
with a derivative given by
\[
\bD_{\bxi}
=
\left(
  \begin{array}{cc}
\bD_{\bbeta} & \mathbf{0}\\
    \\
\mathbf{0} & \bD_{\bgamma} \\
  \end{array}
\right),
\]
where $\bD_{\bbeta}$ and $\bD_{\bgamma}$ are given in~\eqref{def:derivative Lambda beta} and~\eqref{def:derivative Lambda gamma}.
According to Lemma~\ref{lem:inverse},
$\bD_{\bbeta}$ and $\bD_{\bgamma}$ are non-singular, which implies that $\bD_{\bxi}$ is non-singular.
Furthermore, from Lemma~\ref{lem:Dsigma=0} we have 
$\bD_\sigma=\partial\Lambda(\bxi_P,\sigma_P)/\partial\sigma=\mathbf{0}$.
From Lemma~\ref{lem:IF xi sigma}(i),
this means that
\[
\text{\rm IF}(\bs_0;\bgamma,P)
=
-\bD_{\bgamma}^{-1}
\Psi_{\bgamma}(\bs_0;\bxi_P,\sigma_P),
\]
where~$\bD_{\bgamma}^{-1}$ is given in Lemma~\ref{lem:inverse}.
The remaining derivation of the expression for $\text{IF}(\bs_0;\bgamma,P)$ runs along the same line of reasoning as in
(the second part of) the proof of Corollary~8.4 in Lopuha\"a \textit{et al}~\cite{lopuhaa-gares-ruizgazen2023}.
Using that $\log|\bV_P|=0$,
we find that 
\begin{equation}
\label{eq:decompose IF gamma}
\begin{split}
&
\text{IF}(\bs_0;\bgamma,P)\\
&=
-
\bD_{\bgamma}^{-1}
\vc\left(
\Psi_{\bgamma}(\bs_0;\bxi_P,\sigma_P)
\right)\\
&=
\bD_{\bgamma}^{-1}
\bL^T
(\bV_P^{-1}\otimes\bV_P^{-1})
\vc\left\{
ku_1\left(\frac{d_{0,P}}{\sigma_P}\right)
\be_{0,P}\be_{0,P}^T
-
v_1\left(\frac{d_{0,P}}{\sigma_P}\right)\sigma_P^2\bV_P
\right\}\\
&=
k|\bSigma|^{2/k}
u_1\left(c_\sigma\|\bz_0\|\right)
\bD_{\bgamma}^{-1}
\bL^T
\left(\bSigma^{-1}\otimes\bSigma^{-1}\right)
\vc\left\{
\be_0^*(\be_0^*)^T
\right\}\\
&\qquad-
\sigma_P^2|\bSigma|^{1/k}
v_1\left(c_\sigma\|\bz_0\|\right)
\bD_{\bgamma}^{-1}
\bL^T
\left(\bSigma^{-1}\otimes\bSigma^{-1}\right)
\vc(\bSigma),
\end{split}
\end{equation}
where $d_{0,P}^2=\be_{0,P}^T\bV_P^{-1}\be_{0,P}$ with $\be_{0,P}=\by_0-\bX_0\bbeta_{1,P}$,
and where $\bz_0=\bSigma^{-1/2}\be_0^*$ with $\be_0^*=\by_0-\bX_0\bbeta^*$, and
where $c_\sigma=\sigma(P)/|\bSigma|^{1/(2k)}$. 
Consider the first term on the right hand side of~\eqref{eq:decompose IF gamma}.
We have that
\[
\begin{split}
\bD_{\bgamma}^{-1}
\bL^T
(\bSigma^{-1}\otimes\bSigma^{-1})
\vc\left\{
\be_0^*(\be_0^*)^T
\right\}
=
\bD_{\bgamma}^{-1}
\bL^T
(\bSigma^{-1/2}\otimes\bSigma^{-1/2})
\vc\left(\bz_0\bz_0^T\right).
\end{split}
\]
From Lemma~\ref{lem:inverse} we obtain
\begin{equation}
\label{eq:DgammaE}
\begin{split}
&
\bD_{\bgamma}^{-1}
\bL^T
(\bSigma^{-1/2}\otimes\bSigma^{-1/2})\\
&=
a(\bE^T\bE)^{-1}
\bE^T
+
b(\bE^T\bE)^{-1}
\bE^T\vc(\bI_k)\vc(\bI_k)^T\bE
(\bE^T\bE)^{-1}
\bE^T,
\end{split}
\end{equation}
where $\bE=(\bSigma^{-1/2}\otimes\bSigma^{-1/2})\bL$,
$a=1/\omega_1$, and $b=\omega_2/(\omega_1(\omega_1-k\omega_2))$.
This implies that
\begin{equation}
\label{eq:expression in E par2}
\begin{split}
&
\bD_{\bgamma}^{-1}
\bL^T
(\bSigma^{-1}\otimes\bSigma^{-1})
\vc\left(
\be_0^*(\be_0^*)^T
\right)\\
&\quad=
a(\bE^T\bE)^{-1}
\bE^T
\vc\left(\bz_0\bz_0^T\right)\\
&\qquad+
b
(\bE^T\bE)^{-1}
\bE^T\vc(\bI_k)\vc(\bI_k)^T\bE
(\bE^T\bE)^{-1}
\bE^T
\vc\left(\bz_0\bz_0^T\right).
\end{split}
\end{equation}
The first term on the right hand side of~\eqref{eq:expression in E par2} is equal to
\[
\begin{split}
a
\Big(\bL^T(\bSigma^{-1}\otimes\bSigma^{-1})\bL)\Big)^{-1}
\bL^T\left(\bSigma^{-1/2}\otimes\bSigma^{-1/2}\right)
\vc\left(\bz_0\bz_0^T\right).
\end{split}
\]
Since $\bV$ has a linear structure, we have $\vc(\bSigma)=|\bSigma|^{1/k}\vc(\bV_P)=|\bSigma|^{1/k}\bL\bgamma_P$.
This means that
\begin{equation}
\label{eq:E from theta to V}
\begin{split}
\bE\,\bgamma_P
&=
\left(\bSigma^{-1/2}\otimes\bSigma^{-1/2}\right)\bL\,
\bgamma_P\\
&=
\left(\bSigma^{-1/2}\otimes\bSigma^{-1/2}\right)
\vc(\bSigma)/|\bSigma|^{1/k}
=
\vc(\bI_k)/|\bSigma|^{1/k},
\end{split}
\end{equation}
and
\begin{equation}
\label{eq:E from V to theta}
(\bE^T\bE)^{-1}\bE^T\vc(\bI_k)=|\bSigma|^{1/k}\bgamma_P.
\end{equation} 
It follows that the second term on the right hand side of~\eqref{eq:expression in E par2} is equal to
\[
\begin{split}
b
\left(|\bSigma|^{1/k}\bgamma_P\right)
\left(|\bSigma|^{1/k}\bgamma_P\right)^T
\bE^T\vc\left(\bz_0\bz_0^T\right)
&=
b|\bSigma|^{1/k}
\bgamma_P
\vc(\bI_k)^T\vc\left(\bz_0\bz_0^T\right)\\
&=
b|\bSigma|^{1/k}
\bgamma_P
\text{tr}\left(\bz_0\bz_0^T\right)
=
b|\bSigma|^{1/k}
\bgamma_P
\|\bz_0\|^2.\\
\end{split}
\]
It follows that the first term on the right hand side of~\eqref{eq:decompose IF gamma} is equal to
\[
\begin{split}
&
ka
|\bSigma|^{2/k}
u_1\left(c_\sigma\|\bz_0\|\right)
\Big(\bL^T(\bSigma^{-1}\otimes\bSigma^{-1})\bL\Big)^{-1}
\bL^T\left(\bSigma^{-1/2}\otimes\bSigma^{-1/2}\right)
\vc\left(\bz_0\bz_0^T\right)\\
&\qquad\qquad+
kb
|\bSigma|^{3/k}
u_1\left(c_\sigma\|\bz_0\|\right)
\|\bz_0\|^2\bgamma_P.
\end{split}
\]
Next, consider the second term on the right hand side of~\eqref{eq:decompose IF gamma}.
From~\eqref{eq:DgammaE},
together with~\eqref{eq:E from theta to V} and~\eqref{eq:E from V to theta},
we have 
\[
\begin{split}
&
\bD_{\bgamma}^{-1}
\bL^T
(\bSigma^{-1}\otimes\bSigma^{-1})
\vc(\bSigma)\\
&\quad=
\bD_{\bgamma}^{-1}
\bL^T
(\bSigma^{-1/2}\otimes\bSigma^{-1/2})
\vc(\bI_k)\\
&\quad=
a(\bE^T\bE)^{-1}
\bE^T
\vc\left(\bI_k\right)
+
b
(\bE^T\bE)^{-1}
\bE^T\vc(\bI_k)
\vc(\bI_k)^T\bE
(\bE^T\bE)^{-1}
\bE^T
\vc\left(\bI_k\right)\\
&\quad=
a|\bSigma|^{1/k}\bgamma_P
+
b(|\bSigma|^{1/k}\bgamma_P)
\vc(\bI_k)^T\bE
(\bE^T\bE)^{-1}
\bE^T
\vc\left(\bI_k\right)\\
&\quad=
a|\bSigma|^{1/k}\bgamma_P
+
b(|\bSigma|^{1/k}\bgamma_P)
|\bSigma|^{1/k}\bgamma_P^T
\bE^T\vc(\bI_k)\\
&\quad=
a|\bSigma|^{1/k}\bgamma_P
+
b(|\bSigma|^{1/k}\bgamma_P)
\vc(\bI_k)^T\vc(\bI_k)\\
&\quad=
(a+kb)
|\bSigma|^{1/k}\bgamma_P.
\end{split}
\]
It follows that the second term on the right hand side of~\eqref{eq:decompose IF gamma} is equal to
\[
-
\sigma_P^2|\bSigma|^{2/k}
v_1\left(c_\sigma\|\bz_0\|\right)
(a+kb)
\bgamma_P.
\]
Putting things together, we find that $\mathrm{IF}(\bs_0;\bgamma,P)$ is equal to
\[
\begin{split}
&
ka
|\bSigma|^{2/k}
u_1\left(c_\sigma\|\bz_0\|\right)
\Big(\bL^T(\bSigma^{-1}\otimes\bSigma^{-1})\bL\Big)^{-1}
\bL^T\left(\bSigma^{-1/2}\otimes\bSigma^{-1/2}\right)
\vc\left(\bz_0\bz_0^T\right)\\
&\qquad+
kb
|\bSigma|^{3/k}
u_1\left(c_\sigma\|\bz_0\|\right)
\|\bz_0\|^2\bgamma_P
-
\sigma_P^2
|\bSigma|^{2/k}
v_1\left(c_\sigma\|\bz_0\|\right)
(a+kb)
\bgamma_P.
\end{split}
\]
The term with $\bgamma_P$ has coefficient
\[
\begin{split}
&
kb
|\bSigma|^{3/k}
u_1\left(c_\sigma\|\bz_0\|\right)
\|\bz_0\|^2
-
\sigma_P^2
|\bSigma|^{2/k}
v_1\left(c_\sigma\|\bz_0\|\right)
(a+kb)\\
&\quad=
kb
\sigma_P^2
|\bSigma|^{2/k}
v_1\left(c_\sigma\|\bz_0\|\right)
-
\sigma_P^2
|\bSigma|^{2/k}
v_1\left(c_\sigma\|\bz_0\|\right)
(a+kb)\\
&\quad=
-
a\sigma_P^2
|\bSigma|^{2/k}
v_1\left(c_\sigma\|\bz_0\|\right),
\end{split}
\]
using that $v_1(s)=u_1(s)s^2$.
Note that
$a=1/\omega_1=1/(\sigma_P^2|\bSigma|^{2/k}\gamma_1)$.
It follows that 
\[
\begin{split}
\mathrm{IF}(\bs_0;\bgamma,P)
&=
\frac{ku_1\left(c_\sigma\|\bz_0\|\right)}{\sigma_P^2\gamma_1}
\Big(\bL^T(\bSigma^{-1}\otimes\bSigma^{-1})\bL\Big)^{-1}
\bL^T\left(\bSigma^{-1/2}\otimes\bSigma^{-1/2}\right)
\vc\left(\bz_0\bz_0^T\right)\\
&\qquad-
\frac{v_1\left(c_\sigma\|\bz_0\|\right)}{\gamma_1}
\bgamma_P.
\end{split}
\]
Finally, due to linearity of $\bV$,
we have 
\begin{equation}
\label{eq:relation gamma theta*}
\bL\bgamma_P
=
\bV(\bgamma_P)
=
\frac{\bSigma}{|\bSigma|^{1/k}}
=
\frac{\bV(\btheta^*)}{|\bSigma|^{1/k}}
=
\frac{\bL\btheta^*}{|\bSigma|^{1/k}}.
\end{equation}
Since $\bL$ has full rank, we can multiply from the left by $(\bL^T\bL)^{-1}\bL^T$,
which implies that~$\bgamma_P=\btheta^*/|\bSigma|^{1/k}$.
This proves the theorem.
\end{proof}

\paragraph*{Proof of Corollary~\ref{cor:IF theta1}}
\begin{proof}
Write $(\bgamma_P,\sigma_P)=(\bgamma(P),\sigma(P))$
and $(\bgamma_{h,\bs_0},\sigma_{h,\bs_0})=(\bgamma(P_{h,\bs_0}),\sigma(P_{h,\bs_0}))$.
Because $\btheta_1(P)$ is a solution of~\eqref{def:theta1} and
$\bV$ has a linear structure, it follows that
\begin{equation}
\label{eq:relation theta1 gamma}
\bL\btheta_1(P)
=
\bV(\btheta_1(P))
=
\sigma^2_P\bV(\bgamma_P)
=
\sigma^2_P\bL\bgamma_P.
\end{equation}
Since $\bL$ has full rank, we can multiply from the left by $(\bL^T\bL)^{-1}\bL^T$,
which implies that~$\btheta_1(P)=\sigma^2_P\bgamma_P$ and 
a similar property holds for $\btheta_1(P_{h,\bs_0})$.
We find that
\[
\begin{split}
\btheta_1(P_{h,\bs_0})-\btheta_1(P)
&= 
\sigma^2_{h,\bs_0}
\left(
\bgamma_{h,\bs_0}-\bgamma_P
\right)
+
\bgamma_P
\left(
\sigma^2_{h,\bs_0}-\sigma^2_P
\right)\\
&=
\sigma^2_{h,\bs_0}
\left(
\bgamma_{h,\bs_0}-\bgamma_P
\right)
+
\bgamma_P
\left(
\sigma_{h,\bs_0}+\sigma_P
\right)
\left(
\sigma_{h,\bs_0}-\sigma_P
\right).
\end{split}
\]
According to Lemma~\ref{lem:IF sigma elliptical} and Theorem~\ref{th:IF elliptical},
$\text{IF}(\bs_0;\sigma,P)$ and $\text{IF}(\bs_0;\bgamma,P)$ exist.
Together with $\sigma_{h,\bs_0}\to\sigma_P$, we obtain
\[
\text{IF}(\bs_0;\btheta_1,P)
=
\sigma^2_P
\text{IF}(\bs_0;\bgamma,P)
+
2\sigma_P\bgamma_P
\text{IF}(\bs_0;\sigma,P).
\]
Because $\bV$ has a linear structure, from~\eqref{eq:relation gamma theta*},
we have $\bgamma_P=\btheta^*/|\bSigma|^{1/k}$,
after multiplication from the left by $(\bL^T\bL)^{-1}\bL^T$.
The corollary now follows from the expressions obtained in Lemma~\ref{lem:IF sigma elliptical} and Theorem~\ref{th:IF elliptical}.
\end{proof}

\subsection{Influence functions of covariance MM-functionals}
We provide some details about the influence functions of the covariance MM-functionals
in Examples~\ref{ex:LME model}, \ref{ex:multivariate linear regression}, and~\ref{ex:multivariate location-scatter}
for the situation where the distribution $P$ satisfies~(E).
\setcounter{example}{0}
\begin{example}[Linear Mixed Effects model]
\label{ex:IF LME}
For the influence function of covariance MM-functionals in linear mixed effects models,
nothing seems to be available yet.
For model~\eqref{def:linear mixed effects model Copt}, the expression for the influence function of the 
variance component MM-functional now follows from Corollary~\ref{cor:IF theta1} and
the one for the covariance MM-functional and the corresponding shape component follow from~\eqref{eq:IF Vtheta1}
and~\eqref{eq:IF Vgamma}, respectively.
From~\eqref{def:V linear} and~\eqref{def:L}, it follows that 
\begin{equation}
\label{eq:L for LME}
\bL
=
\Big[
  \begin{array}{cccc}
    \vc\left(\bI_k\right) & \vc\left(\bZ_1\bZ_1^T\right) 
 & \cdots &     \vc\left(\bZ_r\bZ_r^T\right) \\
  \end{array}
\Big].
\end{equation}
Furthermore, it can be seen that $\bL^T(\bSigma^{-1}\otimes\bSigma^{-1})\bL$ is equal to the matrix $\bQ$
with entries
\begin{equation}
\label{eq:entries Q}
Q_{ij}=\text{tr}\left(\bZ_i\bZ_i^T\bSigma^{-1}\bZ_j\bZ_j^T\bSigma^{-1}\right),
\quad
i,j=0,1,\ldots,r,
\end{equation}
where $\bZ_0=\bI_k$, and that 
$\bL^T\left(\bSigma^{-1/2}\otimes\bSigma^{-1/2}\right)\vc\left(\bz_0\bz_0^T\right)$ is equal to the vector $\bU$
with coordinates
\[
U_i=(\by_0-\bX_0\bbeta^*)^T\bSigma^{-1}\bZ_i\bZ_i^T\bSigma^{-1}(\by_0-\bX_0\bbeta^*),
\quad
i=0,1,\ldots,r,
\]
This implies that the influence function of the variance component MM-functional is given by
\[
\text{\rm IF}(\bs_0,\btheta_1,P)
=
\alpha_C(c_\sigma\|\bz_0\|)\bQ^{-1}\bU
-
\beta_C(c_\sigma\|\bz_0\|)\btheta^*,
\]
with $\alpha_C$ and $\beta_C$ defined in~\eqref{def:alphaC betaC},
and
where $\bz_0=\bSigma^{-1/2}(\by_0-\bX_0\bbeta^*)$ and $c_\sigma=|\bSigma|^{1/(2k)}/\sigma(P)$.
From~\eqref{eq:IF Vtheta1} we find that the influence functional of the covariance MM-functional is given by
\[
\text{\rm IF}(\bs_0,\vc(\bV(\btheta_1)),P)
=
\alpha_C(c_\sigma\|\bz_0\|)\bL\bQ^{-1}\bU
-
\beta_C(c_\sigma\|\bz_0\|)\vc(\bSigma).
\]
From~\eqref{eq:IF Vgamma}, we obtain the influence function of the shape MM-functional
\[
\text{\rm IF}(\bs_0,\bGamma(\btheta_1),P)
=
\frac{\alpha_C(c_\sigma\|\bz_0\|)}{\sigma^2(P)}
\left\{
\bL\bQ^{-1}\bU
-
\frac{\|\bz_0\|^2}{k}
\vc(\bSigma)
\right\},
\]
and in view of Remark~\ref{rem:IF Vgamma and gamma}, from Theorem~\ref{th:IF elliptical} we obtain
the influence function of the direction MM-functional
\[
\text{\rm IF}(\bs_0,\bgamma,P)
=
\frac{\alpha_C(c_\sigma\|\bz_0\|)}{\sigma^2(P)}
\left\{
\bQ^{-1}\bU
-
\frac{\|\bz_0\|^2}{k}
\btheta^*
\right\}.
\]
When $c_\sigma=1$, then the influence function of the covariance shape MM-functional 
coincides with that of the shape component of the covariance S-functional defined with $\rho_1$,
and similarly for the direction component of the variance component S-functional
defined with $\rho_1$, see Lopuha\"a \emph{et al}~\cite{lopuhaa-gares-ruizgazen2023}.
\end{example}

\begin{example}[Multivariate Linear Regression]
For the multivariate linear regression model~\eqref{def:multivariate linear regression model}, 
Kudraszow and Maronna~\cite{kudraszow-maronna2011} do not consider the
influence function of the covariance MM-functional.
In this model, the matrix $\bL$ is equal to the duplication matrix $\mathcal{D}_k$,
which satisfies
\begin{equation}
\label{eq:prop duplication}
\mathcal{D}_k
\left(
\mathcal{D}_k^T(\bSigma^{-1}\otimes\bSigma^{-1}\mathcal{D}_k
\right)^{-1}
\mathcal{D}_k^T
=
\frac12
\left(
\bI_{k^2}+\bK_{k,k}
\right)(\bSigma\otimes\bSigma),
\end{equation}
(e.g., see Magnus and Neudecker~\cite[Ch.~3, Sec.~8]{magnus&neudecker1988}).
Together with~\eqref{eq:IF Vtheta1} we find that the expression for
the covariance MM-functional is given by
\begin{equation}
\label{eq:IF cov MM mult lin regr}
\begin{split}
\text{\rm IF}(\bs_0,\bV(\btheta_1),P)
&=
\alpha_C\left(c_\sigma\|\bz_0\|\right)
(\by_0-\bB^T\bx_0)(\by_0-\bB^T\bx_0)^T
-
\beta_C(c_\sigma\|\bz_0\|)
\bSigma,
\end{split}
\end{equation}
with $\alpha_C$ and $\beta_C$ defined in~\eqref{def:alphaC betaC},
$\bz_0=\bSigma^{-1/2}(\by_0-\bB^T\bx_0)$, and $c_\sigma=|\bSigma|^{1/(2k)}/\sigma(P)$.
From~\eqref{eq:IF Vgamma}, it follows that 
the influence function of the shape MM-functional is given by
\begin{equation}
\label{eq:IF shape MM mult lin regr}
\text{IF}(\bs_0;\bGamma(\btheta_1),P)
=
\frac{\alpha_C\left(c_\sigma\|\bz_0\|\right)}{\sigma^2(P)}
\left\{
(\by_0-\bB^T\bx_0)(\by_0-\bB^T\bx_0)^T
-
\frac{\|\bz_0\|^2}{k}\bSigma
\right\}.
\end{equation}
When $c_\sigma=1$, 
this also coincides with the influence function of
shape component of the covariance S-functional defined with~$\rho_1$ in the 
multivariate linear regression model.
This is confirmed by application of formula~(8.3) in Kent and Tyler~\cite{kent&tyler1996}
to the expression found in Theorem~2 of Van Aelst and Willems~\cite{vanaelst&willems2005}.
\end{example}

\begin{example}[Multivariate Location and Scatter]
For the multivariate location-scatter model, we also have $\bL=\mathcal{D}_k$.
Since this model is a special case of the multivariate linear regression model~\eqref{def:multivariate linear regression model}
by taking $\bx_i=1$ and $\bB^T=\bmu$,
the expression for the influence function of the covariance MM-functional can 
be obtained from~\eqref{eq:IF cov MM mult lin regr}
and the influence function of the covariance shape MM-functional from~\eqref{eq:IF shape MM mult lin regr},
by replacing $\bB^T\bx_0$ by $\bmu$.
When $c_\sigma=1$, this also coincides with the influence function of
shape component of the covariance S-functional defined with~$\rho_1$ in the multivariate location-scatter model.
This was already observed by Salibi\'an-Barrera \emph{et al}~\cite{SalibianBarrera-VanAelst-Willems2006}.
Finally, again there is a connection with the CM-functionals considered in Kent and Tyler~\cite{kent&tyler1996},
whose influence function depends on a parameter~$\lambda_0$.
By using~\eqref{eq:prop duplication}, one finds that the expression in~\eqref{eq:IF cov MM mult lin regr},
with $\bB^T\bx_0=\bmu$,
is similar to the expression for the influence function of the covariance CM-functional,
for the case that $\lambda_0=\lambda_L$ (see Kent and Tyler~\cite{kent&tyler1996} for details),
and they both coincide when $c_\sigma=1/\sqrt{\lambda_0}$ and $\rho_0=\rho_1$.
\end{example}

\subsection{Proofs for Section~\ref{sec:asymp norm}}
\begin{lemma}
\label{lem:asymp relation xi sigma}
Suppose that $\rho_1$ satisfies (R1)-(R4), such that $u_1(s)$ is of bounded variation,
and suppose that~$\bV$ satisfies~(V4).
Let $\sigma_n$ and $\sigma(P)$ be solutions of~\eqref{def:initial estimators} and~\eqref{def:sigma}, respectively,
and let $\bxi_n=(\bbeta_{1,n},\bgamma_n)$ and~$\bxi(P)=(\bbeta_1(P),\bgamma(P))$ be local minima of $R_n(\bbeta,\bV(\bgamma))$ and 
$R_P(\bbeta,\bV(\bgamma))$, respectively.
Suppose that $(\bxi_n,\sigma_n)\to(\bxi(P),\sigma(P))$, in probability.
Let $\Lambda$ be defined in~\eqref{def:Lambda} with~$\Psi$ defined in~\eqref{def:Psi}
Suppose that $\Lambda$ is continuously differentiable with 
a non-singular derivative $\bD_{\bxi}=\partial\Lambda/\partial\bxi$ 
and derivative $\bD_{\sigma}=\partial\Lambda/\partial\sigma$
at $(\bxi(P),\sigma(P))$,
and suppose that $\E\|\bs\|^2<\infty$.
Then
\[
\begin{split}
\bxi_n-\bxi(P)
&=
-\bD_{\bxi}^{-1}
\bigg\{
\bD_\sigma(\sigma_n-\sigma(P))
+
\int \Psi(\mathbf{s},\bxi(P),\sigma(P))\,\dd (\mathbb{P}_n-P)(\mathbf{s})
\bigg\}\\
&\quad
+
o_P(\|\bxi_n-\bxi(P)\|)
+
o_P(|\sigma_n-\sigma(P)|)
+
o_P(1/\sqrt{n}),
\end{split}
\]
\end{lemma}
\begin{proof}
The proof is similar to that of Theorem~9.1 in Lopuha\"a \textit{et al}~\cite{lopuhaa-gares-ruizgazen2023}.
Due to~\eqref{eq:cond beta0 theta0}, we have that $\bxi_n=\bxi(\mathbb{P}_n)=(\bbeta_{1,n},\bgamma_{n})$ and $\sigma_n=\sigma(\mathbb{P}_n)$.
This means that $(\bxi_n,\sigma_n)$ satisfies score equation~\eqref{eq:Psi=0} for the MM-functionals at $\mathbb{P}_n$, that is
\[
\int \Psi(\bs,\bxi_n,\sigma_n)\,\dd \mathbb{P}_n(\bs)=\mathbf{0},
\]
with $\Psi$ defined in~\eqref{def:Psi}.
Writing $\bxi_P=(\bbeta_{1,P},\bgamma_P)=(\bbeta_1(P),\bgamma(P))$, we decompose as follows
\begin{equation}
\label{eq:decomposition estimator appendix}
\begin{split}
\mathbf{0}=
\int \Psi(\mathbf{s},\bxi_n,\sigma_n)\,\dd P(\mathbf{s})
&+
\int \Psi(\mathbf{s},\bxi_P,\sigma_P)\,\dd (\mathbb{P}_n-P)(\mathbf{s})\\
&+
\int
\left(
\Psi(\mathbf{s},\bxi_n,\sigma_n)-\Psi(\mathbf{s},\bxi_P,\sigma_P)
\right)
\,\dd (\mathbb{P}_n-P)(\mathbf{s}).
\end{split}
\end{equation}
We start by showing that the third term on the right hand side of~\eqref{eq:decomposition estimator appendix}
is of order $o_P(1/\sqrt{n})$.
Since $\Psi=(\Psi_{\bbeta},\Psi_{\bgamma})$, with $\Psi_{\bgamma}=(\Psi_{\bgamma,1},\ldots,\Psi_{\bgamma,l})$,
it suffices to show that
\begin{eqnarray}
\int
\left(
\Psi_{\bbeta}(\bs,\bxi_n,\sigma_n)-\Psi_{\bbeta}(\bs,\bxi_P,\sigma_P)
\right)
\,\dd (\mathbb{P}_n-P)(\bs)
&=&\label{eq:stoch equi vector}
o_P(1/\sqrt{n}),
\\
\int
\left(
\Psi_{\bgamma,j}(\bs,\bxi_n,\sigma_n)-\Psi_{\bgamma,j}(\bs,\bxi_P,\sigma_P)
\right)
\,\dd (\mathbb{P}_n-P)(\bs)
&=&\label{eq:stoch equi matrix}
o_P(1/\sqrt{n}),
\end{eqnarray}
for $j=1,\ldots,l$.
From~\eqref{def:Psi} we have that
\[
\Psi_{\bbeta}(\bs,\bxi,\sigma)
=
u_1\left(
\frac{d}{\sigma}\right)\bX^T\bV^{-1}(\by-\bX\bbeta)
=
u_1\left(d(\by,\bX\bbeta,\sigma^2\bV)\right)\bX^T\bV^{-1}(\by-\bX\bbeta).
\]
Since $\bV$ satisfies~(V1) we have that $\sigma_n^2\bV(\bgamma_n)\to \sigma_P^2\bV(\bgamma_P)$, in probability.
Then~\eqref{eq:stoch equi vector} follows from equation~(96) in the proof of Lemma~11.8 in Lopuha\"a \emph{et al}~\cite{supplement}.
From~\eqref{def:Psi}, we also have
\[
\Psi_{\bgamma,j}(\bs,\bxi,\sigma)
=
u_1\left(\frac{d}{\sigma}\right)
(\by-\bX\bbeta)^T\bV^{-1}
\bH_{1,j}
\bV^{-1}(\by-\bX\bbeta)
-
\mathrm{tr}\left(\bV^{-1}\frac{\partial \bV}{\partial \gamma_j}\right)\log|\bV|,
\]
for $j=1,\ldots,l$,
where $\bH_{1,j}=\bH_{1,j}(\bgamma)$ is defined in~\eqref{def:Hj}
Because $|\bV(\bgamma_n)|=1$, it follows that
\[
\begin{split}
\Psi_{\bgamma,j}(\bs,\bxi_n,\sigma_n)
=
u_1\left(d_n\right)
(\by-\bX\bbeta_{1,n})^T\bV(\bgamma_n)^{-1}
\bH_{1,j}(\bgamma_n)
\bV(\bgamma_n)^{-1}(\by-\bX\bbeta_{1,n}),
\end{split}
\]
where $d_n^2=(\by-\bX\bbeta_{1,n})^T(\sigma_n^2\bV(\bgamma_n))^{-1}(\by-\bX\bbeta_{1,n})$
and similarly for $\Psi_{\bgamma,j}(\bs,\bxi_P,\sigma_P)$.
Because $\bV$ satisfies~(V4) we have that $\sigma_n^2\bV(\bgamma_n)\to \sigma_P^2\bV(\bgamma_P)$ and 
$\bH_{1,j}(\bgamma_n)\to\bH_{1,j}(\bgamma_P)$, in probability.
This implies that~\eqref{eq:stoch equi matrix} follows from equation~(97) in the proof of Lemma~11.8 in Lopuha\"a \emph{et al}~\cite{supplement}.

Then, from~\eqref{eq:decomposition estimator appendix} we can write
\[
\mathbf{0}
=
\Lambda(\bxi_n,\sigma_n)
+
\int \Psi(\mathbf{s},\bxi_P,\sigma_P)\,\dd (\mathbb{P}_n-P)(\mathbf{s})\\
+
o_P(1/\sqrt{n}).
\]
Because $\Lambda$ is continuously differentiable at
$(\bxi_P,\sigma_P)$, it follows that
\[
\begin{split}
\Lambda(\bxi_n,\sigma_n)
&=
\Lambda(\bxi_P,\sigma_n)
+
\frac{\partial\Lambda}{\partial\bxi}(\bxi_P,\sigma_n)
(\bxi_n-\bxi_P)+
o_P(\|\bxi_n-\bxi_P\|)\\
&=
\Lambda(\bxi_P,\sigma_n)
+
\left(
\frac{\partial\Lambda}{\partial\bxi}(\bxi_P,\sigma_P)+o_P(1)
\right)
(\bxi_n-\bxi_P)+
o_P(\|\bxi_n-\bxi_P\|).
\end{split}
\]
Furthermore, since $\bxi_P$ is a solution of~\eqref{eq:Psi=0}, we find
\[
\begin{split}
\Lambda(\bxi_P,\sigma_n)
&=
\Lambda(\bxi_P,\sigma_P)
+
\frac{\partial\Lambda}{\partial\sigma}(\bxi_P,\sigma_P)(\sigma_n-\sigma_P)
+
o_P(|\sigma_n-\sigma_P|)\\
&=
\frac{\partial\Lambda}{\partial\sigma}(\bxi_P,\sigma_P)(\sigma_n-\sigma_P)
+
o_P(|\sigma_n-\sigma_P|).
\end{split}
\]
Putting things together, we obtain
\[
\begin{split}
\mathbf{0}
=
\bD_\sigma(\sigma_n-\sigma_P)
&+
\bD_{\bxi}(\bxi_n-\bxi_P)
+
\int \Psi(\mathbf{s},\bxi_P,\sigma_P)\,\dd (\mathbb{P}_n-P)(\mathbf{s})\\
&+
o_P(\|\bxi_n-\bxi_P\|)
+
o_P(|\sigma_n-\sigma_P|)
+
o_P(1/\sqrt{n}).
\end{split}
\]
Because $\bD_{\bxi}$ is non-singular, it follows that
\[
\begin{split}
\bxi_n-\bxi_P
=
-\bD_{\bxi}^{-1}
\bigg\{
\bD_\sigma(\sigma_n-\sigma_P)
&+
\int \Psi(\mathbf{s},\bxi_P,\sigma_P)\,\dd (\mathbb{P}_n-P)(\mathbf{s})
\bigg\}\\
&+
o_P(\|\bxi_n-\bxi_P\|)
+
o_P(|\sigma_n-\sigma_P|)
+
o_P(1/\sqrt{n}).
\end{split}
\]
This proves the lemma.
\end{proof}

\begin{lemma}
\label{lem:asymp relation sigma zeta}
Suppose that $\rho_0$ satisfies~(R1)-(R2) and~$\bV$ satisfies~(V4).
Let $\bzeta_{0,n}=(\bbeta_{0,n},\btheta_{0,n})$ 
and~$\bzeta_0(P)=(\bbeta_0(P),\btheta_0(P))$ be the pairs of initial estimators and 
corresponding functionals, and let~$\sigma_n$ and~$\sigma(P)$ be solutions of~\eqref{def:initial estimators} and~\eqref{def:sigma}, respectively.
Suppose that $(\bzeta_{0,n},\sigma_n)\to(\bzeta_0(P),\sigma(P)$), in probability.
Let~$\Lambda_0$ be defined in~\eqref{def:Lambda0} with $\Psi_0$ defined in~\eqref{def:Psi0},
and suppose that $\Lambda_0$
is continuously differentiable with derivatives
$D_{0,\sigma}=\partial\Lambda_0/\partial\sigma\ne0$ 
and
$\bD_{0,\bzeta}=\partial\Lambda_0/\partial\bzeta$
at~$(\bzeta(P),\sigma(P))$.
Then 
\[
\begin{split}
\sigma_n-\sigma(P)
&=
-D_{0,\sigma}^{-1}
\bigg\{
\bD_{0,\bzeta}^T(\bzeta_{0,n}-\bzeta_0(P))
+
\int \Psi_0(\mathbf{s},\bzeta_0(P),\sigma(P))\,\dd (\mathbb{P}_n-P)(\mathbf{s})
\bigg\}\\
&\quad+
o(\|\bzeta_{0,n}-\bzeta_0(P)\|)
+
o_P(|\sigma_n-\sigma(P)|)
+
o_P(1/\sqrt{n}).
\end{split}
\]
\end{lemma}
\begin{proof}
Denote $\bzeta_{0,n}=(\bbeta_{0,n},\btheta_{0,n})$ and 
let $\sigma_n$ and $\sigma(P)$ be solutions of~\eqref{def:initial estimators} and~\eqref{def:sigma}, respectively.
Then, according to~\eqref{def:initial estimators}, $(\bzeta_{0,n},\sigma_n)$ satisfies 
\[
\int \Psi_0(\bs,\bzeta_{0,n},\sigma_n)\,\dd \mathbb{P}_n(\bs)=0.
\]
Writing $\bzeta_{0,P}=(\bbeta_{0,P},\btheta_{0,P})$ and $\sigma_P=\sigma(P)$, we decompose as follows
\begin{equation}
\label{eq:decomposition sigma}
\begin{split}
\mathbf{0}
=
\int \Psi_0(\mathbf{s},\bzeta_{0,n},\sigma_n)\,\dd P(\mathbf{s})
+
\int \Psi_0(\mathbf{s},\bzeta_{0,P},\sigma_P)\,\dd (\mathbb{P}_n-P)(\mathbf{s})\\
+
\int
\left(
\Psi_0(\mathbf{s},\bzeta_{0,n},\sigma_n)-\Psi_0(\mathbf{s},\bzeta_{0,P},\sigma_P)
\right)
\,\dd (\mathbb{P}_n-P)(\mathbf{s}).
\end{split}
\end{equation}
Consider the third term on the right hand side, where we can write
\[
\Psi_0(\bs,\bzeta,\sigma)
=
\rho_0
\left(
d(\by,\bX\bbeta,\sigma^2\bGamma(\btheta))
\right)
-
b_0,
\]
where $d(\by,\bX\bbeta,\sigma^2\bGamma(\btheta))$ is defined in~\eqref{def:Mahalanobis distance}
and $\bGamma(\btheta)=\bV(\btheta)/|\bV(\btheta)|^{1/k}$.
Because $\bV$ satisfies~(V4) we have that $\sigma_n^2\bGamma(\btheta_{0,n})\to\sigma_P^2\bGamma(\btheta_{0,P})$, in probability.
Since $\rho_0$ and $\bV$ satisfy the conditions needed to establish (98) in the proof of 
Lemma~11.8 in Lopuha\"a \textit{et al}~\cite{supplement}, it follows that
\[
\int
\left(
\Psi_0(\mathbf{s},\bzeta_{0,n},\sigma_n)-\Psi_0(\mathbf{s},\bzeta_{0,P},\sigma_P)
\right)
\,\dd (\mathbb{P}_n-P)(\mathbf{s})
=
o_P(1/\sqrt{n}).
\]
Then from~\eqref{eq:decomposition sigma} we can write
\[
0
=
\Lambda_0(\bzeta_{0,n},\sigma_n)
+
\int \Psi_0(\mathbf{s},\bzeta_{0,P},\sigma_P)\,\dd (\mathbb{P}_n-P)(\mathbf{s})\\
+
o_P(1/\sqrt{n}).
\]
Because the partial derivative $\partial\Lambda_0/\partial\sigma$  is continuous at
$(\bzeta_{0,P},\sigma_P)$, it follows that
\[
\begin{split}
\Lambda_0(\bzeta_{0,n},\sigma_n)
&=
\Lambda_0(\bzeta_{0,n},\sigma_P)
+
\frac{\partial\Lambda_0}{\partial\sigma}(\bzeta_{0,n},\sigma_P)
(\sigma_n-\sigma_P)+
o_P(|\sigma_n-\sigma_P|)\\
&=
\Lambda(\bzeta_{0,n},\sigma_P)
+
D_{0,\sigma}
(\sigma_n-\sigma_P)+
o_P(|\sigma_n-\sigma_P|).
\end{split}
\]
Furthermore, since $\sigma_P$ is a solution of~\eqref{def:sigma}, we find
\[
\begin{split}
\Lambda(\bzeta_{0,n},\sigma_P)
&=
\Lambda(\bzeta_{0,P},\sigma_P)
+
\frac{\partial\Lambda_0}{\partial\bzeta}(\bzeta_{0,P},\sigma_P)(\bzeta_{0,n}-\bzeta_{0,P})
+
o_P(\|\bzeta_{0,n}-\bzeta_{0,P}\|)\\
&=
\bD_{0,\bzeta}^T(\bzeta_{0,n}-\bzeta_{0,P})
+
o_P(\|\bzeta_{0,n}-\bzeta_{0,P}\|).
\end{split}
\]
Putting things together gives
\[
\begin{split}
0
=
\bD_{0,\bzeta}^T(\bzeta_{0,n}-\bzeta_{0,P})
&+
D_{0,\sigma}
(\sigma_n-\sigma_P)
+
\int \Psi_0(\mathbf{s},\bzeta_{0,P},\sigma_P)\,\dd (\mathbb{P}_n-P)(\mathbf{s})\\
&+
o_P(1/\sqrt{n})
+
o_P(\|\bzeta_{0,n}-\bzeta_{0,P}\|)
+
o_P(|\sigma_n-\sigma_P|).
\end{split}
\]
Because $D_{0,\sigma}\ne0$, it follows that
\[
\begin{split}
\sigma_n-\sigma_P
=
-D_{0,\sigma}^{-1}
\bigg\{
\bD_{0,\bzeta}^T(\bzeta_{0,n}-\bzeta_{0,P})
+
\int \Psi_0(\mathbf{s},\bzeta_{0,P},\sigma_P)\,\dd (\mathbb{P}_n-P)(\mathbf{s})
\bigg\}\\
+
o_P(1/\sqrt{n})
+
o_P(\|\bzeta_{0,n}-\bzeta_{0,P}\|)
+
o_P(|\sigma_n-\sigma_P|).
\end{split}
\]
This proves the lemma.
\end{proof}

\paragraph*{Proof of Theorem~\ref{th:asymp norm symmetry}}
\begin{proof}
Write $\bxi_P=(\bbeta_{1,P},\bgamma_P)=(\bbeta_1(P),\bgamma(P))$ and $\sigma_P=\sigma(P)$. 
Similar to~\eqref{eq:decomposition estimator appendix} we decompose as follows
\[
\begin{split}
\mathbf{0}=
\int \Psi_{\bbeta}(\mathbf{s},\bxi_n,\sigma_n)\,\dd P(\mathbf{s})
&+
\int \Psi{\bbeta}(\mathbf{s},\bxi_P,\sigma_P)\,\dd (\mathbb{P}_n-P)(\mathbf{s})\\
&+
\int
\left(
\Psi_{\bbeta}(\mathbf{s},\bxi_n,\sigma_n)-\Psi_{\bbeta}(\mathbf{s},\bxi_P,\sigma_P)
\right)
\,\dd (\mathbb{P}_n-P)(\mathbf{s}).
\end{split}
\]
From~\eqref{def:Psi} we have that
\[
\Psi_{\bbeta}(\bs,\bxi,\sigma)
=
u_1\left(
\frac{d}{\sigma}\right)\bX^T\bV^{-1}(\by-\bX\bbeta)
=
u_1\left(d(\by,\bX\bbeta,\sigma^2\bV)\right)\bX^T\bV^{-1}(\by-\bX\bbeta).
\]
Since $\bV$ satisfies~(V1) we have that $\sigma_n^2\bV(\bgamma_n)\to \sigma_P^2\bV(\bgamma_P)$, in probability.
Then similar to the proof of equation~(96) in Lemma~11.8 in Lopuha\"a \emph{et al}~\cite{supplement},
it follows
\[
\int
\left(
\Psi_{\bbeta}(\bs,\bxi_n,\sigma_n)-\Psi_{\bbeta}(\bs,\bxi_P,\sigma_P)
\right)
\,\dd (\mathbb{P}_n-P)(\bs)
=
o_P(1/\sqrt{n}).
\]
This means we can write
\begin{equation}
\label{eq:decomp Psi-beta general}
\mathbf{0}
=
\Lambda_{\bbeta}(\bxi_n,\sigma_n)
+
\int \Psi_{\bbeta}(\mathbf{s},\bxi_P,\sigma_P)\,\dd (\mathbb{P}_n-P)(\mathbf{s})\\
+
o_P(1/\sqrt{n}).
\end{equation}
Since $\partial\Lambda_{\bbeta}/\partial\bbeta$ is continuous at $(\bxi_P,\sigma_P)$
and $(\bgamma_n,\sigma_n)\to(\bgamma_P,\sigma_P)$, in probability, it follows that
\[
\begin{split}
\Lambda_{\bbeta}(\bxi_n,\sigma_n)
&=
\Lambda_{\bbeta}(\bbeta_{1,P},\gamma_n,\sigma_n)
+
\left(
\frac{\partial\Lambda_{\bbeta}}{\partial\bbeta}(\bbeta_{1,P},\bgamma_P,\sigma_P)
+
o_P(1)
\right)
(\bbeta_{1,n}-\bbeta_{1,P}).
\end{split}
\]
Because $\bbeta_{1,P}$ is a point of symmetry and $\Psi_{\bbeta}$ is an 
odd function of $\by-\bX\bbeta$, 
it follows that $\Lambda_{\bbeta}(\bbeta_{1,P},\bgamma_n,\sigma_n)=\mathbf{0}$,
so that 
\[
\Lambda_{\bbeta}(\bxi_n,\sigma_n)
=
\bD_{\bbeta}
(\bbeta_{1,n}-\bbeta_{1,P})
+
o(\|\bbeta_{1,n}-\bbeta_{1,P}\|).
\]
Together with~\eqref{eq:decomp Psi-beta general}, we obtain
\[
\begin{split}
\mathbf{0}
=
\bD_{\bbeta}
(\bbeta_n-\bbeta_{1,P})
&+
\int \Psi_{\bbeta}(\mathbf{s},\bxi_P,\sigma_P)\,\dd (\mathbb{P}_n-P)(\mathbf{s})
+
o(\|\bbeta_{1,n}-\bbeta_{1,P}\|)
+
o_P(1/\sqrt{n}).
\end{split}
\]
Since $\rho_1$ satisfies~(R2) and (R4),
according to Lemma~\ref{lem:Psi bounded} there exist a constant $C_1>0$ only depending on $P$ and $\rho_1$, such that
$\|\Psi_{\bbeta}(\bs;\bxi_P,\sigma_P)\|
\leq
C_1\|\bX\|$.
Since $\E\|\bX\|^2<\infty$, the central limit theorem applies to the second term on the right hand side.
From this, we first conclude that $\bbeta_{1,n}-\bbeta_{1,P}=O_P(1/\sqrt{n})$.
After inserting this in the previous equality and use that $\bD_{\bbeta}$ is non-singular, this implies
\[
\bbeta_{1,n}-\bbeta_{1,P}
=
-
\bD_{\bbeta}^{-1}
\int 
\Psi_{\bbeta}(\mathbf{s},\bxi_P,\sigma_P)
\,\dd (\mathbb{P}_n-P)(\mathbf{s})
+
o_P(1/\sqrt{n}).
\]
This finishes the proof.
\end{proof}

\begin{lemma}
\label{lem:asym norm sigma}
Suppose that $P$ satisfies~(E) for some $(\bbeta^*,\btheta^*)\in\R^q\times\bTheta$
and suppose that $\E\|\bX\|<\infty$.
Suppose that $\rho_0$ and $\bV$ satisfy (R2), (R4) and (V4), respectively.
Let $\bzeta_{0,n}=(\bbeta_{0,n},\btheta_{0,n})$ be the pair of initial estimators
and let $\bzeta_0=(\bbeta_0,\btheta_0)$ be the corresponding functional 
satisfying $(\bbeta_0(P),\btheta_0(P))=(\bbeta^*,\btheta^*)$, and suppose that $\bzeta_{0,n}-\bzeta_0(P)=O_P(1/\sqrt{n})$.
Let~$\sigma_n$ be a solution of~\eqref{def:initial estimators} 
and suppose that $\sigma_n\to\sigma(P)$, in probability,
where~$\sigma(P)$ is a solution of~\eqref{def:sigma}.
Suppose that $\E_{\mathbf{0},\bI_k}[\rho_0'(c_\sigma\|\bz\|)c_\sigma\|\bz\|]>0$,
where $c_\sigma=|\bSigma|^{1/(2k)}/\sigma(P)$.
Then~$\sqrt{n}(\sigma_n-\sigma(P))$ is asymptotically normal with mean zero and variance
\[
\frac{\sigma^2(P)\E\left[(\rho_0(c_\sigma\|\bz\|)-b_0)^2\right]}{\big(\E_{\mathbf{0},\bI_k}[\rho_0'(c_\sigma\|\bz\|)c_\sigma\|\bz\|]\big)^2}.
\]
\end{lemma}
\begin{proof}
We apply Lemma~\ref{lem:asymp relation sigma zeta}.
Consider $\Lambda_0$ as defined in~\eqref{def:Lambda0} with $\Psi_0$ from~\eqref{def:Psi0}.
From Lemma~\ref{lem:change order int diff Lambda0}, we have that
$\Lambda_0$ is continuously differentiable at $(\bzeta_0(P),\sigma(P))$, 
with derivatives~$\bD_{0,\bzeta}=\mathbf{0}$ and
$D_{0,\sigma}=-\E_{\mathbf{0},\bI_k}[\rho_0'(c_\sigma\|\bz\|)c_\sigma\|\bz\|]/\sigma(P)<0$,
according to Lemma~\ref{lem:D0zeta=0}.
Since $\bzeta_{0,n}-\bzeta_0(P)=O_P(1/\sqrt{n})$,
from Lemma~\ref{lem:asymp relation sigma zeta}, it then follows that
\begin{equation}
\label{eq:expansion sigma}
\begin{split}
\sigma_n-\sigma(P)
=
-D_{0,\sigma}^{-1}
\int \Psi_0(\mathbf{s},\bzeta_0(P),\sigma(P))\,\dd (\mathbb{P}_n-P)(\mathbf{s})\\
+
o_P(\sigma_n-\sigma(P))
+
o_P(1/\sqrt{n}),
\end{split}
\end{equation}
Since $\rho_0$ is bounded, the central limit theorem applies 
to the first term on the right hand side of~\eqref{eq:expansion sigma}.
We first conclude that $\sigma_n-\sigma(P)=O_P(1/\sqrt{n})$ and
after inserting this in~\eqref{eq:expansion sigma},
we find that 
$\sqrt{n}(\sigma_n-\sigma(P))$ is asymptotically normal with mean zero and variance
$\E\left[\Psi_0(\mathbf{s},\bzeta_0(P),\sigma(P))^2\right]/D_{0,\sigma}^2$.
Since $(\bbeta_0(P),\btheta_0(P))=(\bbeta^*,\btheta^*)$ and 
$\bGamma(\btheta_0(P))=\bV(\btheta^*)/|\bV(\btheta^*)|^{1/k}=\bSigma/|\bSigma|^{1/k}$,
it follows that
\[
\E\left[
\Psi_0(\mathbf{s},\bzeta_0(P),\sigma(P))^2
\right]
=
\E
\left[
\E\left[
\left(
\rho_0
\left(
\frac{d_{\Gamma}^*}{\sigma(P)}\right)
-b_0
\right)^2
\bigg|
\bX
\right]
\right],
\]
where $d_{\Gamma}^*$ is defined in~\eqref{eq:dGamma*}.
The inner expectation on the right hand side
is the conditional expectation of $\by\mid\bX$,
which has the same distribution as 
$\bSigma^{1/2}\bz+\bX\bbeta^*$, where $\bz$ has spherical density~$f_{\mathbf{0},\bI_k}$.
This implies that the inner expectation on the right hand side is equal to
$\E_{\mathbf{0},\bI_k}
[(\rho_0(c_\sigma\|\bz\|)-b_0)^2]$,
where $c_\sigma=|\bSigma|^{1/(2k)}/\sigma(P)$.
This proves the lemma.
\end{proof}

\paragraph*{Proof of Theorem~\ref{th:asymp norm elliptical}}
\begin{proof}
Write $\bxi_P=(\bbeta_{1,P},\bgamma_P)=(\bbeta_1(P),\bgamma(P))$,
$\bV_P=\bV(\bgamma_p)$, and $\sigma_P=\sigma(P)$.
We apply Lemma~\ref{lem:asymp relation xi sigma}.
Consider $\Lambda$ as defined in~\eqref{def:Lambda} with~$\Psi$ as defined in~\eqref{def:Psi linear}.
From Lemmas~\ref{lem:change order int diff Lambda} and~\ref{lem:Lambda derivative},
we have that~$\Lambda$ is continuously differentiable at $(\bxi_P,\sigma_P)$ 
with derivative
\[
\bD_{\bxi}
=
\frac{\partial \Lambda(\bxi_P,\sigma_P)}{\partial \bxi}
=
\left(
  \begin{array}{cc}
\bD_{\bbeta} & \mathbf{0}\\
    \\
\mathbf{0} & \bD_{\bgamma} \\
  \end{array}
\right),
\]
where $\bD_{\bbeta}$ and $\bD_{\bgamma}$ are given in~\eqref{def:derivative Lambda beta} and~\eqref{def:derivative Lambda gamma}.
According to Lemma~\ref{lem:inverse}, we have that 
$\bD_{\bbeta}$ and $\bD_{\bgamma}$ are non-singular, so that $\bD_{\bxi}$ is non-singular.
Furthermore, from Lemma~\ref{lem:Dsigma=0}, we have that 
$\bD_{\sigma}=\partial\Lambda(\bxi_P,\sigma_P)/\partial\sigma=\mathbf{0}$.
Since $\sigma_n-\sigma_P=O_P(1/\sqrt{n})$, 
Lemma~\ref{lem:asymp relation xi sigma} yields that
\[
\begin{split}
\bxi_n-\bxi_P
&=
-\bD_{\bxi}^{-1}
\int 
\Psi(\mathbf{s},\bxi_P,\sigma_P)\,\dd (\mathbb{P}_n-P)(\mathbf{s})
+
o_P(\|\bxi_n-\bxi(P)\|)
+
o_P(1/\sqrt{n}).
\end{split}
\]
According to Lemma~\ref{lem:Psi bounded}, there exist constants $C_1,C_2>0$ only depending on $P$ and $\sigma(P)$, such that
$\left\|
\Psi(\bs;\bxi_P,\sigma_P)
\right\|
\leq
C_1+C_2\|\bX\|$.
Since $\E\|\bX\|^2<\infty$, the central limit theorem applies to the second term on the right hand side.
From this we first obtain that $\bxi_n-\bxi_P=O_P(1/\sqrt{n})$.
After inserting this in the previous equation, we find that
\begin{equation}
\label{eq:expansion xi}
\bxi_n-\bxi_P
=
-\bD_{\bxi}^{-1}
\int \Psi(\mathbf{s},\bxi_P,\sigma_P)\,\dd (\mathbb{P}_n-P)(\mathbf{s})
+
o_P(1/\sqrt{n}).
\end{equation}
It follows that $\sqrt{n}(\bxi_n-\bxi_P)$ is asymptotically normal with mean zero and variance
\begin{equation}
\label{eq:lim var bxi}
\bD_{\bxi}^{-1}
\E\left[
\Psi(\mathbf{s},\bxi_P,\sigma_P)
\Psi(\mathbf{s},\bxi_P,\sigma_P)^T
\right]
\bD_{\bxi}^{-1}.
\end{equation}
Because $\Psi_{\bbeta}(\mathbf{s},\bxi_P,\sigma_P)\Psi_{\bgamma}(\mathbf{s},\bxi_P,\sigma_P)^T$ is
an odd function of $\by-\bX\bbeta_{1,P}$ and~$\bbeta_{1,P}=\bbeta^*$ is a point of symmetry,
it follows that 
$\E[\Psi_{\bbeta}(\mathbf{s},\bxi_P,\sigma_P)\Psi_{\bgamma}(\mathbf{s},\bxi_P,\sigma_P)^T]=\mathbf{0}$.
Hence, also $\E[\Psi(\mathbf{s},\bxi_P,\sigma_P)\Psi(\mathbf{s},\bxi_P,\sigma_P)^T]$
is a block matrix with 
$\E[\Psi_{\bbeta}(\mathbf{s},\bxi_P,\sigma_P)\Psi_{\bbeta}(\mathbf{s},\bxi_P,\sigma_P)^T]$
and 
$\E[\Psi_{\bgamma}(\mathbf{s},\bxi_P,\sigma_P)\Psi_{\bgamma}(\mathbf{s},\bxi_P,\sigma_P)^T]$
on the main diagonal.
We conclude that the limiting variance~\eqref{eq:lim var bxi} of $\sqrt{n}(\bxi_n-\bxi_P)$
is a block matrix.
This proves that $\sqrt{n}(\bbeta_{1,n}-\bbeta_{1,P})$ and
$\sqrt{n}(\bgamma_n-\bgamma_P)$ are asymptotically independent.

Moreover, it also follows that $\sqrt{n}(\bbeta_{1,n}-\bbeta_{1,P})$ is asymptotically normal with
mean zero and variance
\begin{equation}
\label{eq:lim var beta}
\bD_{\bbeta}^{-1}\E[\Psi_{\bbeta}(\mathbf{s},\bxi_P,\sigma_P)\Psi_{\bbeta}(\mathbf{s},\bxi_P,\sigma_P)^T]\bD_{\bbeta}^{-1},
\end{equation}
and $\sqrt{n}(\bgamma_n-\bgamma_P)$ is asymptotically normal with mean zero and variance
\begin{equation}
\label{eq:lim var gamma}
\bD_{\bgamma}^{-1}\E[\Psi_{\bgamma}(\mathbf{s},\bxi_P,\sigma_P)\Psi_{\bgamma}(\mathbf{s},\bxi_P,\sigma_P)^T]\bD_{\bgamma}^{-1}.
\end{equation}
Because $\bV_P=\bSigma/|\bSigma|^{1/k}$,
we can write
\[
\Psi_{\bbeta}(\mathbf{s},\bxi_P,\sigma_P)
=
|\bSigma|^{1/k}
u_1\left(\frac{d^*}{\sigma_P}\right)
\bX^T\bSigma^{-1}(\by-\bX\bbeta^*),
\]
where $(d^*)^2=|\bSigma|^{1/k}(\by-\bX\bbeta^*)^T\bSigma^{-1}(\by-\bX\bbeta^*)$ and
$u_1(s)=\rho_1'(s)/s$.
We find that
\[
\begin{split}
&
\E\left[
\Psi_{\bbeta}(\mathbf{s},\bxi_P,\sigma_P)
\Psi_{\bbeta}(\mathbf{s},\bxi_P,\sigma_P)^T
\right]\\
&\qquad=
|\bSigma|^{2/k}
\E\left[
\bX^T
\E\left[
u_1\left(\frac{d^*}{\sigma_P}\right)^2
\bSigma^{-1}
(\by-\bX\bbeta^*)(\by-\bX\bbeta^*)^T
\bSigma^{-1}
\bigg|
\bX
\right]
\bX
\right].
\end{split}
\]
The inner expectation on the right hand side
is the conditional expectation of $\by\mid\bX$,
which has the same distribution as 
$\bSigma^{1/2}\bz+\bX\bbeta^*$, where $\bz$ has spherical density $f_{\mathbf{0},\bI_k}$.
Therefore, 
similar to the proof of Theorem~\ref{th:IF elliptical},
the inner conditional expectation can be written as
\[
\bSigma^{-1/2}
\E_{\mathbf{0},\bI_k}\left[u_1(c_\sigma\|\bz\|)^2\|\bz\|^2\right]
\E_{\mathbf{0},\bI_k}
\left[\bu\bu^T\right]
\bSigma^{-1/2}
=
\frac{\E_{\mathbf{0},\bI_k}\left[u_1(c_\sigma\|\bz\|)^2\|\bz\|^2\right]}{k}
\bSigma^{-1},
\]
where $c_\sigma=|\bSigma|^{1/(2k)}/\sigma_P$.
Together with~\eqref{def:derivative Lambda beta} 
and $u_1(s)=\rho_1'(s)/s$,
this implies that
the asymptotic variance~\eqref{eq:lim var beta} of $\sqrt{n}(\bbeta_{1,n}-\bbeta_1(P))$ is given by
\[
\frac{\sigma_P^2}{|\bSigma|^{1/k}}
\frac{\E_{\mathbf{0},\bI_k}\left[\rho_1'(c_\sigma\|\bz\|)^2\right]}{k\alpha_1^2}
\left(
\mathbb{E}\left[\mathbf{X}^T\bSigma^{-1}\mathbf{X}\right]
\right)^{-1},
\]
where $\alpha_1$ is defined in~\eqref{def:alpha1-gamma1}.

Next, consider the limiting variance~\eqref{eq:lim var gamma} of $\sqrt{n}(\bgamma_n-\bgamma(P))$. 
According to Lemma~\ref{lem:inverse}, we have
\begin{equation}
\label{eq:Dgammainv}
\bD_{\bgamma}^{-1}
=
a(\bE^T\bE)^{-1}
+
b
(\bE^T\bE)^{-1}
\bE^T\vc(\bI_k)\vc(\bI_k)^T\bE
(\bE^T\bE)^{-1},
\end{equation}
where $\bE=(\bSigma^{-1/2}\otimes\bSigma^{-1/2})\bL$, and 
$a=1/\omega_1$ and $b=\omega_2(\omega_1(\omega_1-k\omega_2))$,
where $\omega_1$ and $\omega_2$ are given in Lemma~\ref{lem:Lambda derivative}.
The rest of the proof is similar to the (second part of the) proof of Corollary~9.2 in Lopuha\"a \emph{et al}~\cite{lopuhaa-gares-ruizgazen2023}.
To make the connection with the proof of Corollary~9.2 in Lopuha\"a \emph{et al}~\cite{lopuhaa-gares-ruizgazen2023},
note that $\bV_P=\bSigma/|\bSigma|^{1/k}$, so that
\[
\Psi_{\bgamma}(\bs,\bxi_P,\sigma_P)
=
-
|\bSigma|^{2/k}
\bL^T
\left(
\bSigma^{-1}
\otimes
\bSigma^{-1}
\right)
\vc\left(
\Psi_\bV(\bs,\bxi_P,\sigma_P)
\right),
\]
and
\[
\begin{split}
&
\E\left[
\vc\left(\Psi_{\bgamma}(\bs,\bxi_P,\sigma_P)\right)
\vc\left(\Psi_{\bgamma}(\bs,\bxi_P,\sigma_P)\right)^T
\right]\\
&=
|\bSigma|^{4/k}
\bE^T
\E\left[
\vc\left(\bSigma^{-1/2}\Psi_{\bV}(\bs,\bxi_P,\sigma_P)\bSigma^{-1/2}\right)
\vc\left(\bSigma^{-1/2}\Psi_{\bV}(\bs,\bxi_P,\sigma_P)\bSigma^{-1/2}\right)^T
\right]
\bE,
\end{split}
\]
where $\Psi_\bV$ is given in~\eqref{def:PsiV}.
Furthermore, with $\bz=\bSigma^{-1/2}(\by-\bX\bbeta^*)$ and $\bu=\bz/\|\bz\|$,
we can write
\[
\begin{split}
\bSigma^{-1/2}\Psi_{\bV}(\bs,\bxi_P,\sigma_P)\bSigma^{-1/2}
&=
ku_1(c_\sigma\|\bz\|)\bz\bz^T
-
v_1(c_\sigma\|\bz\|)\frac{\sigma_P^2}{|\bSigma|^{1/k}}\bSigma\\
&=
ku_1(c_\sigma\|\bz\|)\|\bz\|^2\bu\bu^T
-
v_1(c_\sigma\|\bz\|)\frac{1}{c_\sigma^2}\bSigma\\
&=
\frac{1}{c_\sigma^2}
\Big\{
ku_1(c_\sigma\|\bz\|)(c_\sigma\|\bz\|)^2\bu\bu^T
-
v_1(c_\sigma\|\bz\|)\bSigma
\Big\}.
\end{split}
\]
Similar to the proof of Corollary~9.2 in Lopuha\"a \emph{et al}~\cite{lopuhaa-gares-ruizgazen2023}, we obtain
\[
\begin{split}
\E\left[
\Psi_{\bgamma}(\bs,\bxi_P,\sigma_P)
\Psi_{\bgamma}(\bs,\bxi_P,\sigma_P)^T
\right]
&=
\sigma_P^4|\bSigma|^{2/k}
\Big\{
2\delta_1\bE^T\bE
+
\delta_2
\bE^T\vc(\bI_k)\vc(\bI_k)^T\bE
\Big\},
\end{split}
\]
where 
\[
\begin{split}
\delta_1
&=
\frac{k\E_{\mathbf{0},\bI_k}\left[u_1(c_\sigma\|\bz\|)^2(c_\sigma\|\bz\|)^4\right]}{k+2}\\
\delta_2
&=
\frac{k\E_{\mathbf{0},\bI_k}\left[u_1(c_\sigma\|\bz\|)^2(c_\sigma\|\bz\|)^4\right]}{k+2}
-
2\E_{\mathbf{0},\bI_k}\left[u_1(c_\sigma\|\bz\|)v_1(c_\sigma\|\bz\|)(c_\sigma\|\bz\|)^2\right]\\
&\phantom{= \frac{k\E_{\mathbf{0},\bI_k}\left[u_1(c_\sigma\|\bz\|)^2(c_\sigma\|\bz\|)^4\right]}{k+2}}
+
\E_{\mathbf{0},\bI_k}\left[v_1(\|c_\sigma\bz\|)^2\right].
\end{split}
\]
Because $v_1(s)=u_1(s)s^2$, we find that
\[
\frac{k}{k+2}u_1(s)^2s^4
-
2u_1(s)v_1(s)s^2
+
v_1(s)^2
=
-\frac{2}{k+2}u_1(s)^2s^4.
\]
This means
\begin{equation}
\label{def:delta12}
\begin{split}
\delta_1
&=
\frac{k\E_{\mathbf{0},\bI_k}\left[u_1(c_\sigma\|\bz\|)^2(c_\sigma\|\bz\|)^4\right]}{k+2}\\
\delta_2
&=
-
\frac{2\E_{\mathbf{0},\bI_k}\left[u_1(c_\sigma\|\bz\|)^2(c_\sigma\|\bz\|)^4\right]}{k+2}
=
-\frac{2}{k}\delta_1.
\end{split}
\end{equation}
Together with~\eqref{eq:Dgammainv}, 
as in the proof of Corollary~9.2 in Lopuha\"a \emph{et al}~\cite{lopuhaa-gares-ruizgazen2023},
it follows that
\[
\begin{split}
&
\bD_{\bgamma}^{-1}
\E\left[
\Psi_{\bgamma}(\mathbf{s},\bxi_P,\sigma_P)\Psi_{\bgamma}(\mathbf{s},\bxi_P,\sigma_P)^T
\right]
\bD_{\bgamma}^{-1}\\
&=
2\sigma_1(\bE^T\bE)^{-1}
+
\sigma_2(\bE^T\bE)^{-1}\bE^T\vc(\bI_k)
\vc(\bI_k)^T
\bE(\bE^T\bE)^{-1},
\end{split}
\]
where 
\[
\sigma_1
=
\sigma_P^4|\bSigma|^{2/k}a^2\delta_1
=
\frac{\delta_1}{|\bSigma|^{2/k}\gamma_1^2},
\]
and 
\[
\sigma_2
=
\sigma_P^4|\bSigma|^{2/k}
\left\{
2b(2a+kb)\delta_1+(a+kb)^2\delta_2
\right\}
=
-2\sigma_P^4|\bSigma|^{2/k}a^2\delta_1/k
=
-2\sigma_1/k.
\]
We obtain
\begin{equation}
\label{def:sigma12}
\begin{split}
\sigma_1
&=
\frac{k\E_{\mathbf{0},\bI_k}\left[u_1(\|c_\sigma\bz\|)^2(c_\sigma\|\bz\|)^4\right]}{(k+2)\gamma_1^2|\bSigma|^{2/k}};\\
\sigma_2
&=
-\frac{2\sigma_1}{k}.
\end{split}
\end{equation}
Since $\bV$ has a linear structure, we have $\vc(\bSigma)=\bL\btheta^*$
and $\bE\btheta^*=\vc(\bI_k)$ and 
\[
(\bE^T\bE)^{-1}\bE^T\vc(\bI_k)=\btheta^*.
\]
We find that the limiting variance~\eqref{eq:lim var gamma} is equal to
$2\sigma_1
(\bL^T(\bSigma^{-1}\otimes\bSigma^{-1})\bL)^{-1}
+
\sigma_2
\btheta^*(\btheta^*)^T$.
This finishes the proof.
\end{proof}

\paragraph*{Proof of Corollary~\ref{cor:asymp norm theta1}}
\begin{proof}
Write $\bzeta_{0,P}=\bzeta_0(P)$, $\bxi_P=(\bbeta_{1,P},\bgamma_P)=(\bbeta_1(P),\bgamma(P))$,
$\bV_P=\bV(\bgamma_p)$, and $\sigma_P=\sigma(P)$.
Because $\btheta_{1,n}$ and~$\btheta_1(P)$ are solutions of~\eqref{eq:update theta} and~\eqref{def:theta1}, respectively,
similar to the reasoning in the proof of Corollary~\ref{cor:IF theta1},
we have that
\begin{equation}
\label{eq:decomp theta1}
\begin{split}
\btheta_{1,n}-\btheta_1(P)
&=
\sigma_n^2\bgamma_n-\sigma_P^2\bgamma_P
=
\sigma_n^2(\bgamma_n-\bgamma_P)
+
\bgamma_P(\sigma_n^2-\sigma_P^2)\\
&=
\sigma_n^2(\bgamma_n-\bgamma_P)
+
\bgamma_P(\sigma_n+\sigma_P)(\sigma_n-\sigma_P).
\end{split}
\end{equation}
We will apply Lemmas~\ref{lem:asymp relation xi sigma} and~\ref{lem:asymp relation sigma zeta}.
From Lemma~\ref{lem:change order int diff Lambda0},
we have that $\Lambda_0$ is continuously differentiable at $(\bzeta_{0,P},\sigma_P)$ with derivatives
$\bD_{0,\bzeta}=\mathbf{0}$ and
$D_{0,\sigma}<0$, according to Lemma~\ref{lem:D0zeta=0}.
Since $\bzeta_{0,n}-\bzeta_{0,P}=O_P(1/\sqrt{n})$, it follows from 
Lemma~\ref{lem:asymp relation sigma zeta} that
\[
\begin{split}
\sigma_n-\sigma(P)
=
-D_{0,\sigma}^{-1}
\int \Psi_0(\mathbf{s},\bzeta_0(P),\sigma(P))\,\dd (\mathbb{P}_n-P)(\mathbf{s})\\
+
o_P(\sigma_n-\sigma_P)
+
o_P(1/\sqrt{n}).
\end{split}
\]
Since $\rho_0$ is bounded, the central limit theorem applies 
to the first term on the right hand side of~\eqref{eq:expansion sigma2}
and we find that $\sigma_ n-\sigma_P=O_P(1/\sqrt{n})$.
Therefore,
\begin{equation}
\label{eq:expansion sigma2}
\begin{split}
\sigma_n-\sigma(P)
=
-D_{0,\sigma}^{-1}
\int \Psi_0(\mathbf{s},\bzeta_0(P),\sigma(P))\,\dd (\mathbb{P}_n-P)(\mathbf{s})
+
o_P(1/\sqrt{n}).
\end{split}
\end{equation}
From Lemma~\ref{lem:change order int diff Lambda},
we have that~$\Lambda$ is continuous differentiable at~$(\bxi_P,\sigma_P)$,
with derivative $\bD_{\bxi}=\partial\Lambda(\bxi_P,\sigma_P)/\partial\bxi$ given in
Lemma~\ref{lem:Lambda derivative}, 
which is non-singular according to Lemma~\ref{lem:inverse}.
Furthermore, according to Lemma~\ref{lem:Dsigma=0}, we have~$\bD_{\sigma}=\partial\Lambda(\bxi_P,\sigma_P)/\partial\sigma=\mathbf{0}$.
Since $\rho_1$ satisfies (R1)-(R4), such that $u_1(s)$ is of bounded variation,
and~$\bV$ satisfies~(V4), we may apply Lemma~\ref{lem:asymp relation xi sigma} and obtain
\[
\begin{split}
\bxi_n-\bxi_P
&=
-\bD_{\bxi}^{-1}
\int 
\Psi(\mathbf{s},\bxi_P,\sigma_P)\,\dd (\mathbb{P}_n-P)(\mathbf{s})
+
o_P(\|\bxi_n-\bxi(P)\|)
+
o_P(1/\sqrt{n}).
\end{split}
\]
As in the proof of Theorem~\ref{th:asymp norm elliptical},
we first obtain that $\bxi_n-\bxi_P=O_P(1/\sqrt{n})$ and then conclude that
\[
\bxi_n-\bxi_P
=
-\bD_{\bxi}^{-1}
\int \Psi(\mathbf{s},\bxi_P,\sigma_P)\,\dd (\mathbb{P}_n-P)(\mathbf{s})
+
o_P(1/\sqrt{n}).
\]
From the block structure of $\bD_{\bxi}$ established in Lemma~\ref{lem:Lambda derivative},
it then follows that 
\begin{equation}
\label{eq:asym expansion gamma}
\bgamma_n-\bgamma_P
=
-\bD_{\bgamma}^{-1}
\int \Psi_{\bgamma}(\mathbf{s},\bxi_P,\sigma_P)\,\dd (\mathbb{P}_n-P)(\mathbf{s})
+
o_P(1/\sqrt{n}).
\end{equation}
In particular, this implies that $\bgamma_n-\bgamma_P=O_P(1/\sqrt{n})$, so that
from~\eqref{eq:decomp theta1} we obtain
\[
\sqrt{n}(\btheta_{1,n}-\btheta_1(P))
=
\sigma^2_P
\sqrt{n}(\bgamma_n-\bgamma_P)
+
2\sigma_P\bgamma_P
\sqrt{n}(\sigma_n-\sigma_P)+o_P(1).
\]
From expansions~\eqref{eq:expansion sigma2} and~\eqref{eq:asym expansion gamma}, 
we conclude that if we define
\[
\Psi_{\btheta}(\bs,\bxi,\sigma)
=
\sigma^2_P
\bD_{\bgamma}^{-1}
\Psi_{\bgamma}(\bs,\bxi_P,\sigma_P)
+
2\sigma_P\bgamma_P
D_{0,\sigma}^{-1}
\Psi_{0}(\bs,\bxi_P,\sigma_P),
\]
then $\sqrt{n}(\btheta_{1,n}-\btheta_1(P))$ is asymptotically normal with mean zero and variance
\[
\E\left[
\Psi_{\btheta}(\bs,\bxi_P,\sigma_P)
\Psi_{\btheta}(\bs,\bxi_P,\sigma_P)^T
\right].
\]
As in the proof of Theorem~\ref{th:asymp norm elliptical},
\[
\Psi_{\bgamma}(\bs,\bxi_P,\sigma_P)
=
-
|\bSigma|^{2/k}
\bL^T
\left(
\bSigma^{-1}
\otimes
\bSigma^{-1}
\right)
\vc\left(
\Psi_\bV(\bs,\bxi_P,\sigma_P)
\right),
\]
with
\[
\begin{split}
\Psi_\bV(\bs,\bxi_P,\sigma_P)
&=
k
u_1\left(\frac{d^*}{\sigma_P}\right)
(\by-\bX\bbeta^*)(\by-\bX\bbeta^*)^T
-
v_1\left(\frac{d^*}{\sigma_P}\right)
\frac{\sigma_P^2}{|\bSigma|^{1/k}}
\bSigma.
\end{split}
\]
where $(d^*)^2=|\bSigma|^{1/k}(\by-\bX\bbeta^*)^T\bSigma^{-1}(\by-\bX\bbeta^*)$,
and from the proof of Lemma~\ref{lem:asym norm sigma}
\[
\Psi_{0}(\bs,\bxi_P,\sigma_P)
=
\rho_0
\left(\frac{d_{\Gamma}^*}{\sigma_P}\right)
-
b_0,
\]
where $d_{\Gamma}^*=d^*$.
This means that
\[
\begin{split}
&
\E[\Psi_\bV(\bs,\bxi_P,\sigma_P)\Psi_{0}(\bs,\bxi_P,\sigma_P)]\\
&=
\E\left[
\E\left[
\left(
k
u_1\left(\frac{d^*}{\sigma_P}\right)
\be^*(\be^*)^T
-
\frac{\sigma_P^2}{|\bSigma|^{1/k}}
v_1\left(\frac{d^*}{\sigma_P}\right)
\bSigma
\right)
\left(\rho_0
\left(\frac{d^*}{\sigma_P}\right)
-
b_0\right)
\bigg|
\bX
\right]
\right],
\end{split}
\]
where $\be^*=\by-\bX\bbeta^*$.
The inner expectation is the conditional expectation of $\by\mid\bX$,
which has the same distribution as 
$\bSigma^{1/2}\bz+\bX\bbeta^*$, where $\bz$ has spherical density $f_{\mathbf{0},\bI_k}$.
Therefore, if we denote $c_\sigma=|\bSigma|^{1/(2k)}/\sigma_P$, then the inner expectation can be written as
\[
\begin{split}
&
\mathbb{E}_{0,\mathbf{I}_k}
\left[
\left(
k
u_1\left(c_\sigma|\bz\|\right)
\bSigma^{1/2}\bz\bz^T\bSigma^{1/2}
-
v_1\left(\|\bz\|\right)
\frac{1}{c_\sigma^2}\bSigma
\right)
\left(\rho_0
\left(c_\sigma\|\bz\|\right)
-
b_0\right)
\right]\\
&=
\frac{1}{c_\sigma^2}
\mathbb{E}_{0,\mathbf{I}_k}
\left[
k
u_1\left(c_\sigma|\bz\|\right)(c_\sigma\|\bz\|)^2
\left(\rho_0
\left(c_\sigma\|\bz\|\right)
-
b_0\right)
\bSigma^{1/2}\bu\bu^T\bSigma^{1/2}
\right]\\
&\qquad
-
\frac{1}{c_\sigma^2}
\mathbb{E}_{0,\mathbf{I}_k}
\left[
v_1\left(c_\sigma\|\bz\|\right)
\left(
\rho_0\left(c_\sigma\|\bz\|\right)
-
b_0\right)
\right]
\bSigma\\
&=
\frac{k}{c_\sigma^2}
\mathbb{E}_{0,\mathbf{I}_k}
\left[
v_1\left(c_\sigma\|\bz\|\right)
\left(
\rho_0\left(c_\sigma\|\bz\|\right)
-
b_0\right)
\right]
\bSigma^{1/2}
\mathbb{E}_{0,\mathbf{I}_k}
\left[
\bu\bu^T
\right]
\bSigma^{1/2}\\
&\qquad
-
\frac{1}{c_\sigma^2}
\mathbb{E}_{0,\mathbf{I}_k}
\left[
v_1\left(c_\sigma\|\bz\|\right)
\left(
\rho_0\left(c_\sigma\|\bz\|\right)
-
b_0\right)
\right]
\bSigma\\
&=
\frac{k}{c_\sigma^2}
\mathbb{E}_{0,\mathbf{I}_k}
\left[
v_1\left(c_\sigma\|\bz\|\right)
\left(
\rho_0\left(c_\sigma\|\bz\|\right)
-
b_0\right)
\right]
\frac{1}{k}\bSigma\\
&\qquad
-
\frac{1}{c_\sigma^2}
\mathbb{E}_{0,\mathbf{I}_k}
\left[
v_1\left(c_\sigma\|\bz\|\right)
\left(
\rho_0\left(c_\sigma\|\bz\|\right)
-
b_0\right)
\right]
\bSigma
=
\mathbf{0}.
\end{split}
\]
We find that
\[
\begin{split}
\E\left[
\Psi_{\btheta}(\bs,\bxi,\sigma)
\Psi_{\btheta}(\bs,\bxi,\sigma)^T
\right]
=
\sigma^4_P
\bD_{\bgamma}^{-1}
\E\left[
\Psi_{\bgamma}(\bs,\bxi_P,\sigma_P)\Psi_{\bgamma}(\bs,\bxi_P,\sigma_P)^T
\right]
\bD_{\bgamma}^{-1}\\
+
4\sigma^2_P
D_{0,\sigma}^{-1}
\E\left[
\Psi_{0}(\bs,\bxi_P,\sigma_P)^2
\right]
D_{0,\sigma}^{-1}
\bgamma_P\bgamma_P^T,
\end{split}
\]
which is a linear combination
of the asymptotic variances of 
$\sqrt{n}(\bgamma_n-\bgamma_P)$ and~$\sqrt{n}(\sigma_n-\sigma_P)$:
\[
\begin{split}
&
\E\left[
\Psi_{\btheta}(\bs,\bxi,\sigma)
\Psi_{\btheta}(\bs,\bxi,\sigma)^T
\right]\\
&\qquad=
\sigma^4_P
\text{AVAR}\left(\sqrt{n}(\bgamma_n-\bgamma_P)\right)
+
4\sigma^2_P
\text{AVAR}\left(\sqrt{n}(\sigma_n-\sigma_P)\right)
\bgamma_P\bgamma_P^T.
\end{split}
\]
Since $\bV$ has a linear structure, similar to the proof of Corollary~\ref{cor:IF theta1},
from~\eqref{eq:relation gamma theta*} we have $\bgamma_P=\btheta^*/|\bSigma|^{1/k}$.
Then, from Theorem~\ref{th:asymp norm elliptical} and Lemma~\ref{lem:asym norm sigma}, 
we find that the right hand side is equal to
\[
\begin{split}
&
\frac{2\sigma_1}{c_\sigma^2}
\left\{
\Big(\bL^T\left(\bSigma^{-1}\otimes\bSigma^{-1}\right)\bL\Big)^{-1}
-
\frac{1}{k}
\btheta^*(\btheta^*)^T
\right\}
+
\frac{4}{c_\sigma^2}
\frac{\E\left[(\rho_0(c_\sigma\|\bz\|)-b_0)^2\right]}{\big(\E_{\mathbf{0},\bI_k}[\rho_0'(c_\sigma\|\bz\|)c_\sigma\|\bz\|]\big)^2}
\btheta^*(\btheta^*)^T.
\end{split}
\]
This finishes the proof.
\end{proof}

\subsection{Limiting distributions of covariance estimators}
We provide some details about the limiting variances of the covariance MM-estimators
in Examples~\ref{ex:LME model}, \ref{ex:multivariate linear regression}, and~\ref{ex:multivariate location-scatter}
for the situation where the distribution $P$ satisfies~(E).
\setcounter{example}{0}
\begin{example}[Linear Mixed Effects model]
\label{ex:asymp distr LME}
For linear mixed effects models, 
nothing seems to be available about the limiting distribution of covariance MM-estimators.
For model~\eqref{def:linear mixed effects model Copt}, the limiting distribution of the 
variance component MM-estimator now follows from Corollary~\ref{cor:asymp norm theta1} and
the limiting distributions of the covariance MM-estimator
and the corresponding shape component follow from~\eqref{eq:asymp var theta1}
and~\eqref{eq:asymp var shape}, respectively.
This implies that~$\sqrt{n}(\btheta_{1,n}-\btheta_1(P))$ is asymptotically normal with mean zero and variance
\[
\frac{2\sigma_1}{c_\sigma^2}
\bQ^{-1}
+
\left(
-\frac{2\sigma_1}{kc_\sigma^2}
+
\sigma_3
\right)
\btheta^*(\btheta^*)^T,
\]
where $\sigma_1$ and $\sigma_3$ are defined in~\eqref{def:sigma1} and~\eqref{def:sigma3},
and where $\bQ$ is the matrix with entries given in~\eqref{eq:entries Q}. 
From~\eqref{eq:asymp var theta1} we find that 
$\sqrt{n}(\vc(\bV(\btheta_{1,n}))-\vc(\bSigma))$ is asymptotically normal with 
mean zero and variance
\[
\frac{2\sigma_1}{c_\sigma^2}
\bL\bQ^{-1}\bL^T
+
\left(
-\frac{2\sigma_1}{kc_\sigma^2}
+
\sigma_3
\right)
\vc(\bSigma)
\vc(\bSigma)^T,
\]
where $\bL$ is given by~\eqref{eq:L for LME}.
For the shape component $\bGamma(\btheta)=\bV(\btheta)/|\bV(\btheta)|^{1/k}$,
from~\eqref{eq:asymp var shape} we obtain that 
$\sqrt{n}(\vc(\bGamma(\btheta_{1,n}))-\vc(\bGamma(\btheta_1(P))))$
is asymptotically normal with mean zero and variance
\[
\frac{2\sigma_1}{c_\sigma^2|\bSigma|^{2/k}}
\left\{
\bL\bQ^{-1}\bL^T
-
\frac{1}{k}
\vc(\bSigma)
\vc(\bSigma)^T
\right\}.
\]
When $c_\sigma=1$, the limiting distribution
of the covariance shape MM-estimator in the linear mixed effects model~\eqref{def:linear mixed effects model Copt} 
coincides with that of the shape component corresponding to the covariance S-estimator defined with $\rho_1$,
see Lopuha\"a \emph{et al}~\cite{lopuhaa-gares-ruizgazen2023},
and similarly for the direction component of the variance component MM-estimator.
\end{example}

\begin{example}[Multivariate Linear Regression]
For the multivariate linear regression model~\eqref{def:multivariate linear regression model}, 
Kudraszow and Maronna~\cite{kudraszow-maronna2011} do not consider the
limiting distribution of the covariance MM-estimator.
In this model, the matrix $\bL$ is equal to the duplication matrix~$\mathcal{D}_k$,
which satisfies~\eqref{eq:prop duplication}.
The limiting distribution of the covariance MM-estimator
and the corresponding shape component follow from~\eqref{eq:asymp var theta1}
and~\eqref{eq:asymp var shape}, respectively.
By using~\eqref{eq:prop duplication}, the limiting variance of the covariance MM-estimator becomes
\begin{equation}
\label{eq:asymp var cov MM mult regr}
\frac{\sigma_1}{c_\sigma^2}
\left(
\bI_{k^2}+\bK_{k,k}
\right)(\bSigma\otimes\bSigma)
+
\left(
-\frac{2\sigma_1}{kc_\sigma^2}
+
\sigma_3
\right)
\vc(\bSigma)
\vc(\bSigma)^T,
\end{equation}
whereas the covariance shape estimator has limiting variance
\begin{equation}
\label{eq:asymp var shape MM mult regr}
\frac{\sigma_1}{c_\sigma^2|\bSigma|^{2/k}}
\left\{
\left(
\bI_{k^2}+\bK_{k,k}
\right)(\bSigma\otimes\bSigma)
-
\frac{2}{k}
\vc(\bSigma)
\vc(\bSigma)^T
\right\},
\end{equation}
where $\sigma_1$ and $\sigma_3$ are defined in~\eqref{def:sigma1} and~\eqref{def:sigma3}.
When $c_\sigma=1$, the limit behavior of the covariance shape MM-estimator 
coincides with that of the shape component of the covariance S-estimator defined with $\rho_1$
in the multivariate linear regression model.
\end{example}

\begin{example}[Multivariate Location and Scatter]
For the multivariate location-scatter model, we also have $\bL=\mathcal{D}_k$.
Since this model is a special case of the multivariate linear regression model~\eqref{def:multivariate linear regression model}
by taking $\bx_i=1$ and $\bB^T=\bmu$,
the limiting distributions of the covariance MM-estimator and the corresponding shape component 
are the same as that of their counterparts in the multivariate linear regression model
and the limiting variances have the same expressions as~\eqref{eq:asymp var cov MM mult regr}
and~\eqref{eq:asymp var shape MM mult regr}.
When $c_\sigma=1$, the behavior of the covariance shape MM-estimator 
coincides with that of the covariance shape S-estimator defined with $\rho_1$.
This was already observed by Salibi\'an-Barrera \emph{et al}~\cite{SalibianBarrera-VanAelst-Willems2006},
whose formula~(9) matches with the expression in~\eqref{eq:asymp var shape MM mult regr} with $c_\sigma=1$.
Finally, also here there is a connection with the CM-estimators considered in 
Kent and Tyler~\cite{kent&tyler1996},
whose limiting distribution depends on a parameter~$\lambda_0$.
By using~\eqref{eq:prop duplication}, it can be seen that the limiting distribution 
of $\sqrt{n}(\vc(\bV(\btheta_{1,n}))-\vc(\bSigma))$ is similar to that of 
the covariance CM-estimator for the particular case that $\lambda_0=\lambda_L$ (see Kent and Tyler~\cite{kent&tyler1996} for details),
and that they both coincide when $c_\sigma=1/\sqrt{\lambda_0}$ and $\rho_0=\rho_1$.
\end{example}

\end{document}